\documentclass[11pt, reqno]{article}
\usepackage[utf8]{inputenc}
\usepackage{amsmath}

\usepackage{amsfonts, mathrsfs}
\usepackage{mathtools}
\usepackage{amsthm}
\usepackage{amssymb}
\usepackage{caption}
\usepackage{subcaption}
\usepackage{graphicx}
\usepackage[affil-it]{authblk}
\usepackage[sc]{mathpazo}
\usepackage{dsfont}
\usepackage{enumerate}
\usepackage{wrapfig}
\usepackage[linesnumbered,ruled,vlined]{algorithm2e}

\usepackage[dvipsnames]{xcolor}
\newtheorem{theorem}{Theorem}
\newtheorem{lemma}[theorem]{Lemma}
\newtheorem{prop}[theorem]{Proposition}
\newtheorem{corollary}[theorem]{Corollary}
\newtheorem{assumption}[theorem]{Assumption}
\newtheorem{remark}[theorem]{Remark}

\newtheorem*{claim*}{Claim}
\newtheorem{claim}[theorem]{Claim}
\newtheorem{defn}[theorem]{Definition}

\numberwithin{equation}{section}
\numberwithin{theorem}{section}

\let\plainqed\qedsymbol
\newcommand{\claimqed}{$\lrcorner$}
\newenvironment{claimproof}{\begin{proof}\renewcommand{\qedsymbol}{\claimqed}}{\end{proof}\renewcommand{\qedsymbol}{\plainqed}}

\usepackage[letterpaper, margin=1in]{geometry}
\usepackage{hyperref}
\usepackage[sort&compress,numbers]{natbib}

\usepackage{array}   
\newcolumntype{L}[1]{>{\raggedright\let\newline\\\arraybackslash\hspace{0pt}}m{#1}}
\newcolumntype{C}[1]{>{\centering\let\newline\\\arraybackslash\hspace{0pt}}m{#1}}
\newcolumntype{R}[1]{>{\raggedleft\let\newline\\\arraybackslash\hspace{0pt}}m{#1}}

\newtheorem{innercustomgeneric}{\customgenericname}
\providecommand{\customgenericname}{}
\newcommand{\newcustomtheorem}[2]{%
  \newenvironment{#1}[1]
  {%
   \renewcommand\customgenericname{#2}%
   \renewcommand\theinnercustomgeneric{##1}%
   \innercustomgeneric
  }
  {\endinnercustomgeneric}
}

\newcustomtheorem{customthm}{Theorem}
\newcustomtheorem{customlemma}{Lemma}
\newcustomtheorem{customproposition}{Proposition}

\newcommand{\dto}{\ensuremath{\ \xrightarrow{d} \ }}  
\newcommand{\R}{\mathbb{R}}                 
\newcommand{\N}{\mathbb{N}}                 
\newcommand{\PP}{\mathbb{P}}  				

\newcommand{\sel}{\texttt{select}}

\DeclareMathOperator*{\argmin}{arg\,min}
\newcommand{\dom}{\texttt{dom}}

\newcommand\blfootnote[1]{%
  \begingroup
  \renewcommand\thefootnote{}\footnote{#1}%
  \addtocounter{footnote}{-1}%
  \endgroup
}  

\title{Optimal Rate-Matrix Pruning For Large-Scale Heterogeneous Systems}

\begin{document}
\author{
  Zhisheng Zhao$^1$, Debankur Mukherjee$^2$
}
\renewcommand\Authands{, }
\maketitle
\begin{abstract}
We present an analysis of large-scale load balancing systems, where the processing time distribution of tasks depends on both the task and server types. Our study focuses on the asymptotic regime, where the number of servers and task types tend to infinity in proportion.
In heterogeneous environments, commonly used load balancing policies such as Join Fastest Idle Queue and Join Fastest Shortest Queue exhibit poor performance and even shrink the stability region. Interestingly, prior to this work, finding a scalable policy with a provable performance guarantee in this setup remained an open question.

To address this gap, we propose and analyze two asymptotically delay-optimal dynamic load balancing policies. The first policy efficiently reserves the processing capacity of each server for ``good" tasks and routes tasks using the vanilla Join Idle Queue policy. The second policy, called the speed-priority policy, significantly increases the likelihood of assigning tasks to the respective ``good" servers capable of processing them at high speeds.
By leveraging a framework inspired by the graphon literature and employing the mean-field method and stochastic coupling arguments, we demonstrate that both policies achieve asymptotic zero queuing. Specifically, as the system scales, the probability of a typical task being assigned to an idle server approaches 1.

   \blfootnote{$^1$Georgia Institute of Technology, \emph{Email:} \href{mailto:zhisheng@gatech.edu}{zhisheng@gatech.edu}}
\blfootnote{$^2$Georgia Institute of Technology, \emph{Email:}  \href{mailto:debankur.mukherjee@isye.gatech.edu}{debankur.mukherjee@isye.gatech.edu}}
\blfootnote{\emph{Keywords and phrases}. heterogeneous load balancing, data locality, Join-the-Idle Queue, fluid method, stochastic coupling }
\blfootnote{\emph{Acknowledgements.} The work was partially supported by the NSF grant CIF-2113027.}
\end{abstract}

\section{Introduction}
\label{sec:intro}


\noindent
\textbf{Background and motivation.}
Advanced cloud computing platforms, such as AWS, Azure, and Google Cloud, handle millions of requests per second. Efficiently assigning tasks across servers using a load balancing algorithm is critical for the seamless functioning of these systems. 
This has inspired a huge body of foundational works on large-scale load balancing algorithms in the last decade~\cite{Mitzenmacher96,VDK96,Bramson11,LXKGLG11,EG18,GW19,WZS20,RM22}. See~\cite{BBLM21} for a comprehensive recent survey.
Despite this enormous progress, it is worthwhile to point out that the existing works have predominately focused on homogeneous models, where a task is processed at the same rate at all servers.
In contrast, real-world cloud computing platforms receive requests containing multiple classes of tasks with varying characteristics, such as accessing websites, training machine learning models, or backing up data. 
Additionally, with the expansion of these platforms, servers can be of different types (multi-skilled), as evident even from AWS's website, which lists at least 9 server types with varying memory and bandwidth. Moreover, due to the storage limit (a.k.a.~\textit{data locality}), a server can only have the required resource files to execute only a subset of tasks. 
In such heterogeneous environments, commonly used ultra-low-latency routing policies, such as Join Shortest Queue (JSQ), Join Idle Queue (JIQ), Join Fastest Shortest Queue (JFSQ), and the Join Fastest Idle Queue (JFIQ) are known to perform poorly (as shown in the numerical results in Section~\ref{sec:numerical results}). 
One reason is that these algorithms optimize the task assignment myopically which depends only on the current state.
This often results in the assignment of tasks to servers that cannot process it at a relatively high speed, leading to inefficient server utilization.
Thus, to better model large-scale cloud computing platforms, it is crucial to understand such heterogeneous parallel-server systems, where the time to process a task in a server depends on \textit{both the type of the task and that of the server}, and develop efficient load balancing algorithms for them. 
The above motives the current work.
As a side note, even though our primary motivation is data center networks, it is worthwhile to mention that similar heterogeneity exists in many other service systems as well. For example, in hospitals patients arriving at the emergency room may have different types of emergencies and multiple medical staff available, or in manufacturing systems where different types of machines and workers are present for different operations, such as assembly, packaging, and painting.\\

\noindent
\textbf{Key challenges.} 
There are two fundamental reasons why fully heterogeneous systems have resisted any foundational progress in large-scale performance analysis:
\textbf{First}, if all the server processing speeds are distinct, then servers become non-exchangeable.
This takes us beyond the range of classical mean-field techniques. In fact, before attempting any large-scale analysis, one needs to consistently define a sequence of systems with increasing number of servers. Even this is unclear when service rates are fully heterogeneous.
As we will see, in this paper, we mitigate this issue using an approach inspired from the graphon literature~\cite{ll12}.
\textbf{Second}, when different tasks require different processing times even at the same server,
(under first-come-first-served discipline) a Markovian state descriptor must include the order of the tasks in each queues. Consequently, the system lacks any aggregate state descriptor.
Also, as the system scales, the state space thus explodes exponentially. These lead to an intractable performance analysis.

In addition to the analytical hurdles, the process of designing scalable algorithms for such a heterogeneous system presents significant obstacles.
In the seminal work, Stolyar~\cite{AS05} discussed the two-fold heterogeneous systems and proposed the \textsc{MinDrift} policy, which can be understood as the Gc$\mu$-rule~\cite[Section~4]{MS04} in the (output-queued) load balancing setup. 
Here, it was shown that the \textsc{MinDrift} policy defined in~\ref{sec:MinDrift} asymptotically minimizes the server workload in the conventional heavy traffic regime when the number of servers is fixed and the arrival rate approaches the boundary of the suitably defined capacity region. 
However, large-scale performance analysis of this policy was not performed. 
Furthermore, implementing the \textsc{MinDrift} policy at scale requires the dispatcher to know the total expected workload and service rate of every server for the arriving task, which results in a prohibitive communication burden. 
To the best of our knowledge, for such a heterogeneous setting, our study is the first to design a scalable policy with a provable performance guarantee in the many-server asymptotic regime.\\

\noindent
\textbf{Our contributions.}
The main contributions of this paper are two-fold:\\

\noindent
\textbf{(a)} We provide a framework to define a consistent sequence of heterogeneous systems with increasing size, which makes it amenable for large-scale analysis. 
The key challenge is to capture the heterogeneity in service rates for all pairs of task types and servers. 
We capture this `matrix' of service rates by introducing a notion inspired by graphon theory, which we call $f$-sequence (Definition~\ref{def:f-sequence}). 
From a high level, we map each task type and server into $[0,1)$ through mapping functions $\phi_1$ and $\phi_2$, respectively. 
Both this functions remain fixed for various values of $N$.
Then, the service rate matrix is described using a function $f:[0,1)^2\rightarrow\R_+$.
Specifically, if $\mu^N_{i,j}$ is the the service rate for processing task-type $i$ at server $j$ in $G^N$, then $\mu^N_{i,j}=f(\phi_1(i),\phi_2(j))$, for all $N$.
As we will see that this provides a generalizable framework for large-scale analysis.\\

\noindent
(b)
We design two low-complexity load balancing policies for the fully heterogeneous systems, both of which achieve the asymptotic zero-queueing property, a.k.a.~\emph{ultra-low latency} (where the probability of a task waiting in the queue approaches zero as the system size increases):
\begin{enumerate}[(I)]
    \item \textbf{Intelligent Capacity Reservation and Dispatching (ICRD) approach:} In this approach, we do a `preprocessing step' before the system goes live where the system reserves the capacity of each server for a subset of `good' tasks, which are tasks that can be processed efficiently in that server. 
    More specifically, we carefully prune the service rate matrix (by simply replacing some entries with 0).
    Then, dispatchers assign tasks under the vanilla JIQ policy. 
    \item \textbf{Speed-Priority Dispatching (SPD) approach:} In this approach, each dispatcher clusters its compatible servers into several groups based on their respective service capabilities of the corresponding task type. 
    For each arriving task, first a target group is chosen randomly, which gives more weight to the group that can process that task type at a higher service rate. 
    Next, the task is assigned to an idle queue in the target group. 
\end{enumerate}

Both the above approaches can be implemented in a token-based fashion, thus inheriting all the scalability properties similar to the JIQ policy.
In Theorem~\ref{thm:fluid-limit-2} and Theorem~\ref{thm:fluid-limit-1}, we demonstrate that both approaches achieve asymptotic zero-queueing, whereby the probability of assigning a new task to an idle server approaches 1 as the system scales.
In addition, fairly elementary tools are used to establish these results.

Due to many fundamental challenges mentioned earlier, the direct performance analysis of these systems with a general function $f$ is intractable. 
Therefore, we first investigate a special case where $f$ is stepwise, which serves as an intermediate step towards the general case. 
In the stepwise case, the ICRD approach reduces the system to a union of \textit{dispatcher-independent} systems where the service rates only depend on server-types.
Similarly, under the SPD approach, we show that the system behaves as a union of \textit{server-independent} systems, where the service rates only depend on the task-types.
Decoupling the service rates from either of the two dependencies makes it analytically tractable.
In the end, for the general $f$-sequence, under certain regular assumption, we can always find a stepwise $f'$-sequence whose performance can lower bound that of general one. 
Thus, through stochastic coupling arguments, in Theorem~\ref{thm:general-JIQ}and Theorem~\ref{thm:f-zero-queue}, we show that the zero-queueing property of ICRD and SPD still holds for the general case.\\

\noindent
\textbf{Related works.}
The vast literature on load balancing systems can be broadly classified into three categories based on their modeling characteristics: (i) homogeneous models, (ii) heterogeneous-server models, and (iii) two-fold heterogeneous models. 
Below we discuss a few representative works related to the current paper in each of the above three categories.

\paragraph{Homogeneous models} The canonical homogeneous model assumes a system with one dispatcher and multiple parallel identical servers. Beyond its optimality, the behavior of the JSQ has been comprehensively analyzed in various heavy traffic regimes, including the convergence of the occupancy process~\citep{EG18,GW19,zhao21}, the analysis of steady states~\citep{BM19a,BM19b,Braverman18,HM20,HLM22,LY19,LY21}, and rates of convergence~\citep{Braverman22}. 
However, the JSQ policy requires instantaneous information from all servers, leading to poor scalability. This has motivated the consideration of JSQ($d$) policies (a.k.a.~power-of-$d$-choices)~\cite{Mitzenmacher96,VDK96}.
In the presence of data locality, the JSQ($d$) policy is considered in~\cite{RM22,RM23}.
On a different thread of works, the JIQ policy is analyzed due to its asymptotic optimality properties and scalable, token-based implementation~\cite{AS15,LXKGLG11, MBLW16-1}.

\paragraph{Heterogeneous-server models}
When servers are heterogeneous and all tasks are identical,~\cite{Stolyar17} establishes the zero-queueing property the JIQ policy and
\cite{BM22,HLM21,MM16} considers the throughput optimality and performance improvement of the JSQ($d$) policy in various settings, like FCFS and processor-sharing. 
In the presence of data locality, in this setup, 
Weng et al.~\cite{WZS20} 
showed that the JFSQ/JFIQ policies achieve asymptotic optimality for minimizing mean steady-state waiting time when the bipartite graph is sufficiently well connected and Zhao et al.~\cite{ZhaoMW22} analyzed the JSQ($d$) policy. 
Furthermore, Allmeier and Gast~\cite{AG22} studied the application of (refined) mean-field approximations for heterogeneous systems, which uses an ODE to approximate the evolution of each server.
Gardner and Righter~\cite{GR20} investigated a system with arbitrary task-server compatibilities and demonstrated that the stationary distributions under a random assignment policy and several redundancy policies exhibit product-form structures. However, they also noted that obtaining bounds on mean response time solely from the product-form results was non-trivial.

\paragraph{Two-fold heterogeneous models}
The literature on service systems where service times depend on both task and server types is significantly scarce. 
The initial works on such system with the input-queued setup dates back to the work by Harrison~\cite{Harrison98}, who considered a system with two input streams and two servers and proved the asymptotic optimality of a constructed discrete-review control policy in the heavy traffic regime.
Subsequent works include~\cite{HL99, MS04}.
For output-queued systems, Foss and Chernova~\cite{FC98} proved the stability condition for the model with two types of servers via fluid analysis. 
For the general case, Stolyar~\cite{AS05} established the asymptotic optimality of the \textsc{MinDrift} policy and briefly mentioned that the stability of the system can be proved via the Lyapunov method. Dai and Tezcan~\cite{DT11} provided a general framework for establishing state space collapse results in the heavy traffic regime. 
The behavior in the heavy traffic regime investigated in~\cite{DT11,AS05} is typically qualitatively different from the subcritical regime considered in our work. \\

\noindent
\textbf{Notation.} 
Let $\N$ be the set of natural numbers and $\N_0=\N\cup \{0\}$. Let $\R_+$ be the set of nonnegative real numbers. 
For any $N\in\N$, define $[N]=\{1,2,\cdots,N\}$.
For a polish space $\mathcal{S}$, the space of right
continuous functions with left limits from $[0,\infty)$ to $\mathcal{S}$ is denoted as $\mathbb{D}([0,\infty),\mathcal{S})$, endowed with the Skorokhod J$_1$ topology. WLOG is the acronym of `without loss of generality'.

\section{Model Description}\label{ssec:model-contribution}
We will denote the heterogeneous system by $G^N=(\mathcal{W}^N,\mathcal{V}^N,\boldsymbol\lambda^N,\mathcal{U}^N)$. Here, $\mathcal{W}^N=\{1,...,W(N)\}$ represents the set of dispatchers, where each dispatcher only handles assignments of one type of tasks. Hence, the terms `task-type' and `dispatcher' will be used interchangeably. 
The set of servers is denoted by $\mathcal{V}^N=\{1,..., N\}$, where each server $j\in\mathcal{V}^N$ has a dedicated queue with infinite buffer capacity, and tasks are scheduled at each queue using the First-Come-First-Served (FCFS) policy. 
The arrival process of tasks at dispatcher $i\in\mathcal{W}^N$ is a Poisson process with rate $\lambda^N_i$, independently of other processes, and $\boldsymbol\lambda^N$ denotes the arrival rate vector $(\lambda^N_1,...,\lambda^N_{W(N)})$.
The matrix of service rates is represented by $\mathcal{U}^N=(\mu^N_{i,j},i\in\mathcal{W}^N,j\in\mathcal{V}^N)\in \R_+^{W(N)\times N}$, where the service time of a type $i$ task at server $j$ is exponentially distributed with mean $1/\mu^N_{i,j}$, if $\mu^N_{i,j}>0$. Otherwise (i.e., when $\mu^N_{i,j}=0$), by convention, server $j$ cannot process type $i$ tasks. A server $j\in\mathcal{V}^N$ is considered `\textit{compatible}' for type $i\in\mathcal{W}^N$ tasks if $\mu^N_{i,j}>0$. It is assumed that tasks arriving at a dispatcher must be instantaneously and irrevocably assigned to one of the compatible servers.
A schematic picture of the model is given in Figure~\ref{fig:general-hetero}, where edges corresponding to incompatible $(i,j)$ pairs are not shown.
\begin{figure}[h]
    \centering
\includegraphics[width=0.7\textwidth]{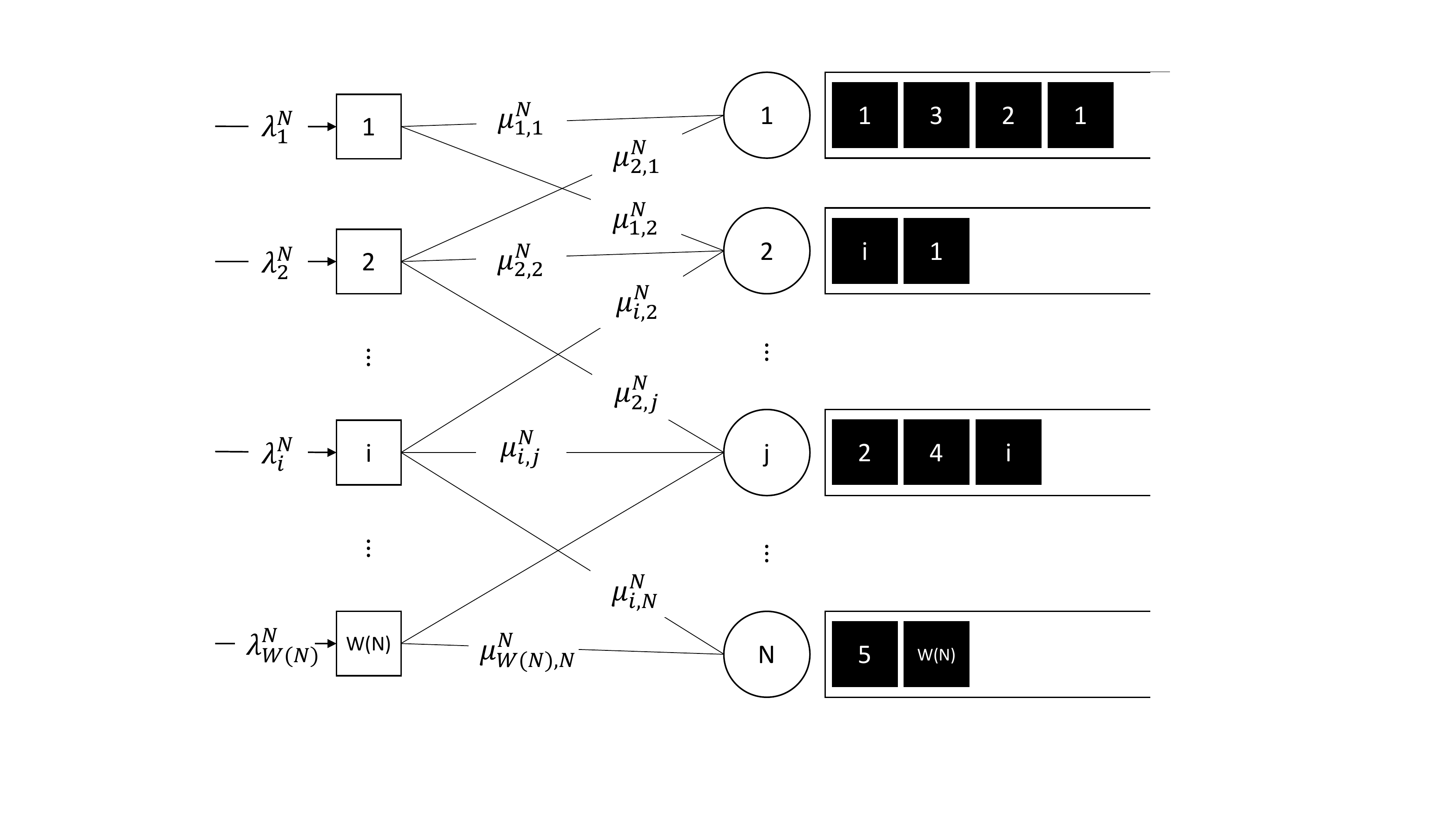}
    \caption{Heterogeneous Load Balancing System $G^N$}
    \label{fig:general-hetero}
\end{figure}

\subsection{Consistent model structure: $f$-sequence}
Consider a sequence $\{G^N=(\mathcal{W}^N,\mathcal{V}^N,\boldsymbol\lambda^N,\mathcal{U}^N)\}_{N\in\N}$ of systems. 
For consistency, we will define server and task type indices in the sequence in a nested way, that is, for all $N\in\N$, $\mathcal{W}^N\subseteq\mathcal{W}^{N+1}$, $\mathcal{V}^N\subseteq\mathcal{V}^{N+1}$, and $\mu^N_{i,j}=\mu^{N+1}_{i,j}$, $\forall (i,j)\in\mathcal{W}^N\times\mathcal{V}^N$. 
Inspired by the notion of graphon~\cite[Chapter~7]{ll12}, we define two membership mapping functions $\phi_i:\N\rightarrow [0,1)$, $i=1,2$, for dispatchers and servers, respectively. 
With these membership mapping functions, we can describe the service rate matrix of the sequence of systems by another function $f:[0,1)^2\rightarrow\R_+$ such that for $(i,j)\in\mathcal{W}^N\times \mathcal{V}^N$, $$\mu^N_{i,j} = f(\phi_1(i),\phi_2(j)).$$
This representation significantly simplifies the description of the service rate matrix since the dimension of $\mathcal{U}^N$ explodes as $N\rightarrow\infty$. 
Hence, below we formally introduce what we call an `$f$-sequence' and use it in the rest of the analysis.

\begin{defn}[$f$-sequence]\label{def:f-sequence}
Given a function $f:[0,1)^2\rightarrow\R_+$. A sequence $\{G^N\}_{N}$, where $G^N=(\mathcal{W}^N,\mathcal{V}^N,\mathcal{U}^N,\boldsymbol\lambda^N)$, is an $f$-sequence, if for each $N\in\N$, $u^N_{i,j}=f(\phi_1(i),\phi_2(j))$, $\forall (i,j)\in\mathcal{W}^N\times \mathcal{V}^N$.
\end{defn}
\noindent
The $f$-sequence is a general notion that encompasses the majority of classic queueing systems discussed in the existing literature.
Below we provide a few example scenarios. 

\begin{enumerate}[{[a]}]
    \item \textit{Homogeneous Systems.} A system consists of one class of tasks and one type of servers. A sequence of such systems can be modeled as an $f$-sequence with $f(x,y)=f(0,0)$, $\forall (x,y)\in[0,1)^2$ for some constant $f(0,0)>0$.

    \item \textit{Multiclass Many-Server Systems.} Service rates in such systems are independent of servers, which implies that all servers are statistically identical. A sequence of such systems can be modeled as an $f$-sequence with $f(x,y)=f(x,0)$, $\forall y\in[0,1)$. We will call such systems as \textit{server-independent system} below.

    \item \textit{Heterogeneous-Server Systems.} Service rates in such systems are independent of dispatchers, which implies that a server deals with tasks from various dispatchers at the same rate. A sequence of such systems can be modeled as an $f$-sequence with $f(x,y)=f(0,y)$, $\forall x\in[0,1)$. We will call such systems as \textit{dispatcher-independent system} below.

    \item \textit{Multiclass Multiserver Systems.} When considering fully heterogeneous case, previous works assume that the number of types of tasks and servers is finite, and service rates depend on both types of tasks and servers. A sequence of such systems can be modeled as an $f$-sequence with a stepwise function $f$ on $[0,1)^2$. We will investigate this special case thoroughly in the next section, which is the key tool for us to analyze the performance of a general $f$-sequence in the large-system asymptotic regime.
\end{enumerate}

\section{Main Results}
In this section, we present the main results involving the zero-queueing property of ICRD and SPD approaches. 
As an initial step, in Section~\ref{sec:stepwise-f}, we focus on the special case where $f$ is stepwise. We demonstrate that the ICRD and SPD approaches can effectively transform fully heterogeneous systems into dispatcher-independent and server-independent systems, respectively. This transformation will allow us to employ classical mean-field analysis techniques. 
Building upon these results, in Section~\ref{sec:general-case}, we extend the results to general $f$ using stochastic coupling arguments.

\subsection{Special Case: Stepwise $f$}\label{sec:stepwise-f}
Due to the intractability of directly analyzing the $f$-sequence with a arbitrary function $f$, we draw inspiration from using simple functions to approximate general ones. Thus, we first focus on analyzing the special case where $f$ is stepwise: 
consider an $f$-sequence $\{G^N\}_{N\in\N}$ with $f$ defined as follows:
Let $H$ and $M$ be any two positive integers. 
Also, let
$0=w_0<w_1<...<w_H=1$ and $0=v_0<v_1<...<v_M=1$ be any two partitions of $[0,1)$.
Then for each $h\in[H]$ and $m\in[M]$,
\begin{equation}\label{eq:f-step}
    f(x,y)=\mu_{h,m}\geq 0,\quad \forall (x,y)\in[w_{h-1},w_h)\times[v_{m-1},v_m).
\end{equation}
By Definition~\ref{def:f-sequence}, for each $h\in[H]$ and $m\in[M]$, and all $(i,j)\in\mathcal{W}\times \mathcal{V}$ such that $(\phi_1(i),\phi_2(j))\in [w_{h-1},w_h)\times [v_{m-1},v_m)$, $\mu^N_{i,j}=f(\phi_1(i),\phi_2(j))=\mu_{h,m}.$
Loosely speaking, this is the simplest model where service rate depends on both task and server types.
It does have all the analytical challenges involving order-dependence of queues as mentioned earlier. 
However, as we will see, after appropriate pruning step, it will be analytically tractable in the large-scale setup.

Note that~\eqref{eq:f-step} implies that if $\phi_1(i_1)$ and $\phi_1(i_2)$ are both in $[w_{h-1},w_h)$, the dispatchers $i_1$ and $i_2$ are statistically indistinguishable since the rows $\boldsymbol\mu^N_{i_1}=(\mu^N_{i_1,1},...,\mu^N_{i_1,N})$ and $\boldsymbol\mu^N_{i_2}=(\mu^N_{i_2,1},...,\mu^N_{i_2,N})$ are identical. 
Hence, dispatchers can be classified into finite classes: Let 
$$\mathcal{W}^N_h=\Big\{i\in\mathcal{W}^N:\phi_1(i)\in[w_{h-1},w_h)\Big\}\quad\forall\ h\in[H]\qquad\text{and}\qquad \mathcal{W}^N=\cup_{h\in[H]}\mathcal{W}^N_h.$$  
Similarly, if the columns $j_1$ and $j_2$ of $\mathcal{U}^N$ are the same, servers $j_1$ and $j_2$ can be viewed as the same type. Let 
$$\mathcal{V}^N_m=\{j\in\mathcal{V}^N:\phi_2(j)\in[v_{m-1},v_m)\}\quad\forall\ m\in[M]\qquad\text{and}\qquad \mathcal{V}^N=\cup_{m\in[M]}\mathcal{V}^N_m.$$ 
Since we are interested in large-scale systems, we need to consider the $f$-sequence $\{G^N\}_{N\in\N}$ in a certain asymptotic regime. 
For the asymptotic analysis and to avoid heavy traffic, we consider the sequence in the appropriately defined subcritical regime. 
For the homogeneous case~\cite{MBLW16-3}, if the sequence of systems in the subcritical regime, then the ratio of the total arrival rate to the total service rate in the $N$-th system converges to a fixed constant number strictly less than 1. 
However, the subcritical regime cannot be defined for the above heterogeneous scenario in the same way. 
Since by the same server provides various service rates to different tasks, the total service rate defined by the sum of service rates of all servers is no longer meaningful. 
We define the subcritical regime as follows.
\begin{defn}[Subcritical regime]\label{def: subcritical-finite}The sequence of systems $\{G^N\}_{N\in\N}$  with stepwise $f$ as in~\eqref{eq:f-step}, is said to be in the \textit{subcritical regime} if the followings are satisfied:
\begin{enumerate}[\normalfont(i)]
    \item For each $h\in[H]$, $\lim_{N\rightarrow\infty}\sum_{i\in\mathcal{W}^N_h}\lambda^N_i/N=\lambda_h>0$;
    \item For each $m\in[M]$, $\lim_{N\rightarrow\infty}\sum_{j\in\mathcal{V}^N}\mathds{1}_{(j\in\mathcal{V}^N_m)}/N=v_m-v_{m-1}$;
    \item There exists a matrix  $\mathbf{p}=(p_{h,m},h\in[H],m\in[M])\in[0,1)^{H\times M}$ with unit row sums, such that for all $m\in[M]$,    \begin{equation}\label{eq:subcritical-finite}
       \sum_{h\in[H]} \frac{\lambda_h p_{h,m}}{\mu_{h,m}(v_m-v_{m-1})}<1.
    \end{equation}
    In this case, we also call the sequence $\mathbf{p}$-subcritical.
\end{enumerate}
\end{defn}
Conditions (i) and (ii) in Definition~\ref{def: subcritical-finite} imply that the system scales proportionally in terms of total arrivals rates at dispatchers of each class and the number of servers of each type. 
Condition (iii) indicates that the load per server is strictly less than 1 for all servers under the open-loop probabilistic routing policy where task of type $h$ is assigned to one of the servers of type $m$ (uniformly at random within this type) with probability $p_{h,m}$. 
Note that the matrix $\mathbf{p}$ in the definition of subcritical regime may not be unique. 

It is well-known in the literature~\cite{Stolyar15, Stolyar17, MBLW16-3} that for the homogeneous setting, the subcritical regime not only guarantees stability, but also implies certain performance features such as zero-queueing under policies like JSQ and JIQ in large-scale systems. 
Next, we discuss how we can make this zero-queueing property hold in the heterogeneous scenario as well using the two approaches mentioned in the introduction. 
These are the contents of Sections~\ref{sec:dispatcher-independent} and~\ref{sec:server-indep.}, respectively.


\subsubsection{Intelligent Capacity Reservation and Dispatching (ICRD)}\label{sec:dispatcher-independent}
In this section, we will introduce the ICRD approach. 
Specifically, via carefully pruning the service rate matrix (i.e., reserving the capacity of servers for some tasks), we will construct a sub-system $\Tilde{G}^N\subseteq G^N$ which is a union of $H$ disjoint task-independent systems.
We will show that the sequence $\{\Tilde{G}^N\}_{N}$ has the zero-queueing property under the JIQ policy (Theorem~\ref{thm:fluid-limit-2}). 

Recall the parameters $(\mathbf{w},\mathbf{v})$, $\lambda_h$, $\mu_{h,m}$, $h\in[H]$, $m\in[M]$ as introduced before as assume that the sequence $\{G^N\}_{N}$ is in $\mathbf{p}$-subcritical regime for 
some matrix $\mathbf{p}$ as in Definition~\ref{def: subcritical-finite}. 
Denote $\varepsilon^*_m\coloneqq (v_m-v_{m-1})-\sum_{h\in[H]}\frac{\lambda_hp_{h,m}}{\mu_{h,m}}$, $m\in[M]$ and let $\boldsymbol\varepsilon\coloneqq(\varepsilon_{h,m})_{h\in[H],m\in[M]}$.
Then define a polyhedron $\mathrm{Poly}(\mathbf{p})$ as follows:
\begin{equation}\label{eq:poly}
    \mathrm{Poly}(\mathbf{p})\coloneqq\Big\{\boldsymbol\varepsilon=\big(\varepsilon_{h,m},h\in[H],m\in[M]\big)\in[0,1)^{H\times M}:\boldsymbol\varepsilon \text{ satisfies } \eqref{defn:poly-P}\Big\}
\end{equation}
\begin{equation}\label{defn:poly-P}
    \begin{split}    \varepsilon_{h,1}:\cdots:\varepsilon_{h,M}  =  p_{h,1}:\cdots:p_{h,M},\ \ \forall h\in[H],\qquad
        \sum_{h\in[H]}\varepsilon_{h,m}  \leq  \varepsilon^*_m,\ \  \forall m\in[M].
    \end{split}
\end{equation}
By the definition of the $\mathbf{p}$-subcritical regime, it is easy to check that the polyhedron $\mathrm{Poly}(\mathbf{p})$ is non-empty. \\

\noindent
\textbf{Intelligent Capacity Reservation.} Consider any fixed feasible solution $\boldsymbol\varepsilon\in\mathrm{Poly}(\mathbf{p})$. 
Using such an $\boldsymbol\varepsilon$, for system $G^N$, construct a sub-system $\Tilde{G}^N(\mathbf{p},\boldsymbol\varepsilon)=(\mathcal{W}^N,\mathcal{V}^N,\boldsymbol\lambda^N,\Tilde{\mathcal{U}}^N)$ as follows: 
\begin{itemize}
    \item Step 1: For each $\mathcal{V}^N_m$ in $\Tilde{G}^N(\mathbf{p},\boldsymbol\varepsilon)$, we divide it into $H+1$ separate sets $\big\{\mathcal{V}^N_{h,m}\big\}_{h\in[H]}\bigcup \mathcal{V}^N_{0,m}$ such that\begin{itemize}
        \item[(a)]  $|\Tilde{\mathcal{V}}^N_{h,m}|=\big\lfloor N(\frac{\lambda_hp_{h,m}}{\mu_{h,m}}+\varepsilon_{h,m})\big\rfloor$, $h\in[H]$
        \item[(b)] $|\Tilde{\mathcal{V}}^N_{0,m}|=|\Tilde{\mathcal{V}}^N_m|-\sum_{h\in[H]}|\Tilde{\mathcal{V}}^N_{h,m}|$ .
    \end{itemize}
    \item Step 2: For each, $h\in[H]$, dispatchers in $\mathcal{W}^N_h$ is allowed to assign tasks only to servers in $\bigcup_{m\in[M]}\mathcal{V}^N_{h,m}$. That is, we set  $\Tilde{\mu}^N_{i,j}=\mu^N_{i,j}=\mu_{h,m}$, for $i\in\mathcal{W}^N_h$ and server $j\in\cup_{m\in[M]}\mathcal{V}^N_{h.m}$, and set $\Tilde{u}^N_{i,j}=0$, otherwise.
\end{itemize}
Note that due to Poisson thinning, the constructed system $\Tilde{G}^N(\mathbf{p},\boldsymbol\varepsilon)$ can now be viewed as a disjoint union of $H$ separate dispatcher-independent (i.e., service rates depend only on server-types) systems as follows: 
For $h\in[H]$, $\Tilde{G}^N_h(\mathbf{p},\boldsymbol\varepsilon)$ contains dispatchers $\mathcal{W}^N_h$ and servers $\bigcup_{m\in[M]}\mathcal{V}^N_{h,m}$. 
Also, for each $h\in[H]$, $
\{\Tilde{G}^N_h(\mathbf{p},\boldsymbol\varepsilon)\}_N$ is in the subcritical regime since we have the following relationship between the total arrival rate and the total service rate: $$\lim_{N\rightarrow\infty}\frac{\sum_{i\in\mathcal{W}^N_h}\lambda^N_i}{\sum_{m\in[M]}\mu_{h,m}\big\lfloor N(\frac{\lambda_hp_{h,m}}{\mu_{h,m}}+\varepsilon_{h,m})\big\rfloor}=\frac{\lambda_h}{\lambda_h+\sum_{m\in[M]}\mu_{h,m}\varepsilon_{h,m}}<1.$$

\noindent
\textbf{Dispatching.} For the system $\Tilde{G}^N(\mathbf{p},\boldsymbol\varepsilon)$, each dispatcher will route new tasks under the ordinary JIQ policy. More specifically, when a task of type $h$ arrives, it will be assigned to one of idle servers uniformly at random in $\bigcup_{m\in[M]}\mathcal{V}^N_{h,m}$, if any exists. If no idle servers are available, the task is routed to one of the servers in $\bigcup_{m\in[M]}\mathcal{V}^N_{h,m}$, chosen uniformly at random. \\

\noindent
\textbf{System State.}
Let $X^N_{h,m,l}(t)$ be the number of servers in $\Tilde{\mathcal{V}}^N_{h,m}$ with queue length at least $l\in \N_0$ for $h\in[H]$ and $m\in[M]$ at time $t$.  
Let $\Tilde{X}^N_{h,m,l}(t)=\frac{X^N_{h,m,l}(t)}{N}$ be the scaled quantities and consider 
$\Tilde{X}^N(t)=(\Tilde{X}^N_{h,m,l}(t),h\in[H],m\in[M],l\in\N_0)$ 
to be the system state for $\Tilde{G}^N(\mathbf{p},\boldsymbol\varepsilon)$. For any $N\geq 1$, we view $\Tilde{X}^N$ as an element of the common space 
\begin{equation*}
    \begin{split}
        \Tilde{\chi}=\Big\{\Tilde{X}=(\Tilde{X}_{h,m,l},h\in[H],m\in[M],l\in\N_0):\forall (h,m)\in[H]\times[M],\ 
        \frac{\lambda_hp_{h,m}}{\mu_{h,m}}+\varepsilon_{h,m}\geq\Tilde{X}_{h,m,0}\geq \Tilde{X}_{h,m,1}\geq\cdots\geq0\Big\},
    \end{split}
\end{equation*}
equipped with metric $$\rho(\Tilde{X},\Tilde{X}')=\sum_{h\in[H]}\sum_{m\in[M]}\sum_{l\in\N_0}2^{-k}\frac{|\Tilde{X}_{h,m,l}-\Tilde{X}_{h,m,l}'|}{1+|\Tilde{X}_{h,m,l}-\Tilde{X}_{h,m,l}'|}.$$
The next theorem states that the process $\Tilde{X}^N(t)$ is ergodic and its stationary distribution $\Tilde{X}^N(\infty)$ converges to a deterministic point $\Tilde{x}^*$ as $N\rightarrow\infty$, where $\Tilde{x}^*=(\Tilde{x}^*_{h,m,l},h\in[H],m\in[M],l\in\N_0)$ and 
\begin{equation}\label{eq:fixed-pt-1}
\Tilde{x}^*_{h,m,1}=\frac{\lambda_hp_{h,m}}{\mu_{h,m}}, \quad \Tilde{x}^*_{h,m,l}=0, \quad \forall h\in[H], m\in[M], l\geq 2
\end{equation}
\begin{theorem}\label{thm:fluid-limit-2}
Consider the $f$-sequence $\{G^N\}_{N}$ in the $\mathbf{p}$-subcritical regime, a fixed $\boldsymbol\varepsilon\in\mathrm{Poly}(\mathbf{p})$, and the sequence $\{\Tilde{G}^N(\mathbf{p},\boldsymbol\varepsilon)\}_{N}$ as constructed above. 
The sequence $\{\Tilde{G}^N(\mathbf{p},\boldsymbol\varepsilon)\}_{N}$ is still in the $\mathbf{p}$-subcritical regime and the Markov chain $\Tilde{X}^N(\cdot)$ is ergodic, and $\Tilde{X}^N(\infty)\dto\Tilde{x}^*$, where $\Tilde{x}^*$ is given in~\eqref{eq:fixed-pt-1}.
\end{theorem}
The proof of Theorem \ref{thm:fluid-limit-2} uses \cite[Theorem 2]{AS15}. 
Details are provided in Section~\ref{sec:proof-simple-systems}.
\begin{corollary}[Zero-Queueing]
    Consider the system $\Tilde{G}^N(\mathbf{p},\boldsymbol\varepsilon)$ as constructed above under the ICRD approach. In steady state, the probability that a new task will be assigned to a busy server (a.k.a., the queueing probability) converges to 0 as $N\rightarrow\infty$.
\end{corollary}
\begin{proof}
The proof is immediate from the following observation.
    Let $P^N_{busy}$ be the probability that in steady state, an arriving task will be assigned to a busy server. Then from Theorem~\ref{thm:fluid-limit-2}, 
    \begin{equation*}
        \begin{split}
\lim_{N\rightarrow\infty}P^N_{busy}&=\lim_{N\rightarrow\infty}\PP\Big(\exists (h,m)\in[H]\times[M]\text{ s.t. }\Tilde{X}^N_{h,m,1}(\infty)\geq \frac{\lambda_hp_{h,m}}{\mu_{h,m}}+\varepsilon_{h,m}\Big)\\
&=\lim_{N\rightarrow\infty}\PP\Big(\exists (h,m)\in[H]\times[M]\text{ s.t. }\Tilde{X}^N_{h,m,1}(\infty)\geq \Tilde{x}^*_{h,m,1}+\varepsilon_{h,m}\Big) = 0.
        \end{split}
    \end{equation*}
\end{proof}

\subsubsection{Speed-Priority Dispatching (SPD)}\label{sec:server-indep.}
In this section, we will show that for large-scale system $G^N$ under a certain SPD policy, its evolution can be viewed as that of the union of $M$ server-independent systems. Such a SPD policy is named as $\mathbf{p}$-based JIQ as described in Definition~\ref{def:p-JIQ}. 
Theorem~\ref{thm:fluid-limit-1} shows that with empty initial state, $\mathbf{p}$-based JIQ achieves asymptotically optimal (i.e., zero-queueing).
Let $\boldsymbol\mu=(\mu_1,...,\mu_K)$ be the set containing all distinct values of $\mu_{h,m}$, $(h,m)\in [H]\times [M]$.\\

\noindent
\textbf{System State.} 
Let $X^N_{m,k}(t)$ be the number of servers in $\mathcal{V}^N_m$ that are serving tasks at rate $\mu_k$ in the $N$-th system at time~$t$. 
Define $\Bar{X}^N_{m,k}(t)\coloneqq\frac{X^N_{m,k}(t)}{N}$. We consider
$\Bar{X}^N(t)\coloneqq\big(\Bar{X}^N_{m,k}(t),m\in[M],k\in[K]\big)$
to be the system state and view $\Bar{X}^N$, for all $N$, as elements of the common space 
\begin{equation*}
    \chi = \Big\{x=\big(x_{m,k},m\in[M],k\in[K]\big)\in \R_+^{M\times K}\Big\},
\end{equation*}
equipped with $\ell_1$-norm. \\

\noindent
Recall the $\mathbf{p}$-subcritical regime for some matrix $\mathbf{p}$ as in Definition~\ref{def: subcritical-finite}.
For each $(m,k)\in[M]\times[K]$, define 
\begin{equation}\label{eq:lambda-p-mk}
    \lambda^{\mathbf{p}}_{m,k}\coloneqq\sum_{h\in[H]}p_{h,m}\lambda_h\mathds{1}_{(\mu_{h,m}=\mu_k)},
\end{equation}
and \begin{equation}\label{eq:x-p-mk}
    x^{\mathbf{p}}_{m,k}\coloneqq\frac{\lambda^{\mathbf{p}}_{m,k}}{\mu_k}.
\end{equation}

We now introduce the $\mathbf{p}$-based JIQ policy as follows.
\begin{defn}[$\mathbf{p}$-based JIQ]\label{def:p-JIQ}
    Consider any fixed dispatcher $i\in\mathcal{W}^N_h$, $h\in[H]$. When a task arrives at the dispatcher $i$, it first selects a target type $m^*$ of servers with discrete distribution $\Bar{p}_h=(p_{h,m})_{m\in[M]}$ and then sends the task to one of idle servers uniformly at random in $\mathcal{V}^N_{m^*}$, if any exists. 
    If no idle servers are available, the task is routed to one of the servers in $\mathcal{V}^N_{m^*}$, chosen uniformly at random.
\end{defn}
\begin{remark}\label{rem:p-JIQ}
\normalfont
Under the $\mathbf{p}$-based JIQ policy, for $h\in[H]$, any dispatcher $i\in\mathcal{W}^N_h$ will assign a new task to a server in $\mathcal{V}^N_{m}$ with discrete distribution $\Bar{p}_h=(p_{h,m})_{m\in[M]}$, independent of the current state of the system. 
Hence, for each $m\in[M]$, $\mathcal{V}^N_m$ will receive tasks from $\mathcal{W}^N_h$, $h\in[H]$ with rate $N\lambda_hp_{h,m}$. Also, servers in $\mathcal{V}^N_m$ are statistically identical. 
Thus, the idea is that by the Poisson thinning property, we can view the system $G^N$ as the union of $M$ server-independent systems $\{\hat{G}^N_m\}_{m\in[M]}$; more details are given in the proof of Theorem~\ref{thm:fluid-limit-1}. For each $\hat{G}^N_m$, it consists of servers as the same as that in $\mathcal{V}^N_m$. The process of tasks arriving at $\hat{G}^N_m$ with service rate $\mu_k$ will be a Poisson process with rate $N\lambda^{\mathbf{p}}_{m,k}$.
\end{remark}

Before proceeding toward large-scale analysis, the next lemma shows the stability of the $\mathbf{p}$-based JIQ policy.
\begin{lemma}\label{lem:stable-p-JIQ}
    For large enough $N$, the system $G^N$ is stable under the $\mathbf{p}$-based JIQ, if the sequence $\{G^N\}_N$ is in the $\mathbf{p}$-subcritical regime.
\end{lemma} 
\begin{proof}
    By Remark~\ref{rem:p-JIQ}, it is sufficient to show that each $\hat{G}^N_m$ is stable under JIQ. Since each $\hat{G}^N_m$ is a server-independent system, we can derive the stability of JIQ for each $\hat{G}^N_m$ based on~\cite[Theorem~2.5]{FC98}. For each $m\in[M]$, define 
    \begin{equation}
        \rho^N_m=\max_{J\subseteq\mathcal{V}^N_m}\frac{1}{|J|}\sum_{k\in[K]}\frac{N\lambda^{\mathbf{p}}_{m,k}|J|}{\mu_k|\mathcal{V}^N_m|}=\sum_{k\in[K]}\frac{N\lambda^{\mathbf{p}}_{m,k}}{\mu_k|\mathcal{V}^N_m|}.
    \end{equation}
    $\rho^N_m$ can be understood as the load per server in $\hat{G}^N_m$. Since the sequence $G^N$ is in the $\mathbf{p}$-subcritical regime, we have 
    \begin{equation}
        \sum_{k\in[K]}\frac{N\lambda^{\mathbf{p}}_{m,k}}{\mu_k|\mathcal{V}^N_m|}\xrightarrow{N\rightarrow\infty}\sum_{k\in[K]}\frac{\lambda^{\mathbf{p}}_{m,k}}{\mu_k(v_m-v_{m-1}))}=\sum_{h\in[H]}\frac{\lambda_hp_{h,m}}{\mu_{h,m}(v_m-v_{m-1})}<1,
    \end{equation}
    where the equality comes from the definition of $\lambda^{\mathbf{p}}_{m,k}$. By~\cite[Theorem~2.5]{FC98}, $\rho^N_m<1$ implies the required stability. 
\end{proof}
Note that the large enough $N$ requirement in Lemma~\ref{lem:stable-p-JIQ} is only a technical restriction since our assumption about relative sizes of $\mathcal{V}^N_m$ is asymptotic: that for each $m\in[M]$, $\lim_{N\rightarrow\infty}\frac{|\mathcal{V}^N_m|}{N}=v_m-v_{m-1}$.
The next result shows that the $\mathbf{p}$-based JIQ policy not only ensures the stability of the system but also assigns tasks to idle servers with high probability.
\begin{theorem}\label{thm:fluid-limit-1}
Let the $f$-sequence $\{G^N\}_{N}$ be in the $\mathbf{p}$-subcritical regime and consider it under the the $\mathbf{p}$-based JIQ policy. 
Also, assume that for each $N$-th system, it starts from the all-empty state, i.e., $\Bar{X}^N_{m,k}(0)=0$, $\forall m\in[M], k\in[K]$. Then for any finite $T\geq 0$,  the scaled process $\Bar{X}^N$ converges weakly to the deterministic process $\Bar{X}$ uniformly on $[0,T]$,
where $\Bar{X}(t)=(\Bar{X}_{m,k}(t),m\in[M],k\in[K])$, $t\in[0,T]$ and each $\Bar{X}_{m,k}(t)$ satisfies the following differential equation: 
$$\frac{d \Bar{X}_{m,k}(t)}{dt}=\lambda^{\mathbf{p}}_{m,k}-\mu_k \Bar{X}_{m,k}(t).$$
\end{theorem}
The complete proof of Theorem~\ref{thm:fluid-limit-1} will be provided in Section~\ref{sec:proof-simple-systems}.
\begin{corollary}[Zero-Queueing]
    Consider the system $G^N$ under $\mathbf{p}$-based JIQ. Given the all-empty initial state, for any $T>0$, the probability that during the time interval $[0,T]$, an arriving task will be assigned to a busy server converges to 0 as $N\rightarrow0$.
\end{corollary}
\begin{proof}
    Observe that given the idle initial state, $\Bar{X}_{m,k}(t)$ cannot exceed $\frac{\lambda^{\mathbf{p}}_{m,k}}{\mu_k}$ for any $t\geq 0$. $\Bar{X}^N$ converges weakly to $\Bar{X}$ uniformly on $[0,T]$ for any finite $T\geq 0$. Combined with the fact that $$\sum_{k\in[K]}\frac{\lambda^{\mathbf{p}}_{m,k}}{\mu_k}=\sum_{k\in[K]}\frac{\sum_{h\in[H]}p_{h,m}\lambda_h\mathds{1}_{(\mu_{h,m}=\mu_k)}}{\mu_k}=\sum_{h\in[H]}\frac{p_{h,m}\lambda_h}{\mu_{h,m}}< v_m-v_{m-1},$$ we have that with idle initial state and under $\mathbf{p}$-based JIQ, on any finite time interval a new task arriving at the $N$-system $G^N$  will be assigned to an idle server with high probability tending 1 as $N\to\infty$. 
\end{proof}
\begin{remark}\normalfont
Although we only show the transient limit in Theorem~\ref{thm:fluid-limit-1}, we expect that the convergence also holds in steady state, as evidenced numerically in Section~\ref{sec:numerical results}. 
However, extending the convergence result steady state poses significant technical challenge. 
This is primarily due to the lack of favorable properties, such as monotonicity, in the process $\bar{X}^N(t)=(\bar{X}^N_{m,k}(t),m\in[M],k\in[K])$. 
In the literature, monotonicity (state-wise dominance) has been the key to demonstrating the convergence of steady states in qualitatively similar scenarios~\cite{AS15}. 
In the current scenario, however, we do not expect any straightforward monotonicity to hold. 
Indeed, since the expected workload of each task depends on its class, while comparing two servers of the same type, the server with more tasks may have less expected workload, which implies that there is no obvious stochastic dominance between two system states. 
\end{remark}

\subsection{General Function Case}\label{sec:general-case}
In this section, we will
extend Theorems~\ref{thm:fluid-limit-2} and~\ref{thm:fluid-limit-1} to the general $f$ case. Since the function $f$ may not be stepwise, we cannot classify dispatchers or servers into finite groups. 
Hence, we need to generalize the notion of $\mathbf{p}$-subcriticality. 
Also, for the consistency in defining the sequence $\{G^N\}_N$, we make some regular assumption about the sequence $\{G^N\}_N$ below:

Recall the membership mapping $\phi_1$ and $\phi_2$ for dispatchers and servers, respectively.
\begin{assumption}\label{ass:f-sequence}
    \begin{enumerate}[\normalfont(i)]
        \item (Arrival rate function) There exists an integrable function $\lambda(\cdot):[0,1)\rightarrow\R_+$ with $\int_0^1 \lambda(x)dx=a>0$ such that 
        $\lambda(\phi_1(i))=\lambda^N_i$, $\forall i\in\mathcal{W}^N, N\in\N$.
        \item (Service rate function) The function $f:[0,1)^2\rightarrow\R_+$ is Riemann integrable, and there exists $\mu^o>0$ such that for all $x\in[0,1)$, $|\{y\in[0,1):f(x,y)\geq \mu^o\}|>0$, where $|\cdot|$ is Lebesgue measure. 
        \item (Regularity of membership map) For any subinterval $E\subseteq [0,1)$, $$\lim_{N\rightarrow\infty}\sum_{i\in\mathcal{W}^N}\frac{\mathds{1}_{(\phi_1(i)\in E)}}{W(N)}= \lim_{N\rightarrow\infty}\sum_{j\in\mathcal{V}^N}\frac{\mathds{1}_{(\phi_2(j)\in E)}}{N} =|E|.$$ 
        \item $\lim_{N\rightarrow\infty}\frac{W(N)}{N}=\xi>0$, where $\xi$ is a constant.
    \end{enumerate}
\end{assumption}
\begin{remark}\normalfont
In Assumption~\ref{ass:f-sequence}, condition (i) implies that the arrival rate of tasks at each dispatcher is determined by the function $\lambda(\cdot)$ and the membership mapping $\phi_1(\cdot)$. Hence, the arrival rates depend on the dispatcher type and can be inhomogeneous. 
One special case is that $\lambda(\cdot)$ is a constant function so the arrival rate of tasks at each dispatcher is the same. 
Next, condition (ii) indicates that for each type of task, there exists some server that can process the task efficiently. 
Condition (iii) suggests that the number of dispatchers (servers) mapped into any subinterval $E\in[0,1)$ increases proportionally to the system size.
Finally, condition (iv) ensures that the total arrival rate scales proportionally with the total number of servers. 
This is a standard assumption in the literature; see for example~\cite{TX13, TX17}.
However, instead of (iv), one could also assume that $\lambda^N_i=\frac{N}{W(N)}\lambda(\phi_1(i))$ in (i), which would allow the number of task types to scale differently from the number of servers. 
\end{remark}
Based on the above assumption, we define the subcritical regime for the general $f$-sequence as follows.
\begin{defn}[Subcritical regime]\label{def: subcritical-general}
The $f$-sequence $\{G^N\}_N$ is in the subcritical regime if the following is satisfied:
There exist a pair of partitions $(\mathbf{w},\mathbf{v})=(0=w_0<w_1<\cdots<w_H=1,0=v_0<v_1<\cdots<v_M=1)$ of $[0,1)$ and a matrix $\mathbf{p}\in[0,1)^{H\times M}$ with unit row sums such that 
\begin{equation}\label{eq:P-subcritical-infinite}
    \rho_m(\mathbf{w},\mathbf{v},\mathbf{p})\coloneqq\sum_{h\in [H]}\frac{p_{h,m}\lambda_h}{(v_m-v_{m-1})\mu^*_{h,m}}<1, \quad m\in[M],
\end{equation}
where, for each $h\in[H]$ and $m\in[M]$, $\lambda_h=\frac{1}{\xi}\int_{w_{h-1}}^{w_h}\lambda(x)dx$ and
$\mu^*_{h,m}=\min_{(x,y)\in[w_{h-1},w_h)\times[v_{m-1},v_m)}f(x,y).$
\end{defn}
The above definition of subcritical regime involves a pair of partitions $(\mathbf{w},\mathbf{v})$ and a matrix $\mathbf{p}\in[0,1)^{H\times M}$ with unit row sums, which may not be unique. 
Hence, we will use the term ``$(\mathbf{w},\mathbf{v},\mathbf{p})$-subcritical regime'' to specify the pair of partitions $(\mathbf{w},\mathbf{v})$ and the  matrix $\mathbf{p}\in[0,1)^{H\times M}$ which will be used in the analysis. 
Given the pair of partitions $(\mathbf{w},\mathbf{v})$, we may have multiple stochastic matrices $\mathbf{p}\in [0,1)^{H\times M}$ such that \eqref{eq:P-subcritical-infinite} holds. Different stochastic matrices $\mathbf{p}$ may affect the maximal load per server. Ideally, we need to choose a $\mathbf{p}$ making the maximal load per server as small as possible. Denote $\rho(\mathbf{w},\mathbf{v})\coloneqq\min_{\mathbf{p}}\max_{m}\rho_m(\mathbf{w},\mathbf{v},\mathbf{p})$ as the load per server with partition $(\mathbf{w},\mathbf{v})$.

\begin{remark}\normalfont
An interesting algorithmic and technical question here is: \emph{given a general $f$-sequence how to find an appropriate partition $(\mathbf{w},\mathbf{v})$ and check if the polyhedron defined by \eqref{eq:P-subcritical-infinite} is non-empty to determine subcriticality?}
Even though, this is not the main focus of the current work, in~\ref{sec:find-wvp-subcritical}, we 
have designed an algorithm (Algorithm~\ref{algo: to find subcritical}) that can efficiently find a desired partition $(\mathbf{w},\mathbf{v})$ and a corresponding matrix $\mathbf{p}$ satisfying~\eqref{eq:P-subcritical-infinite} when they exist, and 
otherwise outputs that there must exist some servers suffering from heavy workload no matter how the partition is made.  
\end{remark}


\noindent
\textbf{Finitely many groups of possibly non-identical dispatchers and servers.}
For the rest of the analysis in this section, we assume that the $f$-sequence is in the $(\mathbf{w},\mathbf{v},\mathbf{p})$-subcritical regime with the partition $(\mathbf{w},\mathbf{v})$, where $(\mathbf{w},\mathbf{v})=(0=w_0<w_1<\cdots<w_H=1,0=v_0<v_1<\cdots<v_M=1)$. 
Now, with the tuple $(\mathbf{w},\mathbf{v},\mathbf{p})$, we can construct a $f'$-sequence $\big\{G'^N=(\mathcal{W}^N,\mathcal{V}^N,\boldsymbol\lambda^N,\mathcal{U}'^N)\big\}_N$ such that:
\begin{enumerate}[(i)]
    \item For each $N$, $G^N$ and $G'^N$ have the same sets of dispatchers and servers, respectively. 
    \item For any $(x,y)\in[w_{h-1},w_{h})\times[v_{m-1},v_m)$, 
    \begin{equation}\label{eq:f'}
        f'(x,y)=\min_{(a,b)\in[w_{h-1},w_{h})\times[v_{m-1},v_m)}f(x,y).
    \end{equation}
\end{enumerate}
It is straightforward to check that the $f'$-sequence $\{G'^N\}_N$ is also in the $(\mathbf{w},\mathbf{v},\mathbf{p})$-subcritical regime and is a sequence with finite types of dispatchers and servers that we have discussed in Section~\ref{sec:stepwise-f}. 
Since by the definition, $f'(x,y)\leq f(x,y)$ for all $(x,y)\in[0,1)^2$, it is intuitive that the system $G^N$ will have better performance than the system $G'^N$ in terms of queue length, which we will formalize using coupling arguments in Section~\ref{sec:proof-general-systems}.
Now, for each $N$, we split the sets $\mathcal{W}^N$ and $\mathcal{V}^N$ into subsets based on the partition $(\mathbf{w},\mathbf{v})$. Let
\begin{align*}
    \mathcal{W}^N_h&=\Big\{i\in\mathcal{W}^N:\phi_1(i)\in[w_{h-1},w_h)\Big\}\text{ for each }h\in[H] \text { and  }\mathcal{W}^N=\bigcup_{h\in[H]}\mathcal{W}^N_h\\
    \mathcal{V}^N_m&=\Big\{j\in\mathcal{V}^N:\phi_2(j)\in[v_{m-1},v_m)\Big\}\text{ for each }m\in[M] \text { and  }\mathcal{V}^N=\bigcup_{m\in[M]}\mathcal{V}^N_m.
\end{align*}
Crucially, however, note that the classification of dispatchers and servers here no longer means that the dispatchers or servers in the same class are statistically identical. 
It is only used in constructing the subsystem via ICRD and designing the $\mathbf{p}$-based JIQ policy.\\

\noindent
\textbf{ICRD approach.}
Give the above groups of dispatchers and servers, construct a subsystem $\Tilde{G}^N\subseteq G^N$ via the ICRD approach in Section~\ref{sec:dispatcher-independent} by choosing any $\boldsymbol\varepsilon\in\mathrm{Poly}(\mathbf{p})$ in~\eqref{eq:poly}. 
The theorem below then states that the vanilla JIQ policy achieves the asymptotic zero-queueing in the $\Tilde{G}^N$. 

\begin{theorem}
\label{thm:general-JIQ}
Consider the $f$-sequence $\{G^N\}_N$ in $(\mathbf{w},\mathbf{v},\mathbf{p})$-subcritical regime. 
Choose $\boldsymbol\varepsilon\in\mathrm{Poly}(\mathbf{p})$ in~\eqref{eq:poly}
and consider the subsystem $\{\Tilde{G}^N(\mathbf{p},\boldsymbol\varepsilon)\}_N$ constructed as above under the JIQ policy.
Then the steady-state probability that an arriving task joins a busy server converges to 0 as $N\to\infty$.
\end{theorem}
Note that Theorem~\ref{thm:general-JIQ} does not claim the uniqueness of the fixed point of the limiting system as Theorem~\ref{thm:fluid-limit-2} does. The main reason is that even though we know the total service rate is equal to the total arrival in steady state, the fraction of busy servers cannot be determined since the complexity of the service rate allows that different subsets of servers can provide the same total service rate.
\vspace{.2cm}

\noindent
\textbf{SPD approach.}
The next theorem will show the asymptotic zero-queueing property of $\mathbf{p}$-base JIQ in $G^N$ generalizing Theorem~\ref{thm:fluid-limit-1}.
Let $\boldsymbol\mu^*=(\mu^*_1,...,\mu^*_K)$ be the set containing all distinct values of $\mu^*_{h,m}$, $h\in[H], m\in[M]$.
For $m\in[M]$ and $k\in[K]$, 
\begin{equation*}
    \lambda^{\mathbf{p}}_{m,k}=\sum_{h\in[H]}p_{h,m}\lambda_h\mathds{1}_{(\mu^*_{h,m}=\mu^*_k)}\quad 
\text{and}\quad
    x^{\mathbf{p}}_{m,k}=\frac{\lambda^{\mathbf{p}}_{m,k}}{\mu^*_k}.
\end{equation*}
\begin{theorem}
\label{thm:f-zero-queue}
Consider the $f$-sequence $\{G^N\}_N$ in the $(\mathbf{w},\mathbf{v},\mathbf{p})$-subcritical regime. Assume that for all $N\geq 1$, the system $G^N$ starts from the all-empty state, then under the $\mathbf{p}$-based JIQ policy, for any $T\geq 0$, $\sup_{t\in[0,T]}\Bar{X}^N_m(t)\leq \sum_{k}x^{\mathbf{p}}_{m,k}$ with high probability tending to 1 as $N\rightarrow\infty$, where $\Bar{X}^N_m(t)$ is the fraction of busy servers in $\mathcal{V}^N_m$ at time $t$.
Consequently, with the all-empty initial state, on any finite time interval, the probability that an arriving task is assigned to a busy server converges to 0 as $N\rightarrow\infty$.
\end{theorem}
The proofs of Theorem~\ref{thm:general-JIQ} and Theorem~\ref{thm:f-zero-queue} are both based on coupling arguments since we have shown results for the system $G'^N$ in Section~\ref{sec:stepwise-f}. The complete proof is provided in Section~\ref{sec:proof-general-systems}. 

\section{Proof for the Stepwise $f$ Case}\label{sec:proof-simple-systems}
In this section, we will prove Theorems~\ref{thm:fluid-limit-2} and~\ref{thm:fluid-limit-1}, staring with the proof of Theorem~\ref{thm:fluid-limit-2}. 
The key element in this proof is that the system $\Tilde{G}^N$ is a union of $H$ dispatcher-independent systems and use the results of Stolyar~\cite{AS15} where such systems have been analyzed comprehensively. 
For completeness, we have briefly included the model and the main result of \cite{AS15} in~\ref{AS15}.
\begin{proof}[Proof of Theorem~\ref{thm:fluid-limit-2}]
 By the construction of $\Tilde{G}^N(\mathbf{p},\boldsymbol\varepsilon)$, we have that for any fixed $N$, $\Tilde{G}^N_h(\mathbf{p},\boldsymbol\varepsilon)$ evolves independently of $\Tilde{G}^N_{h'}$ for any $h'\neq h\in[H]$. 
 Consequently, $\Tilde{X}^N_h(\cdot)=(\Tilde{X}^N_{h,m,k}(\cdot),m\in[M],k\in\N_0)$, $h\in[H]$ are mutually independent. Hence, we can consider $\{\Tilde{G}^N_h(\mathbf{p},\boldsymbol\varepsilon)\}_{N}$, $h\in[H]$, separately, and  it is sufficient to show that for all $h\in[H]$, $\Tilde{X}^N_h(\cdot)$ is ergodic, and $\Tilde{X}^N_h(\infty)\dto \Tilde{x}^*_h=(\Tilde{x}^*_{h,m,k},m\in[M],k\in \N_0)$, where \begin{equation}
    \Tilde{x}^*_{h,m,1}=\frac{\lambda_hp_{h,m}}{\mu_{h,m}},\quad \Tilde{x}^*_{h,m,k}=0,\quad \forall m\in [M],k\geq 2.
\end{equation}
Fix any $h\in[H]$ and consider $\Tilde{X}^N_h(t)$ as the system state of $\Tilde{G}^N_h$ at time $t\geq 0$ for all $N\geq 1$. 
We need to check that $\{\Tilde{G}^N_h\}_{N\geq 1}$ satisfies the model structure and asymptotic regime proposed in~\cite{AS15}(\ref{AS15}). 
By the construction of $\Tilde{G}^N_h$, we can view the dispatcher set $\Tilde{\mathcal{W}}^N_h$ as one dispatcher that receives tasks at rate $\sum_{i\in\Tilde{\mathcal{W}}^N_h}\Tilde{\lambda}^N_i=N\lambda_h$. 
This is because all dispatchers in $\Tilde{\mathcal{W}}^N_h$ can assign tasks to all servers in $\cup_{m\in[M]}\Tilde{\mathcal{V}}^N_{h,m}$ and for any two dispatchers $i,i'\in\Tilde{\mathcal{W}}^N_h$, if they assign tasks to the same server, the server will provide service with the same rate that does not depend on $i$ or $i'$. 
For the server set $\cup_{m\in[M]}\Tilde{\mathcal{V}}^N_{h,m}$, by construction, it is the union of $M$ server pool $\Tilde{\mathcal{V}}^N_{h,m}$, $m\in[M]$ where servers within each pool are identical. 
Also, note that the size of each $\Tilde{\mathcal{V}}^N_{h,m}$ is $\lfloor N(\frac{\lambda_hp_{h,m}}{\mu_{h,m}}+\varepsilon_{h,m})\rfloor$ and $\sum_{m\in[M]}\lfloor N(\frac{\lambda_hp_{h,m}}{\mu_{h,m}}+\varepsilon_{h,m})\rfloor\mu_{h,m}>N\lambda_h$. Hence, both the arrival rate of tasks and the server pools' sizes increase in proportion to $N$ and satisfy the subcritical load condition. 
Thus, by Theorem~\ref{thm:AS15}, we can conclude that for large enough $N$, $\Tilde{X}^N_h(\cdot)$ is ergodic, and $\Tilde{X}^N_h(\infty)\dto \Tilde{x}^*_h=(\Tilde{x}^*_{h,m,k},m\in[M],k\in \N_0)$, where $\Tilde{x}^*_h$ is the unique solution to the following equation: 
\begin{equation}\label{eq:fixed-point-dispatcher-indep}
    \begin{split}
        \lambda_h=&\sum_{m\in [M]}\mu_{h,m}x^*_{h,m,1}\\
        \frac{\mu_{h,m}\Tilde{x}^*_{h,m,1}}{\frac{\lambda_hp_{h,m}}{\mu_{h,m}}+\varepsilon_{h,m}-\Tilde{x}^*_{h,m,1}}=&\frac{\mu_{h,m'}\Tilde{x}^*_{h,m',1}}{\frac{\lambda_hp_{h,m'}}{\mu_{h,m'}}+\varepsilon_{h,m'}-\Tilde{x}^*_{h,m',1}},\quad m,m'\in[M]\\
        \Tilde{x}^*_{h,m,k}=&0,\quad \forall m\in[M],k\geq 2.
    \end{split}
\end{equation}
Since $\boldsymbol\varepsilon=(\varepsilon_{h,m},h\in[H],m\in[M])\in \mathrm{Poly}(\mathbf{p})$, then for each $h\in[H]$, $$\varepsilon_{h,1}:\cdots:\varepsilon_{h,M}  =  p_{h,1}:\cdots:p_{h,M},$$ which implies that the unique solution to \eqref{eq:fixed-point-dispatcher-indep} is $\Tilde{x}^*_{h,m,1}=\frac{\lambda_hp_{h,m}}{\mu_{h,m}}$, $\forall h\in[H], m\in[M]$.
\end{proof}

Next, we will consider the $\mathbf{p}$-based JIQ as the routing policy used in the heterogeneous systems and prove Theorem~\ref{thm:fluid-limit-1}. 
The main idea here is that under the $\mathbf{p}$-based JIQ, the evolution of the system $G^N$ can be coupled with that of a union of $M$ server-independent systems which route tasks by JIQ. Hence, we need some preliminary results about the server-independent systems about which we provide a thorough discussion in~\ref{app:server-independent-fluid}. 
Also, for the analysis of the evolution of $G^N$ under the $\mathbf{p}$-based JIQ policy, define the Markovian state descriptor as follows: 
For $G^N=(\mathcal{W}^N,\mathcal{V}^N,\boldsymbol\lambda^N,\mathcal{U}^N)$, let $Z^N_j(t)\in\N_0$ be the queue length of server $j\in \mathcal{V}^N$ at time $t\geq 0$. For each server $j\in\mathcal{V}^N$ with $Z^N_j(t)>0$, let 
\begin{equation}\label{defn:queue-of-j}
    X^N_j(t)\coloneqq\Big(x^{(1)}_j(t),...,x^{(Z^N_j(t))}_j(t)\Big)\in\mathcal{H}^{\infty},
\end{equation}
where $x^{(n)}_j(t)$ is the type of the $n^{th}$ task at queue $j$ at time $t$ for $n\in[Z^N_j(t)]$ and $\mathcal{H}^{\infty}$ be the set of finitely terminated
sequences taking values in $\mathcal{H}\coloneqq[H]$. Define the weighted queue length $Q^N_j(X^N_j(t))$ as 
\begin{equation}\label{defn:weighted-queue-length}
    Q^N_j(X^N_j(t))\coloneqq\sum_{n=1}^{Z^N_j(t)}1/\mu^N_{x^{(n)}_j,j}.
\end{equation}
If the server $j$ is idle, let $X^N_j(t)=0$ and $Q^N_j(X^N_j(t))=0$. For convenience, denote $Q^N_j(t)=Q^N_j(X^N_j(t))$ and $\mathbf{Q}^N(t)=\big(Q^N_j(t),j\in\mathcal{V}^N\big)\in\R_+^{N}$.

\begin{proof}[Proof of Theorem~\ref{thm:fluid-limit-1}]
Under the $\mathbf{p}$-based JIQ policy, when a task arrives at dispatcher $i\in \mathcal{W}^N_h$, $h\in[H]$, the dispatcher selects the target type of servers with the discrete distribution $\Bar{p}_h=(p_{h,m})_{m\in[M]}$, which is independent of the current state in the system. 
By Poisson thinning, we can view $G^N$ under the $\mathbf{p}$-based JIQ as a union of $M$ server-independent systems under the JIQ policy. 
Hence, we construct $\hat{G}^N=\cup_{m\in[M]}\hat{G}^N_m$, where $\hat{G}^N_m$ is defined as follows:
for each $m\in[M]$, \begin{itemize}
    \item The set $\hat{\mathcal{V}}^N_m$ of servers in $\hat{G}^N_m$ is the same as $\mathcal{V}^N_m$;
    \item There are $K$ dispatchers in $\hat{G}^N_m$ and denote the dispatcher set as $\hat{\mathcal{W}}^N_m$; each dispatcher $k\in\hat{\mathcal{W}}^N_m$ handles the assignment of only one type of tasks called type $k$ tasks; the arrival of tasks at dispatcher $k\in\hat{\mathcal{W}}^N_m$ is a Poisson process with rate $N\sum_{h\in[H]}\lambda^{\mathbf{p}}_{m,k}$, independently of the other processes;
    \item The service time of type $k$ task in $\hat{G}^N_m$ is exponentially distributed with mean $\frac{1}{\mu_k}$.
\end{itemize}
For the system $\hat{G}^N$, let $\hat{Z}^N_j(t)$ be the queue length of server $j\in\hat{G}^N$ at time $t$ and let $$\hat{X}^N_j(t)\coloneqq\Big(\hat{x}^{(1)}_{j}(t),...,\hat{x}^{(\hat{Z}^N_j(t))}_{j}(t)\Big)\in\mathcal{K}^{\infty}$$ 
be the ordering of tasks at the queue of the server $j$, where $\mathcal{K}^{\infty}$ be the set of finitely terminated sequences taking values in $\mathcal{K}\coloneqq[K]$; $\hat{X}^N_j(t)=0$ if the server $j$ is idle. 
Define the weighted queue length $\hat{Q}^N_j(t)$ of server $j\in\hat{G}^N$ at time $t$ as:
$$\hat{Q}^N_j(t)\coloneqq\sum_{n=1}^{\hat{Z}^N_j(t)}\sum_{k\in[K]}\mathds{1}_{(\hat{x}^{(n)}_{j}(t)=k)}\frac{1}{\mu_k}.$$
Let $\hat{X}^N_{m,k}(t)$ be the number of servers in $\hat{\mathcal{V}}^N_m$ that are serving type $k$ tasks at time $t$, $k\in\hat{\mathcal{W}}^N_m$, $m\in [M]$, $N\geq 1$, and $\hat{x}^N_{m,k}(t)=\frac{\hat{X}^N_{m,k}(t)}{N}$. 
Define $\hat{x}^N(t):=(\hat{x}^N_{m,k}(t),m\in[M],k\in[K])$, $t\geq 0$.
\vspace{3mm}

\noindent
Next, we show that starting from all empty state, under a natural coupling as described below, the system queue length process in $G^N$ and $\hat{G}^N$ are identical.
This will mean that it is enough to establish fluid limit for the evolution of queues in  $\hat{G}^N$. 
The system queue length process in $G^N$ and $\hat{G}^N$ are identical, if for each $j\in\mathcal{V}^N$, $Z^N_j=\hat{Z}^N_j$, and when $Z^N_j=\hat{Z}^N_j \geq 1$, server $j$ in both $G^N$ and $\Tilde{G}^N$ processes their $n$-th tasks at the same rate 
for $1\leq n\leq Z^N_j$.
\\ 

\noindent
\textbf{Coupling construction.} 
We will use forward induction on event times.
Suppose at time $t_0$ the system states of $G^N$ and $\hat{G}^N$ are identical.
Now, let $t_1$ be the next event time.
We couple the evolution of queues in $\hat{G}^N$ and $G^N$ in the following way:
\vspace{.2cm}

\noindent
\textit{Departure.} By the definition of the same state of $G^N$ and $\hat{G}^N$, for all $j\in\mathcal{V}^N(\text{or }\hat{\mathcal{V}}^N)$, server $j$ in $G^N$ and server $j$ in $\hat{G}^N$ are processing tasks at the same rate. Hence, we synchronize the departure epochs of all servers in the two systems. That is, if there is a departure from server $j$ in system $G^N$, then a departure happens at server $j$ in system $\hat{G}^N$ as well.
\vspace{.2cm}

\noindent
\textit{Arrival.} We know that the rate at which the server set $\mathcal{V}^N_m$ in system $G^N$ receives tasks which will be processed at rate $\mu_k$ is given by $N\sum_{h\in[H]}\lambda_hp_{h,m'}\mathds{1}_{(\mu_{h,m'}=\mu_k)}=N\lambda^{\mathbf{p}}_{m',k}$. By the construction of $\hat{G}^N$, the arrival rate of tasks at dispatcher $k$ in $\hat{G}^N_{m}$ is $N\lambda^{\mathbf{p}}_{m',k}$ so we can synchronize the arrival processes of systems $G^N$ and $\hat{G}^N$ as follows.  By the construction of $\hat{G}^N$ Suppose $t_1$ is a time for a new task arriving at dispatcher $i\in\mathcal{W}^N_h$, $h\in[H]$. Also, suppose it selects the target type $m\in[M]$ of servers (using the discrete distribution $\Bar{p}_h$) and by JIQ, the new task is assigned to one of the servers in $\mathcal{V}^N_m$. Then, for the system $\hat{G}^N$, at time $t_1$, we let  a new task arrive at the dispatcher $k\in\hat{\mathcal{W}}^N_m$ in $\hat{G}^N_m$ and assign it to server $j\in\hat{\mathcal{V}}^N_m$ as well.
\vspace{.2cm}

\noindent
It is easy to check that under the natural coupling, the two systems evolve according to their own statistical laws, and the system states of $G^N$ and $\hat{G}^N$ remain identical if they start from the same initial state. 
Hence, in order to show the desired theorem, it is sufficient to show that if for all large enough $N$, the system $\hat{G}^N$ starts with all-empty initial state and routes tasks under the JIQ policy, then for any $T>0$, the process $\hat{x}^N$ converges weakly to the deterministic process $\Bar{X}$ uniformly on $[0,T]$,
where $\Bar{X}(t)=(\Bar{X}_{m,k}(t),m\in[M],k\in[K])$,$t\in[0,T]$ and each $\Bar{X}_{m,k}(t)$ satisfies the following differential equation: 
\begin{equation}\label{eq:proof-1-1}
    \frac{d \Bar{X}_{m,k}(t)}{dt}=\lambda^{\mathbf{p}}_{m,k}-\mu_k \Bar{X}_{m,k}(t).
\end{equation}
Moreover, for each $N\geq 1$, $\hat{G}^N_{m}$, $m\in[M]$ are mutually independent which implies that $\hat{x}^N_m(t)=(\hat{x}^N_{m,k}(t),k\in[K])$, $m\in[M]$ are mutually independent. Hence, we can consider the sequences $\{\hat{G}^N_{m}\}_{N\geq 1}$, $m\in[M]$, separately. Also, by Proposition~\ref{prop:fluid-multiclass},  for each $m\in[M]$, $\{\hat{x}^N_m(t)=(\hat{x}_{m,k}(t),k\in[K]),t\in[0,T]\}$ converges weakly to $\{\Bar{X}_m(t)=(\Bar{X}_{m,k}(t),k\in[K]),t\in[0,T]\}$ where $\Bar{X}_{m,k}$ is defined as \eqref{eq:proof-1-1}. Due to the independence of $\{\hat{G}^N_m\}_{N\geq 1}$, $m\in[M]$, the desired convergence of $x^N$ result holds.
\end{proof}

\section{Analysis of General Systems}\label{sec:proof-general-systems}
In this section, we prove Theorems~\ref{thm:general-JIQ} and~\ref{thm:f-zero-queue} by stochastic coupling arguments.
As discussed in Section~\ref{sec:general-case}, if the $f$-sequence is in the $(\mathbf{w},\mathbf{v},\mathbf{p})$-subcritical regime, we can construct a system $G'^N$ to lower bound the performance of the system $G^N$. 
As constructed in Section~\ref{sec:general-case}, the service rates of the sequence $\{G'^N\}_N$ are determined by the stepwise function $f'(\cdot,\cdot)$ defined by \eqref{eq:f'}. Hence, the sequence $\{G'^N\}_N$ is the special case discussed in Section~\ref{sec:stepwise-f}. 
We start by proving  Theorem~\ref{thm:general-JIQ}, in which we primarily compare the number of busy servers between two constructed subsystems $\Tilde{G}^N\subseteq G^N$ and $\Tilde{G}'^N\subseteq G'^N$ (see Claim~\ref{claim:monotonicity} below).

\begin{proof}[Proof of Theorem~\ref{thm:general-JIQ}]
Recall that by Theorem~\ref{thm:fluid-limit-2}, we have the following result for the $f'$-sequence $\{G'^N\}_N$. 
For each $G'^N$, via ICRD, we can construct a subsystem $\Tilde{G}'^N$ as a union of $H$ separate dispatcher-independent systems $\{\Tilde{G}'^N_h\}_{h\in[H]}$, where for each $h\in[H]$, $\Tilde{G}'^N_h$, consists of dispatchers $\Tilde{\mathcal{W}}'^N_h$ and servers $\Tilde{\mathcal{V}}'^N_h\coloneqq\cup_{m\in[M]}\Tilde{\mathcal{V}}'^N_{h,m}$. 
Now, let $X'^N_{h,m,k}(t)$ be the number of servers in $\Tilde{\mathcal{V}}'^N_{h,m}$ with queue length at least $k\in\N_0$ at time $t\geq 0$. 
Let $\Tilde{X}'^N_{h,m,k}(t)=\frac{X'^N_{h,m,k}(t)}{N}$ be the scaled quantity. 
Theorem~\ref{thm:fluid-limit-2} shows that  $\Tilde{X}'^N(\infty)\dto\Tilde{x}'^*$, where $\Tilde{x}'^*=(\Tilde{x}'^*_{h,m,k},h\in[H],m\in[M],k\in\N_0)$ and 
\begin{equation}\label{eq:tilde-G'-steady}
    \Tilde{x}'^*_{h,m,1}=\frac{\lambda_hp_{h,m}}{\mu_{h,m}}, \quad \Tilde{x}'^*_{h,m,k}=0, \quad \forall h\in[H], m\in[M],k\geq 2.
\end{equation}
Given $G^N$, we can construct a subsystem $\Tilde{G}^N$ which is a union of $H$ separate systems $\{\Tilde{G}^N_h\}_{h\in[H]}$. 
Also, $\Tilde{G}^N_h$ contains the same dispatcher set and server set as $\Tilde{G}'^N_h$, $h\in[H]$. For $h\in[H]$ and $m\in[M]$, let $X^N_{h,m,k}(t)$ be the number of servers in $\Tilde{\mathcal{V}}^N_{h,m}$ with queue length at least $k\in\N_0$ at time $t\geq 0$. Let $\Tilde{X}^N_{h,m,k}(t)=\frac{X'^N_{h,m,k}(t)}{N}$. 
Since for all $(x,y)\in[0,1)^2$ $f'(x,y)\leq f(x,y)$, it is intuitive that $\Tilde{X}^N_{h,m,k}(\infty)\leq \Tilde{X}'^N_{h,m,k}(\infty)$ for all $(h,m,k)\in[H]\times[M]\times\N_0$, which is shown in the next claim.
\begin{claim}\label{claim:monotonicity}
Fix any $N\in\N_0$ and $h\in[H]$. Consider two processes $(\Tilde{X}^N_{h,m,k}(\cdot), m\in[M],k\in\N_0)$ and $(\Tilde{X}'^N_{h,m,k}(\cdot), m\in[M],k\in\N_0)$. 
Then the processes can be constructed on a common probability space so that 
\begin{equation}\label{eq:5.1-ineq}
\Tilde{X}^N_{h,m,k}(t)\leq \Tilde{X}'^N_{h,m,k}(t), \quad \forall m\in[M],k\in\N_0,t\geq 0    
\end{equation}
holds almost surely, given that the inequality holds at time $t=0$. Consequently, 
\begin{equation*}
    \Tilde{X}^N_{h,m,k}(\infty)\leq \Tilde{X}'^N_{h,m,k}(\infty), \quad \forall m\in[M],k\in\N_0,t\geq 0.
\end{equation*}
\end{claim}
\begin{claimproof}
Fix any $h\in[H]$. Let $Z^N_j(t)$ (resp. $Z'^N_j$) be the queue length of server $j$ in system $G^N_h$ (resp. $G'^N_h$).
Note that for each $m\in[M]$, servers in $\mathcal{V}'^N_{h,m}$ are exchangeable since $\Tilde{u}'^N_{i,j}=\mu_{h,m}$ for all $(i,j)\in\Tilde{\mathcal{W}}'^N_h\times\Tilde{\mathcal{V}}'^N_{h,m}$. Hence, when~\eqref{eq:5.1-ineq} holds, WLOG, we can assume that for all $j\in\mathcal{V}^N_h$,  $Z^N_j(t)\leq Z'^N_j(t)$. Hence, it is sufficient to show that for all $j\in\mathcal{V}^N_h$, $t\geq 0$,  $Z^N_j(t)\leq Z'^N_j(t)$ if $Z^N_j(0)\leq Z'^N_j(0)$.
We will use induction on event times to prove it.
The two systems $\Tilde{G}^N_h$ and $\Tilde{G}'^N_h$ are coupled as follows:
\vspace{.1cm}

\noindent
\textit{Arrival}. Since both systems $\Tilde{G}^N_h$ and $\Tilde{G}'^N_h$ have the same set of dispatchers, 
we synchronize the arrival epochs at each dispatcher. 
Now suppose that for all $j\in\mathcal{V}^N_h$, $t\geq 0$,  $Z^N_j(t-)\leq Z'^N_j(t-)$ holds before an arrival at time $t$. 
We will show that the inequality holds after the arrival is routed to a server. 
Based on the states of the systems $\Tilde{G}^N_h$ and $\Tilde{G}'^N_h$, we consider three cases below:
\begin{itemize}
    \item Case 1: All servers in both systems are busy. In $\Tilde{G}^N_h$, we can select a server $j^{*}$ uniformly at random from $\Tilde{\mathcal{V}}^N_h$ and assign the arrival to server $j^{*}$. Since $\Tilde{\mathcal{V}}^N_h=\Tilde{\mathcal{V}}'^N_h$, we can assign the arrival in $\Tilde{G}'^N_h$ to server $j^{*}$ as well.
    \item Case 2: The system $\Tilde{G}^N_h$ has idle servers but all servers in $\Tilde{G}'^N_h$ is busy. Two systems will independently assign tasks under JIQ policy. 
    The system $\Tilde{G}^N_h$ assigns the task to an idle server $j_1$, so $Z^N_{j_1}(t)=Z^N_{j_1}(t-)+1\leq Z'^N_{j_1}(t-)=Z'^N_{j_1}(t)$. For the system $\Tilde{G}'^N_h$ routes the task to a busy server $j_2$, so $Z^N_{j_2}(t)=Z^N_{j_2}(t-)\leq Z'^N_{j_2}(t-)+1=Z'^N_{j_2}(t)$.
    \item Case 3: Both systems have idle servers. Clearly, the set of idle servers in $\Tilde{G}'^N_h$ is a subset of that in $\Tilde{G}^N_h$. Hence, we can select a server $j^*$ uniformly at random from the set of idle servers in $\Tilde{G}^N_h$ and assign the arrival to server $j^*$. In the system $\Tilde{G}'^N_h$, if server $j^*$ is idle, then we assign the arrival to it as well; otherwise, we will select a server $j'$ uniformly at random from the set of idle servers in $\Tilde{G}'^N_h$ and assign the arrival to server $j'$, implying $Z^N_{j^*}(t)=1\leq Z'^N_{j^*}(t-)=Z'^N_{j^*}(t)$ and $Z^N_{j'}(t)=0<1=Z'^N_{j'}(t)$.
\end{itemize}
Therefore, for all $j\in\mathcal{V}^N_h$,  $Z^N_j(t)\leq Z'^N_j(t)$ holds for all three cases.
\vspace{.2cm}

\noindent
\textit{Departure}. Suppose that for all $j\in\mathcal{V}^N_h$,  $Z^N_j(t-)\leq Z'^N_j(t-)$ holds before a departure clock rings at time $t>0$. By the definition of function $f'(\cdot,\cdot)$, we have that $\Tilde{u}^N_{i,j}\geq \mu_{h,m}$ for all $(i,j)\in \Tilde{\mathcal{W}}^N_h\times\Tilde{\mathcal{V}}^N_{h,m}$. We couple the departure clock of each server $j$ in systems $\Tilde{G}^N_h$ and $\Tilde{G}'^N_h$ in the following way:
\begin{itemize}
    \item Case 1: both server $j$ in systems $\Tilde{G}^N_h$ and $\Tilde{G}'^N_h$ are busy. When the server $j\in \Tilde{G}^N_h$ completes a task of type $i\in\Tilde{\mathcal{W}}^N_h$, then there is server in $\Tilde{G}'^N_h$ finishing a task with probability $\frac{\mu_{h,m}}{\Tilde{u}^N_{i,j}}\leq 1$.
    \item Case 2: the server $j$ in system $\Tilde{G}^N$ is idle. The departure clocks of both server $j$ in systems $\Tilde{G}^N$ and $\Tilde{G}'^N$ ring independently.
\end{itemize}
Note this way that after any departure epoch, for all $j\in\mathcal{V}^N_h$,  $Z^N_j(t)\leq Z'^N_j(t)$ still holds.
\end{claimproof}
\noindent
By Claim~\ref{claim:monotonicity} and \eqref{eq:tilde-G'-steady}, for each $h\in[H]$, $\sum_{m\in[M]}\Tilde{X}^N_{h,m,1}(\infty)\leq\sum_{m\in[M]}\Tilde{X}^N_{h,m,1}(\infty)\xrightarrow{N\rightarrow\infty}\sum_{m\in[M]}\frac{\lambda_hp_{h,m}}{\mu_{h,m}}<1$, which implies that under JIQ, the steady probability of assigning tasks to idle servers approaches to 1 as $N\rightarrow\infty$.  
\end{proof}

Before proceeding to prove Theorem~\ref{thm:f-zero-queue}, we need to define a useful metric for comparing the performance of $G^N$ and $G'^N$. 
Recall the weighted queue length $Q^N_j$ defined in~\eqref{defn:weighted-queue-length}.
Consider a subinterval $E\subseteq[0,1)$. 
Denote $\dom(E,N):=\{j\in\mathcal{V}^N:\phi_2(j)\in E\}$. 
Let $$\mathcal{X}_t(E,G^N)=\big(Q^N_j(t):j\in\dom(E,N)\big),$$ where $Q^N_j(t)$ is the weighted queue length of server $j\in\dom(E,N)$ in the system $G^N$ at time $t$. 
Let $$\mathcal{X}^o_t(E,G^N)=\big(Q^N_{(n)}(t),1\leq n\leq |\dom(E,N)|\big)$$ be the ordered sequence of $\mathcal{X}_t(E,G^N)$ from the minimum to the maximum. 
Similarly, for the system $G'^N$, we define $\mathcal{X}_t(E,G'^N)=(Q'^N_j(t):j\in\dom(E,N))$ and $\mathcal{X}^o_t(E,G'^N)=(Q'^N_{(n)}(t),1\leq n\leq |\dom(E,N)|)$. 

Now suppose the queue length processes in $G^N$ and $G'^N$ are defined in a common probability space.
We say that \textit{Property I} holds for the interval $E$ at time $t$ if 
\begin{enumerate}[\normalfont(i)]
\item $\mathcal{X}^o_t(E,G^N)\leq \mathcal{X}^o_t(E,G'^N)$ componentwise;
\item $Q^N_{(1)}(t)=Q'^N_{(1)}(t)=0$;
\item $\forall j\in\dom(E,N)$, the queue length of server $j$ in the system $G^N$ (resp. $G'^N$) is less than or equal to 1; 
\end{enumerate}
When \textit{Property I} holds for intervals $[v_{m-1},v_m)$ for all $m\in[M]$ at time $t_0$, 
we couple the systems $G^N$ and $G'^N$ in the following way: 
\vspace{.2cm}

\noindent
\textit{Arrival}. Since both the systems $G^N$ and $G'^N$ have the same dispatcher set and route tasks under the $\mathbf{p}$-based JIQ, 
we can synchronize their arrival processes: 
if a task arrives at dispatcher $i\in\mathcal{W}^N_h$, then it selects a target server block $\mathcal{V}^N_{m'}$ with the discrete distribution $(p_{h,m},m\in[M])$ and assigns the new task to one of idle servers (say, $j\in\mathcal{V}^N_{m'}$) in $\mathcal{V}^N_m$ uniformly at random.
By synchronization of arrival epochs, for the system $G'^N$, a task arrives at dispatcher $i\in \mathcal{W}^N_h$ and is assigned to one of the idle servers (say, $j'$) in $\mathcal{V}^N_{m'}$ as well. 
Note that the servers $j$ and $j'$ can be different but both servers are idle.
\vspace{.2cm}

\noindent
\textit{Departure}. Consider $[v_{m-1},v_m)$ for any fixed $m\in[M]$. We synchronize the departure epochs of the $n$-th ordered servers in $\dom([v_{m-1},v_m),N)$ of $G^N$ and $G'^N$ in the following way. 
\begin{itemize}
    \item Case 1: Both the $n$-th ordered servers are busy and the departure clock  for both will ring at rate $1/Q^N_{(n)}(t_0)$. 
    Also, when the clock rings, the $n$-th ordered server in $G^N$  finishes its task but the $j$th ordered server in $G'^N$ will finish its task with probability $Q^N_{(n)}(t_0)/Q'^N_{(n)}(t_0)$ independent of everything else.

    \item Case 2: Only the $n$-th ordered server in $\dom([v_{m-1},v_m),N)$ of the system $G'^N$ is busy and the departure clock for $n$-th servers rings at rate $1/Q'^N_{(n)}(t_0)$.  
    When the clock rings, the $n$-th ordered server in $G'^N$ will finish its task.

    \item Case 3: Both the $n$-th ordered servers are idle. There is no departure from these servers. 
\end{itemize}
Note that due to \textit{Property I}, there is no case that only $n$-th the ordered server in $\dom([v_{m-1},v_m),N)$ of system $G^N$ is busy.
\vspace{.2cm}

\begin{lemma}\label{lem:GN-G'N}
Assume that \textit{Property I} holds for all intervals $[v_{m-1},v_m)$, $m\in[M]$ at time $t_0$ and 
that the systems $G^N$ and $G'^N$ are coupled as above. 
Let $t_1$ be the time for the next event time. Then,
\begin{equation}\label{eq:order-weighted-queue}
    \mathcal{X}^o_{t_1}([v_{m-1},v_m),G^N)\leq \mathcal{X}^o_{t_1}([v_{m-1},v_m),G'^N),\quad \forall m\in[M].
\end{equation}
\end{lemma}
\begin{proof}
First, let $t_1$ be an arrival epoch and suppose that dispatcher $i\in\mathcal{W}^N_h$ (resp.~$i\in \mathcal{W}'^N_h$) assigns the new task to a server $j\in\mathcal{V}^N_{m'}$ (resp.~$j'\in\mathcal{V}'^N_{m'}$) in system $G^N$ (resp.~$G'^N$). Since the new arrival does not affect the state of servers in $\mathcal{V}^N\setminus\mathcal{V}^N_{m'}$ (resp.~$\mathcal{V}'^N\setminus\mathcal{V}'^N_{m'}$), then at time $t_1$, $\forall m\in[M]\setminus \{m'\}$, $$\mathcal{X}^o_{t_1}([v_{m-1},v_m),G^N)=\mathcal{X}^o_{t_0}([v_{m-1},v_m),G^N)\leq \mathcal{X}^o_{t_0}([v_{m-1},v_m),G'^N)=\mathcal{X}^o_{t_1}([v_{m-1},v_m),G'^N).$$
    Consider $\mathcal{X}^o_{t_1}([v_{m'-1},v_{m'}),G^N)=(Q^N_{(n)}(t_1),1\leq n\leq |\dom([v_{m'-1},v_{m'}),N)|)$. Since the new task is assigned to the server $j$ which is idle at time $t_0$, we have $Q^N_j(t_1)=1/f(\phi(i),\phi(j))$. Also, $Q^N_{\hat{j}}(t_1)=Q^N_{\hat{j}}(t_0)$ for all ${\hat{j}}\in\dom([v_{m-1},v_m),N)\setminus\{j\}$. Hence, there exists $l\in[|\dom([v_{m'-1},v_{m'}),N)|]$ such that
    \begin{enumerate}[\normalfont(i)]
        \item $Q^N_{(l)}(t_1)=\frac{1}{f(\phi(i),\phi(j'))}$,
        \item $Q^N_{(n)}(t_1)=Q^N_{(n+1)}(t_0)$, $\forall n\in[l-1]$,
        \item $Q^N_{(n)}(t_1)=Q^N_{(n)}(t_0)$, $\forall l+1\leq n\leq |\dom([v_{m'-1},v_{m'}),N)|$, if $l\in[|\dom([v_{m'-1},v_{m'}),N)|-1]$,
        \item $Q^N_{(l)}(t_1)<Q^N_{(l+1)}(t_1)$ if $l\in[|\dom([v_{m'-1},v_{m'}),N)|-1]$.
    \end{enumerate}
    Consider $\mathcal{X}^o_{t_1}([v_{m'-1},v_{m'}),G'^N)=(Q'^N_{(n)}(t_1),1\leq n\leq |\dom([v_{m'-1},v_{m'}),N)|)$. Similarly, there exists $l'\in[|\dom([v_{m'-1},v_{m'}),N)|]$ such that 
    \begin{enumerate}[\normalfont(i)]
        \item $Q'^N_{(l')}(t_1)=\frac{1}{\mu_{h,m'}}$,
        \item $Q'^N_{(n)}(t_1)=Q'^N_{(n+1)}(t_0)$, $\forall n\in[l'-1]$,
        \item $Q'^N_{(n)}(t_1)=Q'^N_{(n)}(t_0)$, $\forall l'+1\leq n\leq |\dom([v_{m'-1},v_{m'}),N)|$, if $l'\in[|\dom([v_{m'-1},v_{m'}),N)|-1]$,
        \item $Q'^N_{(l')}(t_1)<Q'^N_{(l'+1)}(t_1)$ if $l'\in[|\dom([v_{m'-1},v_{m'}),N)|-1]$.
    \end{enumerate}
    According to the values of $l$ and $l'$, we discuss the relationship between $ \mathcal{X}^o_{t_1}([v_{m'-1},v_{m'}),G^N)$ and $ \mathcal{X}^o_{t_1}([v_{m'-1},v_{m'}),G'^N)$ in the following three cases.\\
    
    \noindent
    Case 1: $1\leq l< l'\leq  |\dom([v_{m'-1},v_{m'}),N)|$. 
     \begin{enumerate}[\normalfont(i)]
         \item $Q^N_{(n)}(t_1)=Q^N_{(n+1)}(t_0)\leq Q'^N_{(n+1)}(t_0)=Q'^N_{(n)}(t_1)$, $\forall n\in[l-1]$;
         \item $Q^N_{(l)}(t_1)<Q^N_{(l+1)}(t_1)=Q^N_{(l+1)}(t_0)\leq Q'^N_{(l+1)}(t_0)=Q'^N_{(l)}(t_1)$, $n=l$;
         \item $Q^N_{(n)}(t_1)=Q^N_{(n)}(t_0)\leq Q'^N_{(n)}(t_0)\leq Q'^N_{(n+1)}(t_0)=Q'^N_{(n)}(t_1)$, $\forall l+1\leq n\leq l'-1$;
         \item $Q^N_{(l')}(t_1)=Q^N_{(l')}(t_0)\leq Q'^N_{(l')}(t_0)=Q'^N_{(l'-1)}(t_1)\leq Q'^N_{(l')}(t_1)$,  $n=l'$;
         \item $Q^N_{(n)}(t_1)=Q^N_{(n)}(t_0)\leq Q'^N_{(n)}(t_0)=Q'^N_{(n)}(t_1)$, $\forall l'+1\leq n\leq |\dom([v_{m'-1},v_{m'}),N)|$.
     \end{enumerate} 
     Case 2: $1\leq l'\leq l\leq  |\dom([v_{m'-1},v_{m'}),N)|$.
     \begin{enumerate}[\normalfont(i)]
         \item $Q^N_{(n)}(t_1)=Q^N_{(n+1)}(t_0)\leq Q'^N_{(n+1)}(t_0)=Q'^N_{(n)}(t_1)$, $\forall n\in[l'-1]$;
         \item $Q^N_{(n)}(t_1)\leq Q^N_{(l)}(t_1)= 1/f(\phi(i),\phi(j'))\leq 1/\mu_{h,m'}=Q'^N_{(l')}(t_1)\leq  Q'^N_{(n)}(t_1)$, $\forall l'\leq n\leq l'$;
         \item $Q^N_{(n)}(t_1)=Q^N_{(n)}(t_0)\leq Q'^N_{(n)}(t_0)=Q'^N_{(n)}(t_1)$, $\forall l'+1\leq n\leq |\dom([v_{m'-1},v_{m'}),N)|$.
     \end{enumerate} 
     Therefore, \eqref{eq:order-weighted-queue} is preserved at  $t_1$.
    \vspace{3mm}
    
    \noindent
Next, let $t_1$ be a departure epoch and suppose that the clock for $l$-th ordered server in $\dom([v_{m-1},v_m),N)$ rings. 
According to whether there are tasks being finished in system $G^N$ and $G'^N$, we discuss the relationship between $ \mathcal{X}^o_{t_1}([v_{m'-1},v_{m'}),G^N)$ and $ \mathcal{X}^o_{t_1}([v_{m'-1},v_{m'}),G'^N)$ in the following three cases.

\noindent
Case 1: both the $l$-th ordered servers in $\dom([v_{m-1},v_m),N)$ of systems $G^N$ and $G'^N$ finish tasks, and then \begin{enumerate}[\normalfont(i)]
     \item $Q^N_{(1)}(t_1)=Q'^N_{(1)}(t_1)=0$;
     \item $Q^N_{(n)}(t_1)=Q^N_{(n-1)}(t_0)\leq Q'^N_{(n-1)}(t_0)=Q'^N_{(n)}(t_1)$, $\forall 2\leq n\leq l$;
     \item $Q^N_{(n)}(t_1)=Q^N_{(n)}(t_0)\leq Q'^N_{(n)}(t_0)=Q'^N_{(n)}(t_1)$, $\forall l+1\leq n\leq |\dom([v_{m'-1},v_{m'}),N)|$.
 \end{enumerate} 
 Case 2: only the $l$-th ordered servers in  $\dom([v_{m-1},v_m),N)$ of system $G^N$ finishes task, and then
 \begin{enumerate}[\normalfont(i)]
     \item $Q^N_{(1)}(t_1)=Q'^N_{(1)}(t_1)=0$;
     \item $Q^N_{(n)}(t_1)=Q^N_{(n-1)}(t_0)\leq Q'^N_{(n-1)}(t_0)=Q'^N_{(n-1)}(t_1)\leq Q'^N_{(n)}(t_1)$, $\forall 2\leq n\leq l$;
     \item $Q^N_{(n)}(t_1)=Q^N_{(n)}(t_0)\leq Q'^N_{(n)}(t_0)=Q'^N_{(n)}(t_1)$, $\forall l+1\leq n\leq |\dom([v_{m'-1},v_{m'}),N)|$.
 \end{enumerate}
 Case 3: only the $l$-th ordered servers in  $\dom([v_{m-1},v_m),N)$ of system $G'^N$ finishes task, and then
 \begin{enumerate}[\normalfont(i)]
     \item $Q^N_{(1)}(t_1)=Q'^N_{(1)}(t_1)=0$;
     \item $0=Q^N_{(n)}(t_1)\leq Q'^N_{(n-1)}(t_0)=Q'^N_{(n)}(t_1)$, $\forall 2\leq n\leq l$;
     \item $Q^N_{(n)}(t_1)=Q^N_{(n)}(t_0)\leq Q'^N_{(n)}(t_0)=Q'^N_{(n)}(t_1)$, $\forall l+1\leq n\leq |\dom([v_{m'-1},v_{m'}),N)|$.
 \end{enumerate}
Thus, \eqref{eq:order-weighted-queue} is preserved at $t_1$.
\end{proof}

\begin{proof}[Proof of Theorem~\ref{thm:f-zero-queue}]
Fix any $T>0$ and large enough $N$. Consider the system $G^N$ and the system $G'^N$ both with the idle initial state. Clearly, \textit{Property I} holds for all intervals $[v_{m-1},v_m)$, $m\in[M]$ at time $0$ so we can couple the system $G^N$ with the system $G'^N$ as above at time $t_0=0$. Let $t_1>t_0$ be the next event time. 
By Lemma~\ref{lem:GN-G'N}, $\mathcal{X}^o_{t_1}([v_{m-1},v_m),G^N)\leq \mathcal{X}^o_{t_1}([v_{m-1},v_m),G'^N), \forall m\in[M]$. 
Let $\Bar{X}'^N_{m,k}(t)$ be the fraction of servers of type $m$ serving a task at rate $\mu_k$ in the system $G'^N$ at time $t$. 
If at time $t_1$, for all $m\in[M]$, $\sum_{k\in[K]}\Bar{X}'^N_{m,k}(t)\leq  \sum_{k\in[K]}x^{\mathbf{p}}_{m,k}<1$, then \textit{Property I} holds for all intervals $[v_{m-1},v_m)$, $m\in[M]$ at time $t_1$ so we can still couple the system $G^N$ with the system $G'^N$ as above after time $t_1$. 
Hence, we can couple the system $G^N$ with the system $G'^N$ as above for all $t\in[0,T\land \tau^N)$, where $$\tau^N\coloneqq\inf \Big\{t\geq 0: \exists\ m\in[M], \quad \sum_{k\in[K]}\Bar{X}'^N_{m,k}(t)> \sum_{k\in[K]}x^{\mathbf{p}}_{m,k}\Big\}.$$ 
For $t\in[T\land \tau^N,T]$, the system $G^N$ and the system $G'^N$ are assumed to evolve independently. Let $\Bar{X}^N_m(t)$ be the busy servers in $\dom([v_{m-1},v_m),N)$ in system $G^N$, $m\in[M]$.
\begin{equation*}
    \begin{split}
        \lim_{N\rightarrow\infty}\PP(\sup_{t\in[0,T]}\Bar{X}^N(t)\leq \sum_{m,k}x^{\mathbf{p}}_{m,k})\leq & \lim_{N\rightarrow\infty}\PP(\sup_{t\in[0,T]}\Bar{X}^N_m(t)\leq \sum_{k}x^{\mathbf{p}}_{m,k},\quad \forall m\in[M])\\
        =&\lim_{N\rightarrow\infty}\PP(\sup_{t\in[0,T]}\Bar{X}^N_m(t)\leq \sum_{k}x^{\mathbf{p}}_{m,k},\quad \forall m\in[M],\tau^N\geq T)\\
        &\hspace{1cm}+\lim_{N\rightarrow\infty}\PP(\sup_{t\in[0,T]}\Bar{X}^N_m(t)\leq \sum_{k}x^{\mathbf{p}}_{m,k},\quad \forall m\in[M],\tau^N<T)\\
    \end{split}
\end{equation*}
Since for all $t\in[0,T\land \tau^N]$, \textit{Property I} holds for all intervals $[v_{m-1},v_m)$, $m\in[M]$, then 
$$\PP(\sup_{t\in[0,T]}\Bar{X}^N_m(t)\leq \sum_{k}x^{\mathbf{p}}_{m,k},\quad \forall m\in[M],\tau^N\geq T)=\PP(\tau^N\geq T).$$
By Theorem~\ref{thm:fluid-limit-1}, $\lim_{N\rightarrow\infty}\PP(\tau^N\geq T)=1$. Therefore, 
$$\lim_{N\rightarrow\infty}\PP(\sup_{t\in[0,T]}\Bar{X}^N(t)\leq \sum_{m,k}x^{\mathbf{p}}_{m,k})=1$$
\end{proof}

\section{Numerical Results}\label{sec:numerical results}

In this section, we will present the simulation results to support the discussion in Section~\ref{sec:intro} and validate the theoretical results. 
For the simulation, we generate the heterogeneous systems as follows:
\begin{enumerate}[Step 1.]
    \item Decide the size of the system $N$ and set $\xi=1$, where $\xi$ is defined in Assumption~\ref{ass:f-sequence}. 
    \item Sample two vectors $\mathcal{I},\mathcal{J}\in[0,1)^N$ as the membership of dispatchers and servers, respectively.
    \item Decide the arrival rate of tasks at each dispatcher via the arrival rate function $\lambda(\cdot)$, where for all $x\in\mathcal{I}$, $\lambda(x)=5x.$
    \item Decide the service rate for each pair $(x,y)\in\mathcal{I}\times\mathcal{J}$ via the service rate function $f(\cdot,\cdot)$, where $f(\cdot,\cdot)$ is a stepwise function on $[0,1)^2$ represented as Figure~\ref{fig:g-function}.

    \begin{figure}[!htb]
    \minipage{0.5\textwidth}\centering
  \includegraphics[width=0.9\linewidth]{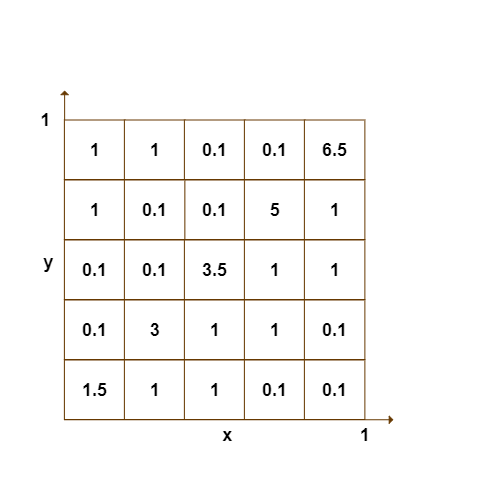}
  \caption{Function $f(x,y)$}\label{fig:g-function}
\endminipage\hfill
\minipage{0.5\textwidth}\centering
  \includegraphics[width=0.9\linewidth]{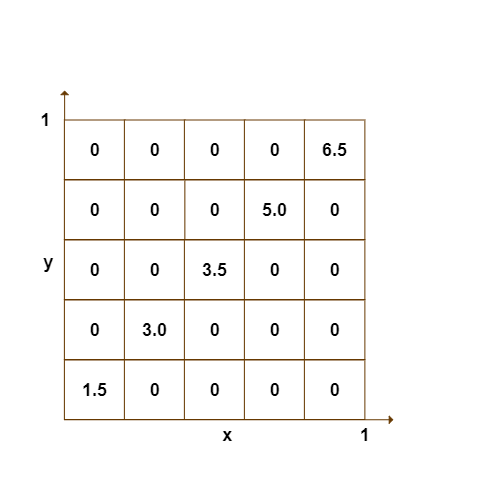}
  \caption{Function $f'(x,y)$}\label{fig:g'-function}
\endminipage\hfill
\end{figure}
    
\end{enumerate}
For both ICRD and SPD approaches, we set the partitions \begin{equation*}
    \begin{split}
        (\mathbf{w},\mathbf{v})=(&w_0=0<w_1=0.2<w_2=0.4<w_3=0.6<w_4=0.8<w_5=1, \\&v_0=0<v_1=0.2<v_2=0.4<v_3=0.6<v_4=0.8<v_5=1)
    \end{split}
\end{equation*} and the matrix $\mathbf{p}=\big(p_{h,m},h\in[5],m\in[5]\big)=\begin{pmatrix}
        1 & 0 & 0 & 0 & 0\\     
        0 & 1 & 0 & 0 & 0\\  
        0 & 0 & 1 & 0 & 0\\   
        0 & 0 & 0 & 1 & 0\\     
        0 & 0 & 0 & 0 & 1
    \end{pmatrix}$. It is easy to check that the partitions $(\mathbf{w},\mathbf{v})$ and the matrix $\mathbf{p}$ satisfy \eqref{eq:P-subcritical-infinite}. Interestingly, with such $(\mathbf{w},\mathbf{v},\mathbf{p})$, the ICRD and SPD approaches are equivalent, since SPD assigns tasks of different classes to servers of different types with probability 1, that is, one type servers are reserved for only one class tasks as ICRD does. For ICRD, we need to construct the subsystem by reserving the capacity of each server, which is equivalent to deciding the service rate of each pair $(x,y)\in\mathcal{I}\times\mathcal{J}$ via function $f'(\cdot,\cdot)$ (Figure~\ref{fig:g'-function}). 
The simulation includes three parts: (i) comparing the performance of ICRD (SPD) with other routing policies, (ii) validating the asymptotic zero-queueing property of ICRD (SPD), and (iii) numerically verifying the convergence of steady states.
\vspace{.2cm}

\noindent
\textbf{Compare ICRD with other policies.} Besides ICRD (SPD), we simulate JIQ, JFIQ, JFSQ and \textsc{MinDrift} (Definition~\ref{def:MinDrift}) policies in a heterogeneous system with $N=50$.
\begin{figure}[h]
    \centering
    \includegraphics[scale=0.4]{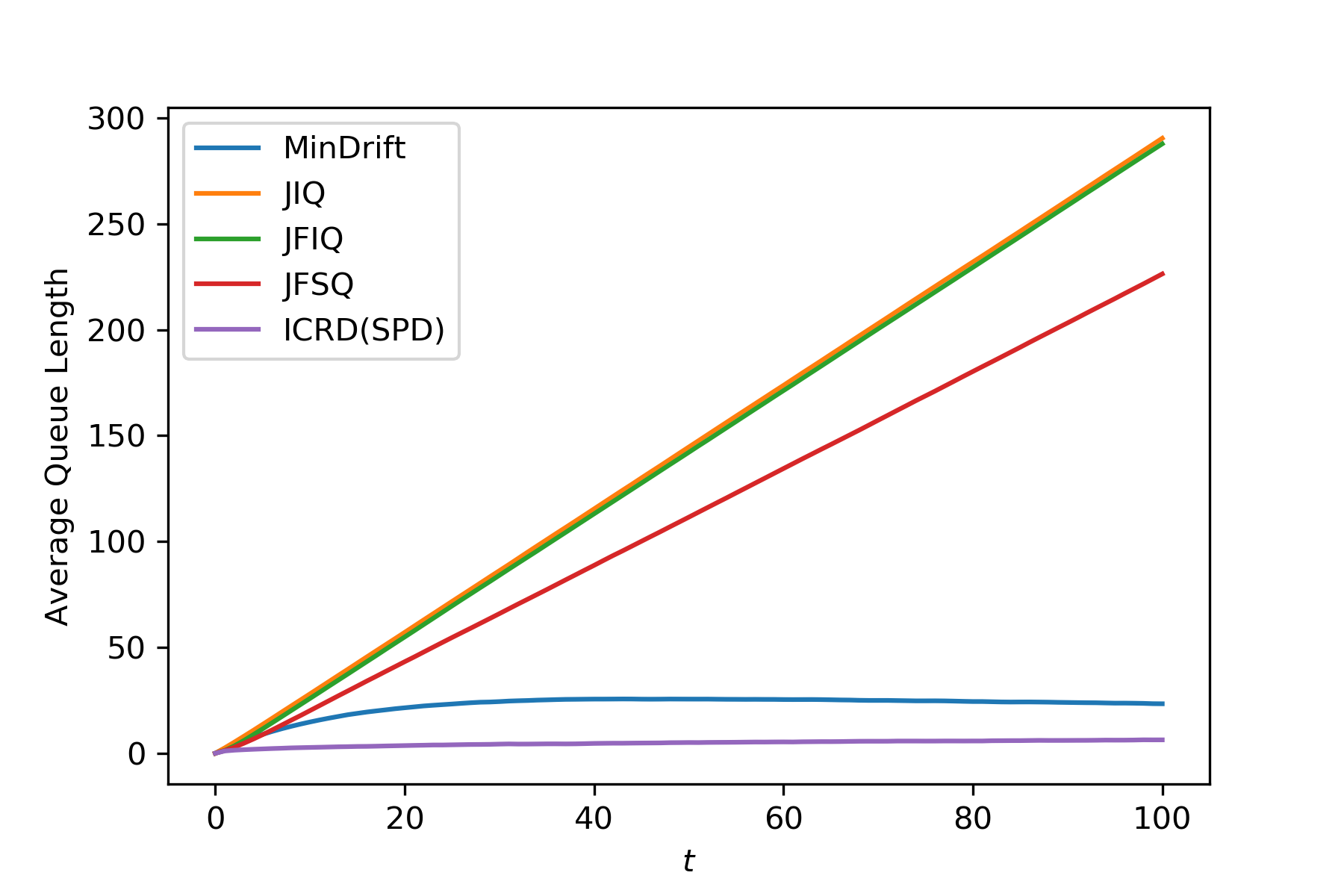}
    \caption{Performance of various policies}
    \label{fig:my_label}
\end{figure}
The simulation results shows that under JIQ, JFIQ and JFSQ, the average queue length increases linearly, leading to the instability of the system. The main reason is that these three policies assign tasks to ``bad'' servers, which will process at rate less than 0.5, with high probability; see details in Table~\ref{table:prob-bad-server}.
\begin{table}[h]
\centering
\begin{tabular}{ |c|c|c|c|c| } 
 \hline
 JIQ & JFIQ & JFSQ & \textsc{MinDrift} & ICRD (SPD) \\ 
 \hline
  0.3991 & 0.3963 & 0.0990 & 0.0039 & 0 \\ 
 \hline
\end{tabular}
\caption{Probability of assigning tasks to ``bad'' servers}
\label{table:prob-bad-server}
\end{table}

It is interesting to see that the performance of ICRD (SPD) is even better than that of \textsc{MinDrift}. This observation is reasonable since based on the definition of \textsc{MinDrift}, it does not differentiate idle servers based on processing rate. Also, in Table~\ref{table:prob-bad-server}, there is a small portion of tasks assigning to ``bad'' servers under \textsc{MinDrift}. ICRD (SPD) will never assign tasks to ``bad'' servers. 
\vspace{.2cm}

\noindent
\textbf{Asymptotic zero-queueing property of ICRD (SPD).}
With the current setting, the system $G^N$ consists of $5$ disjoint systems $G^N_h$. For each $G^N_h$, it has dispatchers $\mathcal{W}^N_h=\{x\in\mathcal{I}: 0.2(h-1)\leq x<0.2h\}$ and servers $\mathcal{V}^N_h=\{y\in\mathcal{J}: 0.2(h-1)\leq x<0.2h\}$. We simulate ICRD (SPD) and record the fraction of busy servers $\Tilde{X}^N_{h}$ in each $\mathcal{V}^N_h$ with $N=100,500,1000,5000$. The simulated sample path starting with the idle initial state is shown in Figure~\ref{fig:sample path}. From the simulation results, we verify that ICRD (SPD) achieves asymptotic zero-queueing, since all $\Tilde{X}^N_h$, $h=1,2,3,4,5$, with large $N$, are less than 1.  
\begin{figure}[!htb]
     \centering
     \begin{subfigure}[b]{0.3\textwidth}
         \centering
         \includegraphics[width=\textwidth]{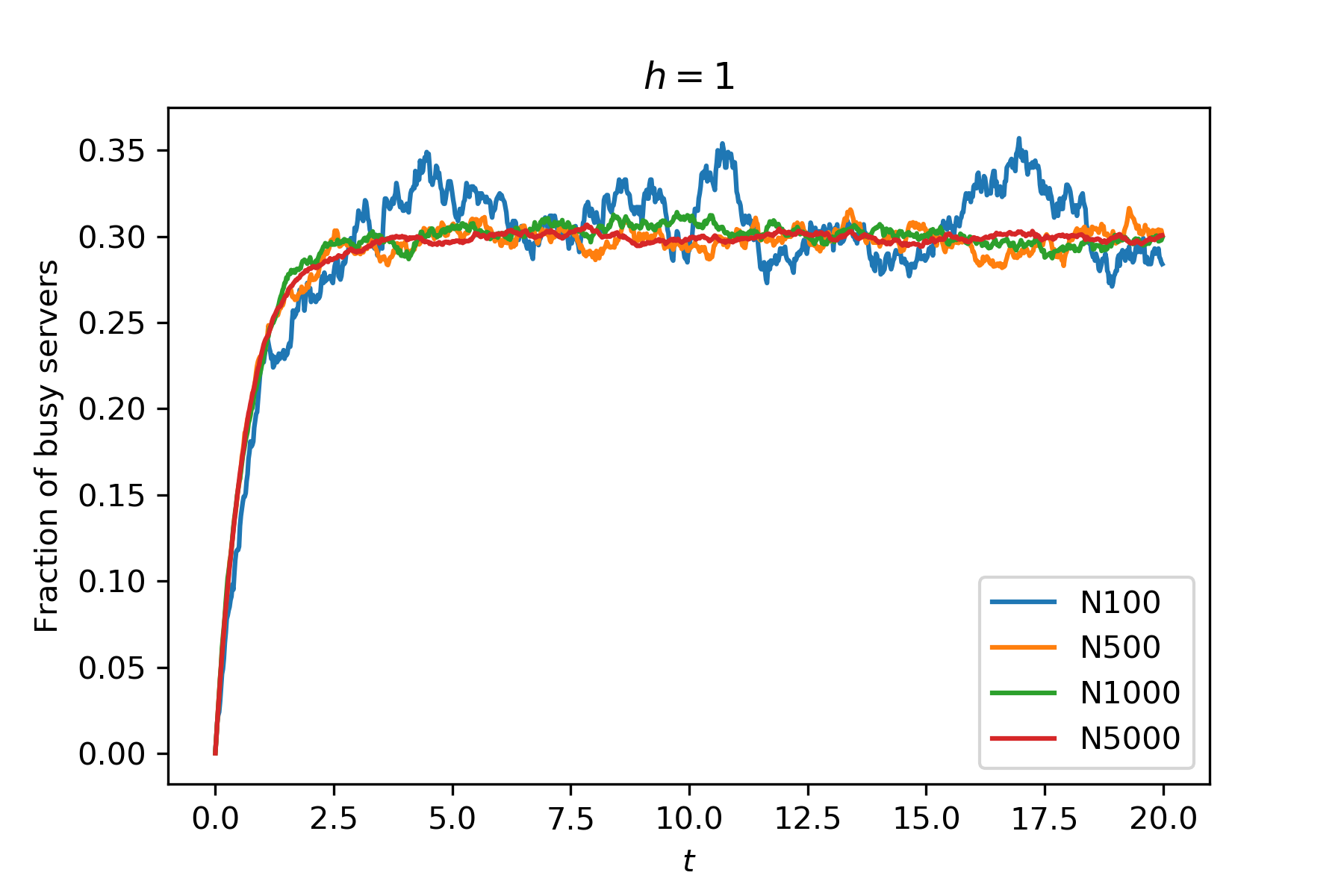}
     \end{subfigure}
     \hfill
     \begin{subfigure}[b]{0.3\textwidth}
         \centering
         \includegraphics[width=\textwidth]{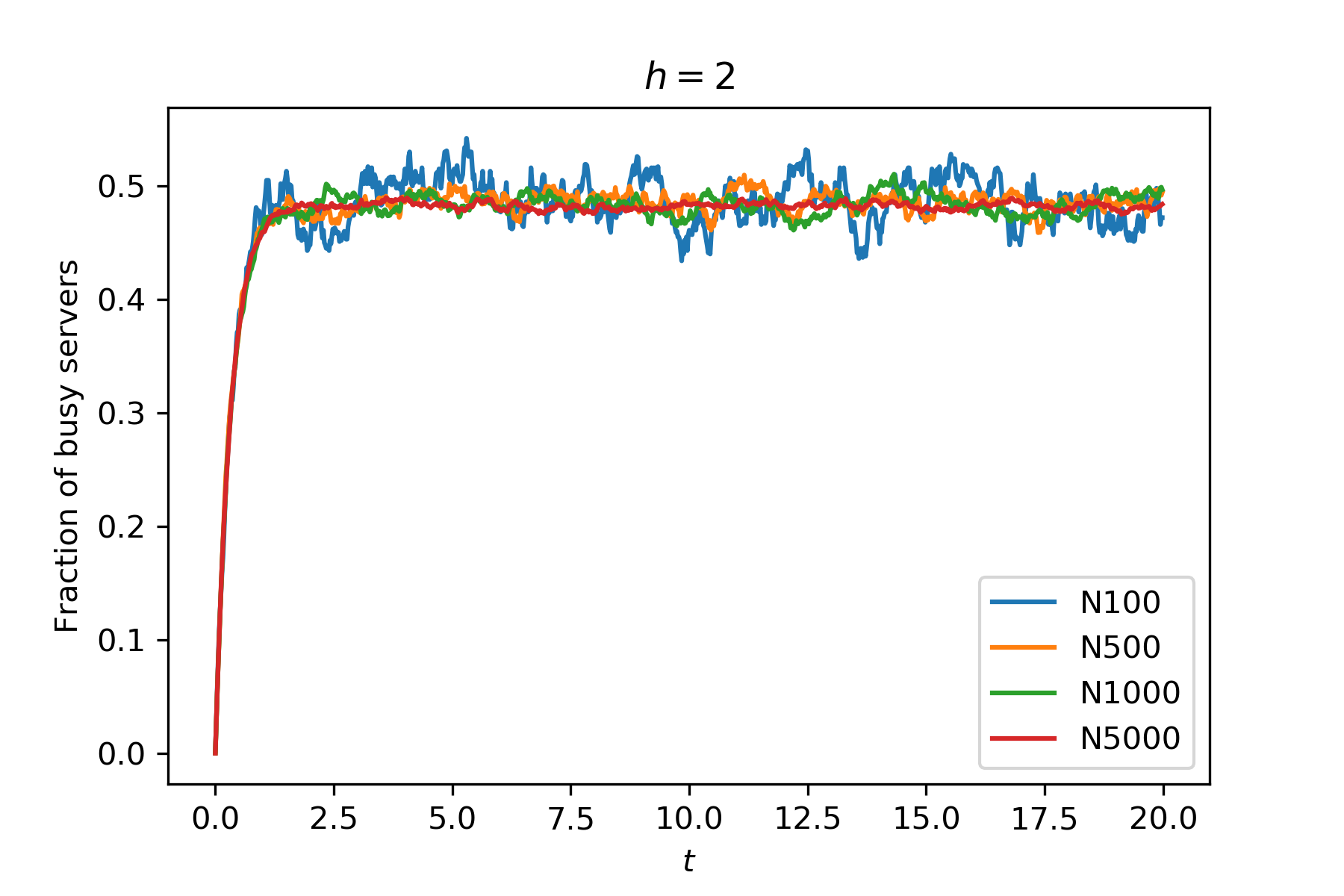}
     \end{subfigure}
     \hfill
     \begin{subfigure}[b]{0.3\textwidth}
         \centering
         \includegraphics[width=\textwidth]{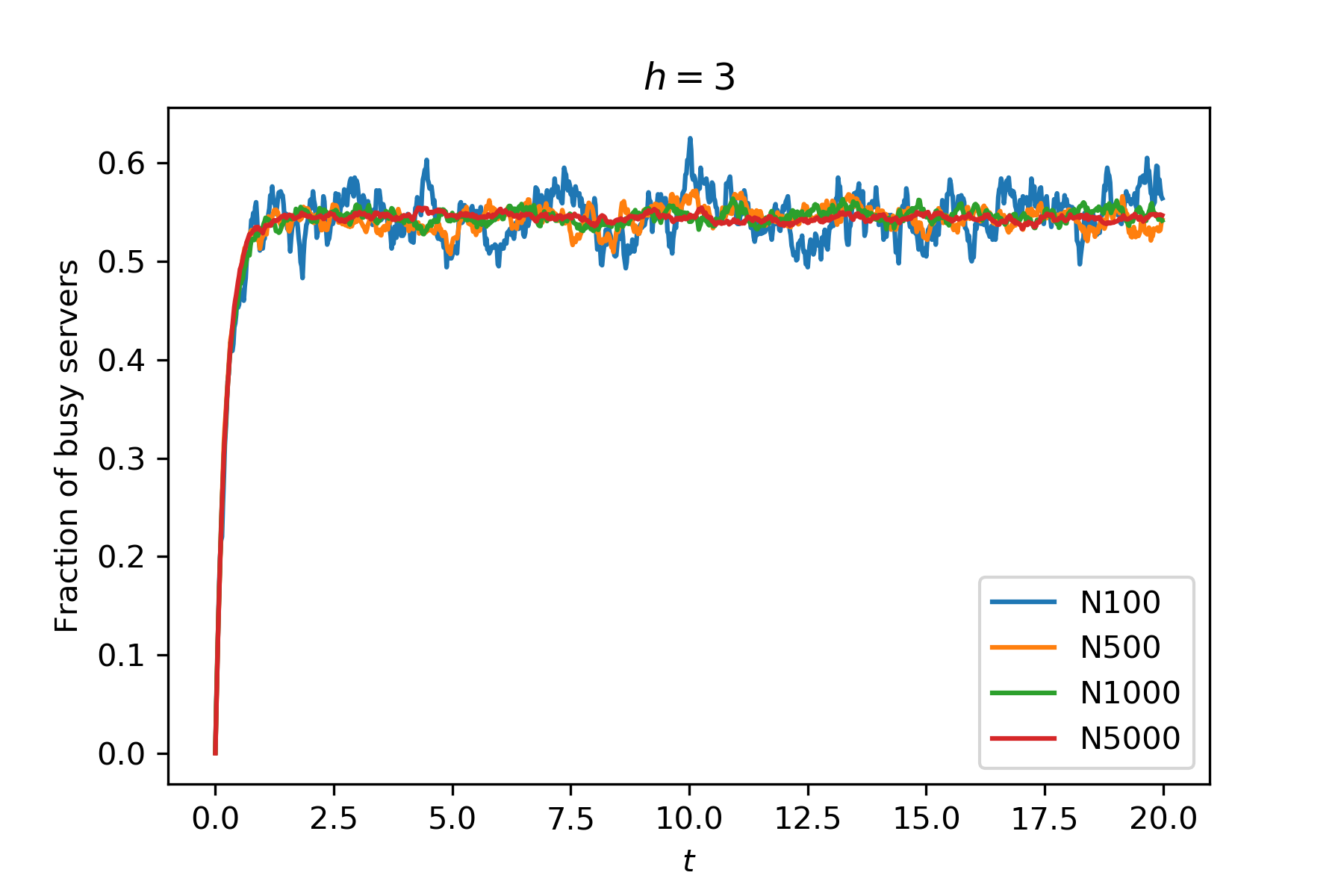}
     \end{subfigure}
     \hfill
     \begin{subfigure}[b]{0.3\textwidth}
         \centering
         \includegraphics[width=\textwidth]{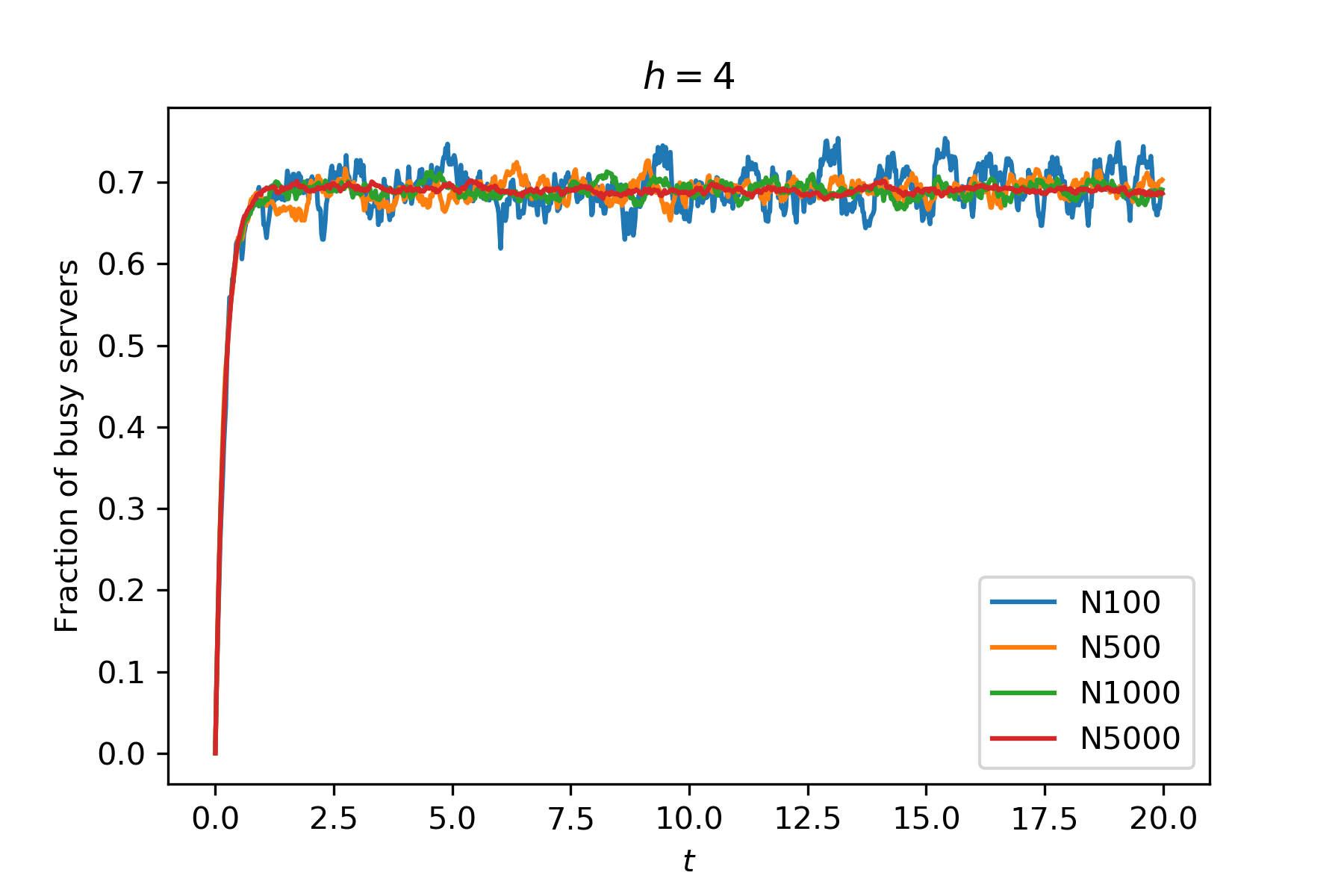}
     \end{subfigure}
     \begin{subfigure}[b]{0.3\textwidth}
         \centering
         \includegraphics[width=\textwidth]{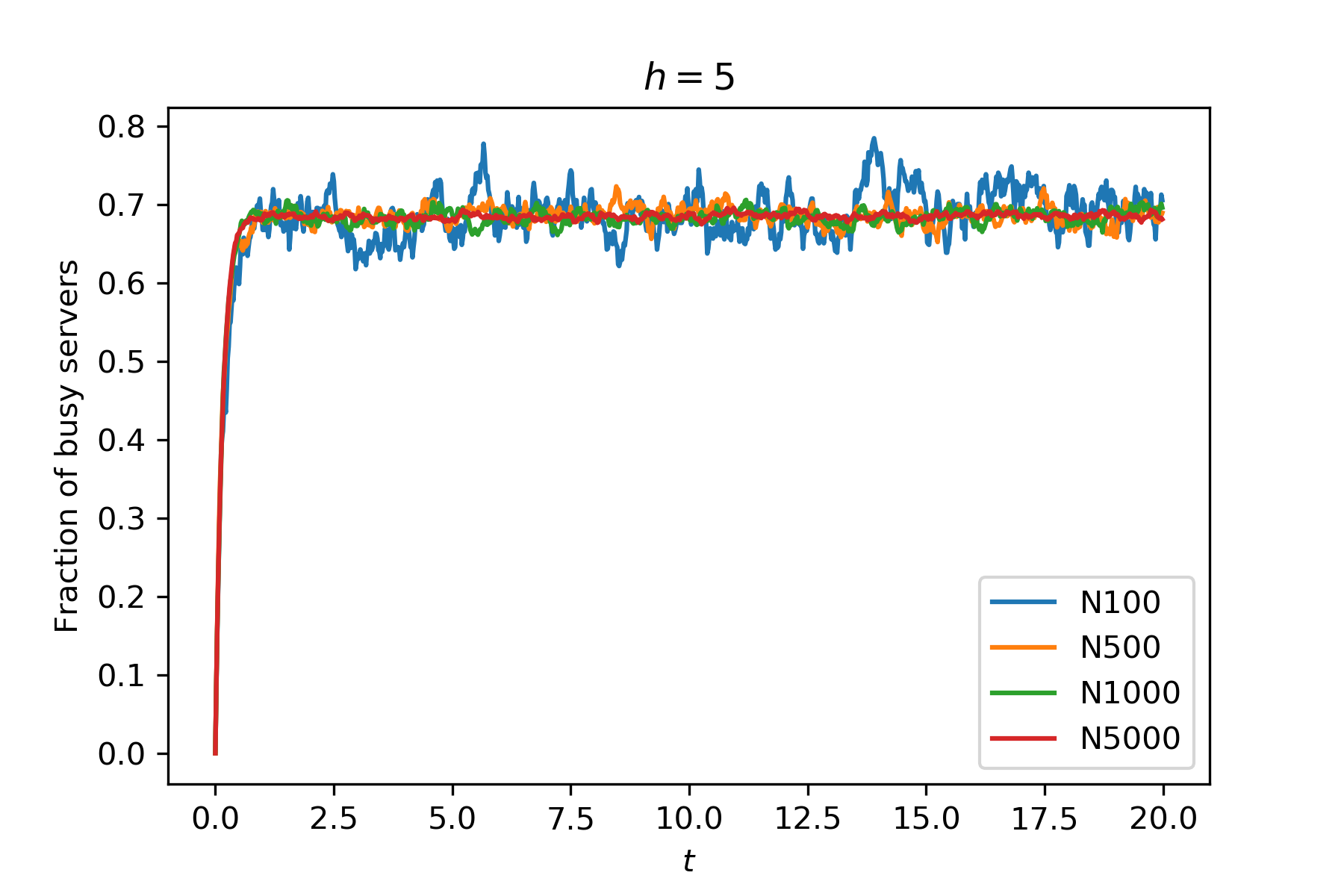}
     \end{subfigure}
     \caption{Sample Path of $\Tilde{X}^N_h$, $h=1,2,3,4,5$}
    \label{fig:sample path}
\end{figure}


\noindent
\textbf{Convergence of steady states.}
We numerically verify the convergence of steady states. Besides the case with idle initial states shown above, we also simulate two extra scenarios with different initial states: one is that all servers have 1 customer in the queue, and the other is that half of servers are idle and the rest half has 1 customer in the queue. The simulation results of both scenarios with $N=100,500,1000,5000$ are provided in~\ref{app:more-numerical}, which shows that the steady state of the $N$-th system converges as $N\rightarrow\infty$ under both scenarios. For comparison of the steady states of the limit systems with three various initial settings (i.e., all-idle, all-busy, and half-busy-half-idle), we draw the sample path of the $N$-th system ($N=5000$) with three various initial states in Figure~\ref{fig:comparison}. We find that when the system size is large, $N=5000$, the sample paths with various initial states become very close after a short time.
\begin{figure}[!htb]
     \centering
     \begin{subfigure}[b]{0.3\textwidth}
         \centering
         \includegraphics[width=\textwidth]{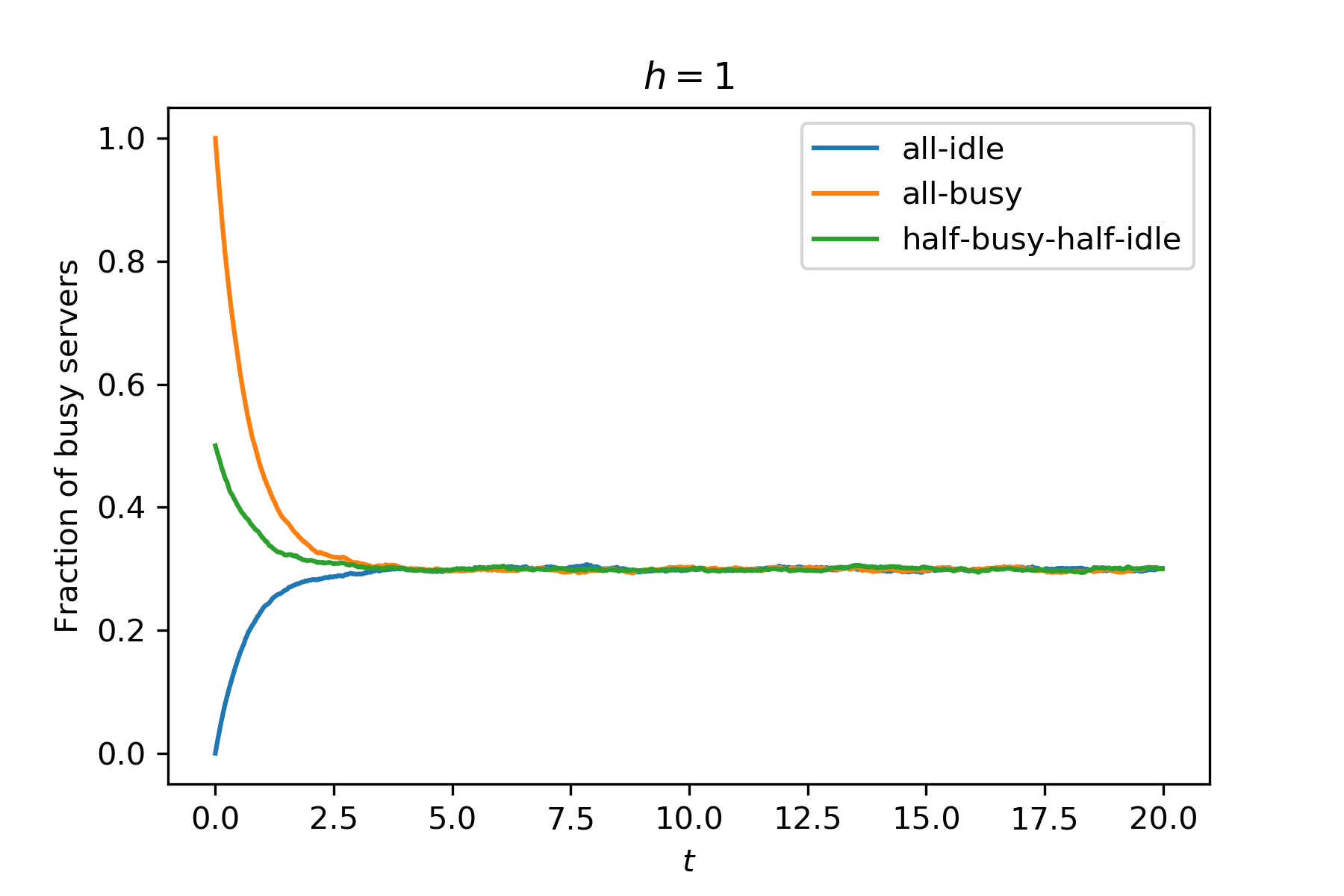}
     \end{subfigure}
     \hfill
     \begin{subfigure}[b]{0.3\textwidth}
         \centering
         \includegraphics[width=\textwidth]{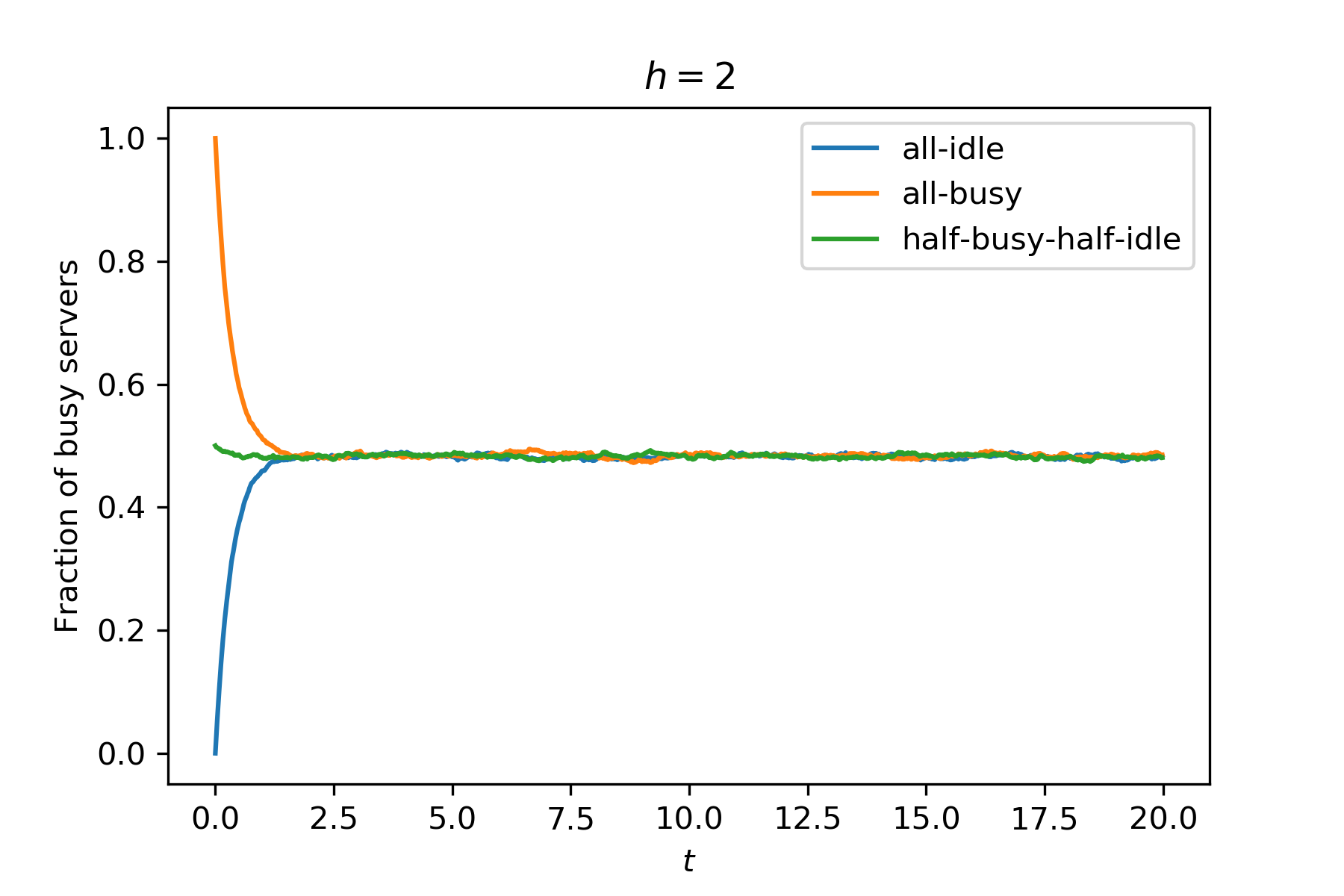}
     \end{subfigure}
     \hfill
     \begin{subfigure}[b]{0.3\textwidth}
         \centering
         \includegraphics[width=\textwidth]{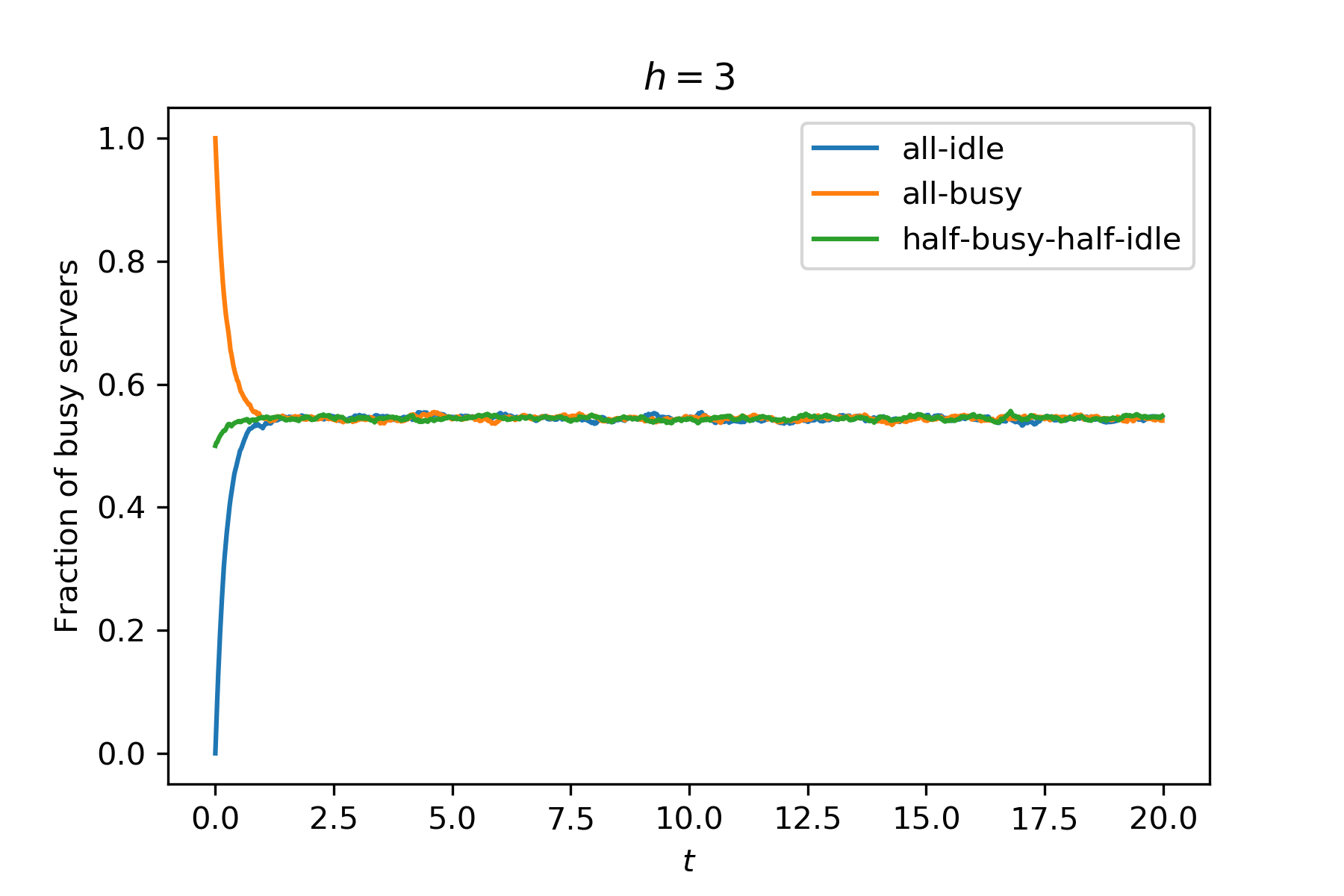}
     \end{subfigure}
     \hfill
     \begin{subfigure}[b]{0.3\textwidth}
         \centering
         \includegraphics[width=\textwidth]{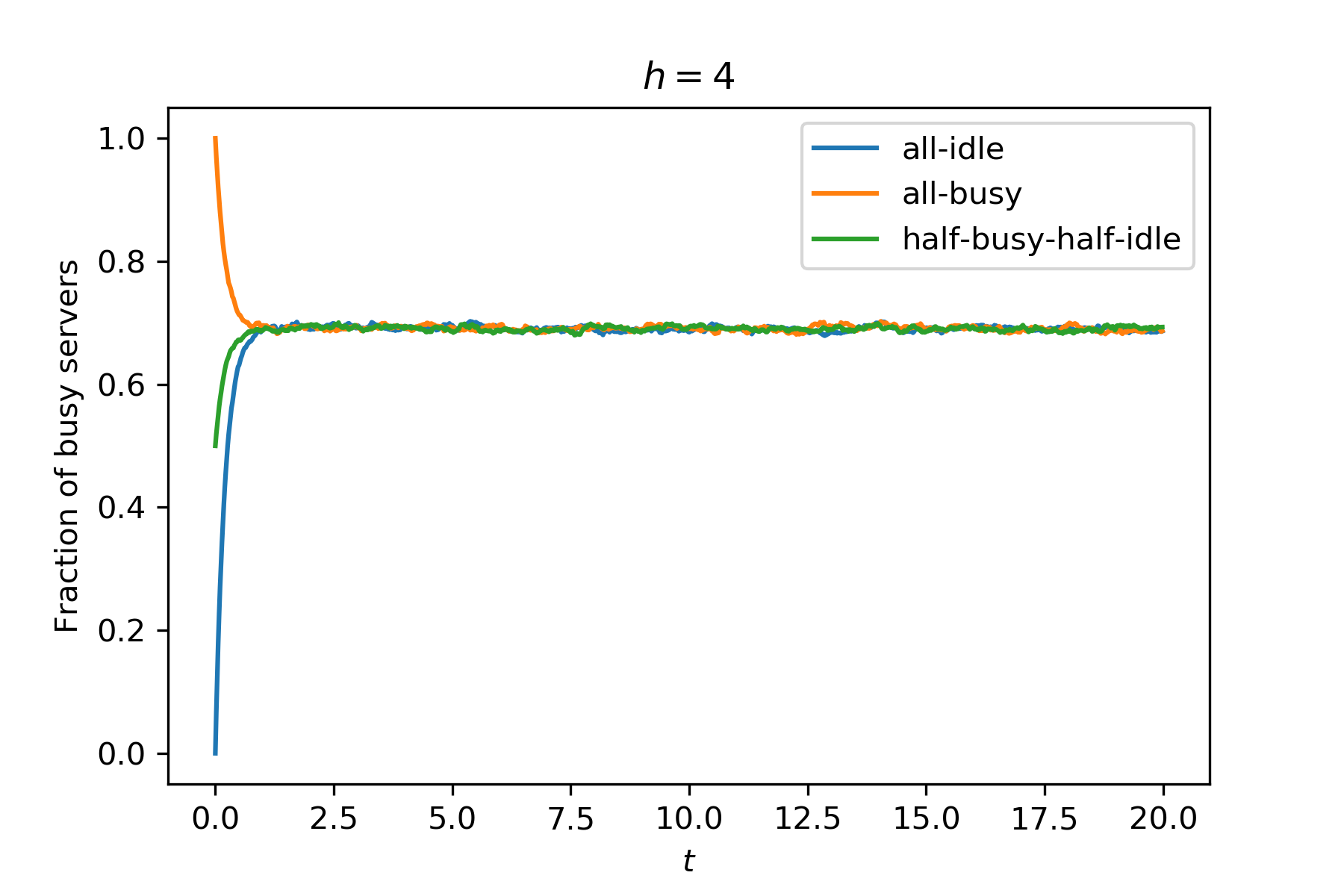}
     \end{subfigure}
     \begin{subfigure}[b]{0.3\textwidth}
         \centering
         \includegraphics[width=\textwidth]{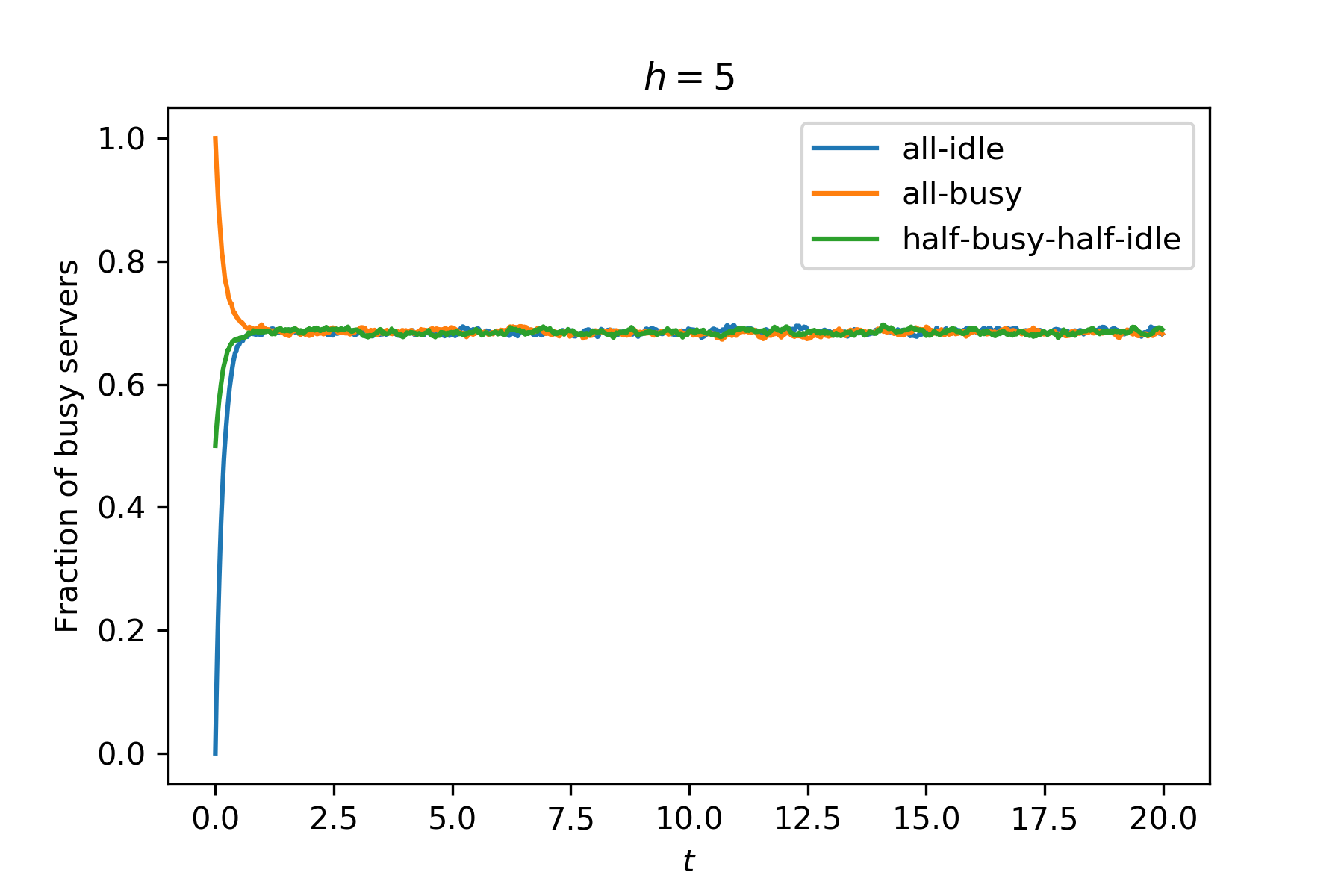}
     \end{subfigure}
     \caption{Sample Path with three various initial states: all-idle, all-busy,half-busy-half-idle}
   \label{fig:comparison}
\end{figure}



\appendix

\section{Fluid limit of \textit{dispatcher-independent} systems from~\cite{AS15}}\label{AS15}

\noindent
\textbf{Model.} The system consists of one dispatcher and $J\geq 1$ server pools $\mathcal{N}_j$, $j\in \mathcal{J}=\{1,...,J\}$. Each pool $\mathcal{N}_j$ has $N_j$ identical servers. Let $\mathcal{N}=\cup_{j}\mathcal{N}_j$ be the set of all servers. The arrival process of tasks at the dispatcher is a Poisson process with rate $\Lambda$, independently of the other processes. Each new task will be assigned to one of servers under JIQ policy. The service time of a task at a server in pool $j\in\mathcal{J}$ is exponentially distributed with mean $1/\mu_j\in(0,\infty)$. 
\vspace{3mm}

\noindent
\textbf{Asymptotic regime.} The total number of servers $N=\sum_{j}|\mathcal{N}_j|$ is the scaling parameter which increases to infinity. Also, the arrival rate and the server pools' sizes increase in proportion to $N$: $\Lambda=N\lambda$, $|\mathcal{N}_j|=N\beta_j$, $j\in\mathcal{J}$ where $\lambda$, $\beta_j$, $j\in\mathcal{J}$ are positive constants with $\sum_{j}\beta_j=1$. $\lambda$, $\beta_j$, $j\in\mathcal{J}$ satisfy the subcritical load condition: $\lambda<\sum_j \mu_j\beta_j$.
\vspace{3mm}

\noindent
\textbf{System state.} Define $x^N_{k,j}$ be the fraction of the total number of servers, which are in pool $j$ and have queue length at least $k\in \N_0$, $j\in\mathcal{J}$. Consider $x^N(t)=(x^N_{k,j}, k\in\N_0, j\in\mathcal{J})$ to be the system state at time $t\geq 0$. $x^N(\cdot)$ is a Markov process.
\begin{theorem}[{\cite[Theorem~2]{AS15}}]
\label{thm:AS15}
For all sufficient large $N$, the Markov process $x^N(\cdot)$ is ergodic and $x^N(\infty)\Rightarrow x^*=(x^*_{k,j},k\in\N_0,j\in\mathcal{J})$, where $x^*$ is the solution of the following equation:
\begin{equation*}
    \begin{split}
        \lambda=\sum_j\mu_jx^*_{1,j},\qquad
        \mu_jx^*_{j,1}/(\beta_j-x^*_{j,1})=\mu_lx^*_{l,1}/(\beta_l-x^*_{l,1}), \forall j,l\in\mathcal{J},\quad\text{and}\quad 
        x^*_{k,j}=0,\forall k\geq 2, j\in\mathcal{J}.
    \end{split}
\end{equation*}
\end{theorem}

\section{Fluid limit of server-independent systems}\label{app:server-independent-fluid}
\noindent
\textbf{Model.} The system consists of $K$ dispatchers and $N$ identical servers. For each dispatcher $k\in[K]$, the arrival process of tasks is a Poisson process with rate $\lambda^N_k=N\lambda_k$, independent of the other processes. Each new task will be assigned to one of servers under JIQ policy. Tasks from the dispatcher $k\in[K]$ are called type $k$ tasks. The service time of a task of type $k\in[K]$ at a server is exponentially distributed with mean $1/\mu_{k}\in(0,\infty)$.
\vspace{3mm}

\noindent
\textbf{Asymptotic regime.} For all $k\in[K]$, $\lambda_k$ is a finite positive constant. $\lambda_k$ and $\mu_k$, $k\in[K]$, satisfy the subcritical load condition: $\sum_{k\in[K]}\lambda_k/\mu_k<1$.
\vspace{3mm}

\noindent
\textbf{System state.} Let $X^N_{k}$ be the number of servers that are currently serving a task of type $k\in[K]$ and $x^N_{k}=\frac{X^N_k}{N}$. Consider $x^N(t)=(x^N_k(t),k\in[K])$ to be the system state at time $t\geq 0$. For any $N$, we view $x^N$ as an element of the common space
$$\chi_1=\Big\{x=(x_k,k\in[K]):\sum_{k\in[K]}x_k=1\text{ and } x_k\geq 0,\forall k\in[K]\Big\}$$
equipped with the $\ell_1$-norm.

\begin{prop}\label{prop:fluid-multiclass}
Assume that for each $N$-th system, it starts with the idle state, i.e., $x^N_k(0)=0$, $\forall k\in[K]$. Then for any finite $T\geq 0$,  the scaled process $x^N$ converges weakly to the deterministic process $x$ uniformly on $[0,T]$,
where $x(t)=(x_{k}(t),k\in[K])$, $t\in[0,T]$ and each $x_{k}(t)$ satisfies the following differential equation: 
\begin{equation}\label{eq:ode-multitask}
    \frac{d x_{k}(t)}{dt}=\lambda_{k}-\mu_k x_{k}(t).
\end{equation}
\end{prop}
The proof is fairly straightforward. We provide it here for completeness.
\begin{proof}
First, consider the truncated system where all servers have unit buffer capacity. 
In other words, if a task arrives at the system and finds that there is no idle server, then it will be lost. In the truncated system under the JIQ policy, a new task will always be assigned to an idle server, if any. Let $\Bar{x}^N_k$ be the fraction of servers that are serving tasks of type $k\in[K]$ and $\Bar{x}^N=(\Bar{x}^N_k,k\in[K])$. Let $A_k$ and $D_k$, $k\in[K]$, are Poisson process with unit rate and independent of each other. 
Hence, we can write the system evolution as follows: for all $k\in[K]$,
\begin{equation}\label{eq:xN-martingale}
    \begin{split}
        \Bar{x}^N_k(t)=&\frac{1}{N}A_k\big(N\lambda_k\mathds{1}_{(\sum_{k\in[K]}\Bar{x}^N_k(t)<1)}t\big)-\frac{1}{N}D_k\big(N \Bar{x}^N_k(t) \mu_k t\big)
        = \lambda_k t \mathds{1}_{(\sum_{k\in[K]}\Bar{x}^N_k(t)<1)}- \Bar{x}^N_k(t) \mu_k t + \mathcal{M}^N_{A_k}(t)-\mathcal{M}^N_{D_k}(t),
    \end{split}
\end{equation}
where $$\mathcal{M}^N_{A_k}(t)=\frac{1}{N}A_k\big(N\lambda_k\mathds{1}_{(\sum_{k\in[K]}\Bar{x}^N_k(t)<1)}\big)-\lambda_k\mathds{1}_{(\sum_{k\in[K]}\Bar{x}^N_k(t)<1)}\quad \text{and}\quad \mathcal{M}^N_{D_k}(t)=\frac{1}{N}D_k\big(N \Bar{x}^N_k(t) \mu_k\big)- \Bar{x}^N_k(t) \mu_k. $$
By Doob's inequality, we have that for any $\varepsilon>0$,
\begin{equation}
    \begin{split}
        \PP\big(\sup_{t\in[0,T]}\mathcal{M}^N_{A_k}(t)\geq \varepsilon  \big)
        =&\PP\big(\sup_{t\in[0,T]}A_k\big(N\lambda_k\mathds{1}_{(\sum_{k\in[K]}\Bar{x}^N_k(t)<1)}\big) - N \lambda_k\mathds{1}_{(\sum_{k\in[K]}\Bar{x}^N_k(t)<1)} \geq N\varepsilon  \big)\\
        \leq &\frac{N\lambda_k}{N^2\varepsilon^2}\rightarrow 0,\text{ as }N\rightarrow\infty,
    \end{split}
\end{equation}
which implies that $\{\mathcal{M}^N_{A_k}(t)\}_{t\geq 0}\dto0$. Similarly, we get $\{\mathcal{M}^N_{D_k}(t)\}_{t\geq 0}\dto0$, and by the independence of the processes $A(\cdot)$ and $D(\cdot)$, 
\begin{equation}\label{eq:martingale-conv}
    \Big\{\mathcal{M}^N_{A_k}(t)-\mathcal{M}^N_{D_k}(t)\Big\}_{t\geq 0}\dto0.
\end{equation}
Now, we prove the relative compactness of the processes $\Bar{x}^N$. To show the relative compactness, we will verify the conditions in Theorem~\ref{thm: relative-compact}. 
Since $\sum_{k\in[K]}\Bar{x}^N_k=1$, condition~\eqref{eq:cond-a} holds trivially. For condition~\eqref{eq:cond-b}, note that
\begin{equation}\label{eq:r-x1-x2}
    \begin{split}
        r(\Bar{x}^N(t_1)-\Bar{x}^N(t_2))\leq &\sum_{k\in[K]}\big|\lambda_k\mathds{1}_{(\sum_{k\in[K]}\Bar{x}^N_k(t_1)<1)}t_1-\lambda_k\mathds{1}_{(\sum_{k\in[K]}\Bar{x}^N_k(t_2)<1)}t_2\big|
        +\sum_{k\in[K]}\big|\mu_k(\Bar{x}^N_k(t_1)t_1-\Bar{x}^N_k(t_2)t_2)\big|\\
        &\hspace{4cm}+\sum_{k\in[K]}\big|\mathcal{M}^N_{A_k}(t_1)-\mathcal{M}^N_{D_k}(t_1)-\mathcal{M}^N_{A_k}(t_2)+\mathcal{M}^N_{D_k}(t_2)\big|\\
        \leq & \sum_{k\in[K]}(\lambda_k+\mu_k)|t_1-t_2|
        +\sum_{k\in[K]}\big|\mathcal{M}^N_{A_k}(t_1)-\mathcal{M}^N_{D_k}(t_1)-\mathcal{M}^N_{A_k}(t_2)+\mathcal{M}^N_{D_k}(t_2)\big|
    \end{split}
\end{equation}
By \eqref{eq:martingale-conv} and \eqref{eq:r-x1-x2}, we have that for any finite partition $\{t_i\}_{i=1}^n$ of $[0,T]$,
$$\max_{i}\sup_{s,t\in[t_{i-1},t_i)}r(\Bar{x}^N(s)-\Bar{x}^N(t))\leq \sum_{k\in[K]}(\lambda_k+\mu_k)\max_{i}(t_i-t_{i-1})+\zeta^N,$$
where $\PP(\zeta^N\geq \eta/2)\leq \eta$ for all large enough $N$. Now, take $\delta=\eta/(4(\sum_{k\in[K]}(\lambda_k+\mu_k))$ and any partition with $\max_{i}(t_i-t_{i-1})\leq \eta / (2(\sum_{k\in[K]}(\lambda_k+\mu_k))$ and $\min_{i}(t_i-t_{i-1})>\delta$. Hence, on the event $\{\zeta^N\geq \eta/2\}$, 
$$\max_{i}\sup_{s,t\in[t_{i-1},t_i)}r(\Bar{x}^N(s)-\Bar{x}^N(t))\leq\eta.$$
Therefore, for all large enough $N$,
$$\PP\big( \max_{i}\sup_{s,t\in[t_{i-1},t_i)}r(\Bar{x}^N(s)-\Bar{x}^N(t))\leq\eta\big)\leq \PP(\zeta^N\geq \eta/2)\leq \eta.$$
Next, we need to show that the limit $\Bar{x}(\cdot)$ of any convergent subsequence satisfies the following equation:
$$\frac{d\Bar{x}_k(t)}{dt}=\lambda_k  \mathds{1}_{(\sum_{k\in[K]}\Bar{x}^N_k(t)<1)}- \Bar{x}^N_k(t) \mu_k,\quad \forall k\in[K].$$
Define for each $k\in[K]$,
$$F_k(\Bar{x},\mathcal{M}_A,\mathcal{M}_D)=\lambda_k t \mathds{1}_{(\sum_{k\in[K]}\Bar{x}_k(t)<1)}- \Bar{x}_k(t) \mu_k t + \mathcal{M}_A(t)-\mathcal{M}_D(t).$$
It is easy to check that $\mathbf{F}=(F_1,...,F_K)$ is a continuous map, and by the continuous mapping theorem, we get the desired result. Now, we consider the original, nontruncated system described by $x^N(\cdot)$. Define 
$$\tau^N=\inf \Big\{t\geq 0: \sum_{k\in[K]}x^N_k\geq 1\Big\}.$$
By natural coupling, for all $t\in[0,\tau^N]$, $x^N$ and $\Bar{x}^N$ are identical. 
Since for any fixed $T\geq 0$, 
\begin{equation}
    \begin{split}
        \PP(\tau^N\leq T)=\PP(\sup_{t}\sum_{k\in[K]}\Bar{x}^N_k(t)\geq 1)\leq \PP(\sum_{k\in[K]}\lambda_k/\mu_k\geq 1)=0,
    \end{split}
\end{equation}
we conclude that $x^N\Rightarrow x$ which satisfies \eqref{eq:ode-multitask}. 
\end{proof}

\noindent
Let $(E,r)$ be complete and separable metric space. For any $x\in D_E[0,\infty)$, $\delta>0$ and $T>0$, define 
$$\omega'(X_n,\delta,T)\geq \eta)=\inf_{t_i}\max_{i}\sup_{s,t\in [t_{i-1},t_i)}r(x(s),x(t)),$$ 
where $\{t_i\}$ ranges over all partitions of the form $0=t_0<t_1<\cdots<t_{n-1}<T\leq t_n$ with $\min_{1\leq i\leq n}(t_i-t_{i-1})>\delta$ and $n\geq1$.

\begin{theorem}[{\cite[Theorem~7.2]{EK09}}]\label{thm: relative-compact}
Let $(E,r)$ be complete and separable, and let $\{X_n\}_{n\geq1}$ be a family of processes with sample paths in $D_E[0,\infty)$. Then, $\{X_n\}_{n\geq1}$ is relatively compact if and only if the following two conditions hold: 

(a) For every $\eta>0$ and rational $t\geq0$, there exists a compact set $\Gamma_{\eta,t}\subset E$ such that
\begin{equation}\label{eq:cond-a}
    \varliminf_{n\rightarrow\infty}\PP(X_n\in \Gamma_{\eta,t})\geq 1-\eta.
\end{equation}

(b) For every $\eta>0$ and $T>0$, there exists $\delta>0$ such that
\begin{equation}\label{eq:cond-b}
    \varlimsup_{n\rightarrow\infty} \PP(\omega'(X_n,\delta,T)\geq \eta)\leq \eta.
\end{equation}

\end{theorem}

\section{Find a Desired Partition $(\mathbf{w},\mathbf{v})$}\label{sec:find-wvp-subcritical}
For convenience of discussion and notation, we assume that $f$ is continuous and $\xi=1$ in this section.
The inputs of Algorithm~\ref{algo: to find subcritical} are the arrival rate function $\lambda(\cdot)$, the service rate function $f(\cdot,\cdot)$, and the maximal allowed workload $\rho^*\in(0,1)$ for per server. There are two possible outputs. One is that we can get a specific pair of partitions $(\mathbf{w},\mathbf{v})$ and its corresponding stochastic matrices $\mathbf{p}$ such that \eqref{eq:P-subcritical-infinite} holds. The other one is that no matter how we make the partition, the workload of some servers is greater than $\rho^*$, which implies that some servers suffer from heavy workloads. These results are formalized in the next proposition.

\begin{algorithm}
\caption{To find a subcritical regime}\label{algo: to find subcritical}
\KwData{$\lambda(\cdot)$, $f(\cdot,\cdot)$ and $\rho^*$;}
$\text{flag} \gets -1$\;
$n \gets 1$\;
\While{$\text{flag} = -1$}{
consider the partition $0<1/2^n<2/2^n<\cdots<(2^n-1)/2^n<2^n/2^n=1$\;
$\lambda_h \gets \frac{1}{\xi}\int_{(h-1)/2^n}^{h/2^n}\lambda(x)dx$, $h\in[2^n]$\;
$\hat{\mu}_{h,m} \gets \max_{(x,y)\in[(h-1)/2^n,h/2^n)\times [(m-1)/2^n,m/2^n)}f(x,y)$, $h,m\in[2^n]$\;
$\mu^*_{h,m} \gets \min_{(x,y)\in[(h-1)/2^n,h/2^n)\times [(m-1)/2^n,m/2^n)}f(x,y)$, $h,m\in[2^n]$\;
Check the polyhedra $P^n_1$ and $P^n_2$ are empty or not, where
$$P^n_1\coloneqq\{\mathbf{p}=(p_{h,m},h\in[2^n],m\in[2^n])\in[0,1)^{2^n\times 2^n}: \mathbf{p}\text{ satisfies }\eqref{eq: rho_1}\},$$
\begin{equation}\label{eq: rho_1}
    \begin{split}
        \sum_{m\in[2^n]}p_{h,m}=&1,\quad \forall h\in[2^n],\\
        \bar{\rho}^n_m(\mathbf{p})\coloneqq\sum_{h\in[2^n]}\frac{p_{h,m}\lambda_h }{\hat{\mu}_{h,m}/2^n}<&1,\quad \forall m\in[2^n].
    \end{split}
\end{equation}
$$P^n_2\coloneqq\{\mathbf{p}=(p_{h,m},h\in[2^n],m\in[2^n])\in[0,1)^{2^n\times 2^n}: \mathbf{p}\text{ satisfies }\eqref{eq: rho_2}\},$$
\begin{equation}\label{eq: rho_2}
    \begin{split}
        \sum_{m\in[2^n]}p_{h,m}=&1,\quad \forall h\in[2^n],\\
        \rho^n_m(\mathbf{p})\coloneqq\sum_{h\in[2^n]}\frac{p_{h,m}\lambda_h }{\mu^*_{h,m}/2^n}<&1,\quad \forall m\in[2^n].
    \end{split}
\end{equation}

\uIf{$P^n_1\neq \O$}{
$\bar{\rho}^n\coloneqq\min_{\mathbf{p}\in P^n_1}\max_{m\in[2^n]}\bar{\rho}^n_m(\mathbf{p})$
\;}
\Else{$\bar{\rho}^n\coloneqq1$\;
flag $\gets$ 0 \;}

\uIf{$P^n_2\neq \O$}{
$\rho^n\coloneqq\min_{\mathbf{p}\in P^n_2}\max_{m\in[2^n]}\rho^n_m(\mathbf{p})$
\;}
\Else{$\rho^n\coloneqq1$\;}

\uIf{$\bar{\rho}^n< \rho^*$ and $\rho^n\ge 1$}{
    $n \gets n+1$\; 
  }
  \uElseIf{$\bar{\rho}^n\ge\rho^*$}{
    flag$\gets$0\;
  }
  \Else{
    flag $\gets$ 1\;
    $\mathbf{w}=(0,1/2^n,2/2^n,...,1)$\;
    $\mathbf{v}=(0,1/2^n,2/2^n,...,1)$\;
    $\mathbf{p}=\argmin_{\mathbf{p}\in P^n_2}\max_{m\in[2^n]}\rho^n_m(\mathbf{p})$
  }

}
\uIf{flag=0}{
    return ``Some servers suffer from heavy workload''\;
  }
  \Else{return $(\mathbf{w},\mathbf{v},\mathbf{p})$\;}
\end{algorithm}

\begin{prop}\label{prop:algo-find-subcritical}
Given an $f$-sequence $\{G^N\}_{N\in\N}$ and the maximum allowed workload $\rho^*$. Algorithm~\ref{algo: to find subcritical} will either return a tuple  $(\mathbf{w},\mathbf{v},\mathbf{p})$ such that \eqref{eq:P-subcritical-infinite} holds, or return ``Some servers suffer from heavy workload'' which implies that there is no partition $(\mathbf{w}.\mathbf{v})$ making maximum workload per server less than $\rho^*$.
\end{prop}
For Proposition~\ref{prop:algo-find-subcritical}, we start to show the following two lemmas. The first lemma shows the monotonicity of $\rho^n$ and $\Bar{\rho}^n$ (defined in Algorithm~\ref{algo: to find subcritical}) with respect to $n$.
\begin{lemma}\label{lem:monotonicity-rho}
    Consider the sequences $\{\rho^n\}_n$ and $\{\bar{\rho}^n\}_n$. For all $n\in\N$, $\rho^n\geq \rho^{n+1}$ and $\Bar{\rho}^n\leq \Bar{\rho}^{n+1}$.
\end{lemma}
\begin{proof}
We only show $\rho^n\geq \rho^{n+1}$ for all $n\in\N$, since the proof of $\Bar{\rho}^n\leq \Bar{\rho}^{n+1}$ is similar.
    Consider any $n\in\N$. $(\mathbf{w}^n,\mathbf{v}^n)=(0<1/2^n<...<2^n/2^n=1,0<1/2^n<...<2^n/2^n=1)$, $\rho^n=\rho(\mathbf{w}^n,\mathbf{v}^n)$ is the optimal value with the optimal solution $\mathbf{p}^*=(p^*_{h,m},h\in[2^n],m\in[2^n])$  of the following linear programming (LP):
    \begin{align}\label{eq:lp-w-v}
        \min_{\mathbf{p}\in[0,1)^{2^n\times 2^n}}&\quad  \rho\nonumber\\
        s.t.\quad \sum_{m\in[2^n]}p_{h,m}&=1,\quad \forall h\in[2^n],\\
        \sum_{h\in[2^n]}\frac{p_{h,m}\lambda_h}{\mu^*_{h,m}/2^n}&\leq \rho,\quad \forall m\in[2^n],\nonumber
    \end{align}
    where $\lambda_h$ and $\mu^*_{h,m}$ are defined in Algorithm~\ref{algo: to find subcritical}. For $n+1$, $(\mathbf{w}^{n+1},\mathbf{v}^{n+1})=(0<1/2^{n+1}<1/2^n<3/2^{n+1}...<(2^{n+1}-1)/2^{n+1}<2^n/2^n=1,0<1/2^{n+1}<1/2^n<3/2^{n+1}...<(2^{n+1}-1)/2^{n+1}<2^n/2^n=1)$, $\rho^{n+1}=\rho(\mathbf{w}^{n+1},\mathbf{v}^{n+1})$ is the optimal value of the following LP:
    \begin{align}\label{eq:lp-w'-v'}
        \min_{\mathbf{p}\in[0,1)^{2^{n+1}\times 2^{n+1}}}&\quad  \rho'\nonumber\\
        s.t.\quad \sum_{m'\in[2^{n+1}]}p_{h',m'}&=1,\quad \forall h'\in[2^{n+1}],\\
        \sum_{h'\in[2^{n+1}]}\frac{p_{h',m'}\lambda'_{h'}}{\mu'^{*}_{h',m'}/2^{n+1}}&\leq \rho',\quad \forall m'\in[2^{n+1}],\nonumber
    \end{align}
    where $\lambda'_{h'}$ and $\mu'^*_{h',m'}$ are defined in Algorithm~\ref{algo: to find subcritical}.
    We construct a feasible solution $\mathbf{p}'=(p'_{h',m'},h'\in[2^{n+1}],m'\in[2^{n+1}])$ of the LP~\eqref{eq:lp-w'-v'} as follows: 
    for any $h'\in[2^{n+1}]$ and $m'\in[2^{n+1}]$, $p'_{h',m'}=\frac{1}{2}p^*_{\lfloor (h'+1)/2\rfloor,\lfloor (m'+1)/2\rfloor}$.
    Hence, for any $m'\in[2^{n+1}]$ with $m=\lfloor(m'+1)/2\rfloor$, we have
    \begin{equation*}
        \begin{split}
            \sum_{h'\in[2^{n+1}]}\frac{\lambda'_{h'}p'_{h',m'}}{\mu'^*_{h',m'}/2^{n+1}}&=\sum_{h\in[2^n]}\sum_{h'=2h-1}^{2h}\frac{\lambda'_{h'} p^*_{h,m}/2}{\mu'^*_{h',m'}/2^{n+1}}\\
            &=\sum_{h\in[2^n]}\sum_{h'=2h-1}^{2h}\frac{\lambda'_{h'}p^*_{h,m}}{\mu'^*_{h',m'}/2^{n}}\leq \sum_{h\in[H]}\frac{\lambda_{h}p^*_{h,m}}{\mu^*_{h,m}/2^{n}}= \rho(\mathbf{w}^n,\mathbf{v}^n), \\
        \end{split}
    \end{equation*}
where the last inequality comes from the fact that if $h=\lfloor (h'+1)/2\rfloor$ and $m=\lfloor (m'+1)/2\rfloor$,
$ \mu'^*_{h',m'}\geq \mu^*_{h,m},$
and $\sum_{h'=2h-1}^{2h} \lambda'_{h'}=\lambda_h$.
The constructed solution $\mathbf{p}'$ may not be the optimal solution to the LP~\eqref{eq:lp-w'-v'} so $\rho(\mathbf{w}^{n+1},\mathbf{v}^{n+1})\leq\rho(\mathbf{w}^n,\mathbf{v}^n)$.
\end{proof}

The next lemma states that Algorithm~\ref{algo: to find subcritical} can always find a desired tuple $(\mathbf{w}',\mathbf{v}',\mathbf{p}')$, if the $f$-sequence is in a certain $(\mathbf{w},\mathbf{v},\mathbf{p})$-subcritical regime with $\rho(\mathbf{w},\mathbf{v},\mathbf{p})< \rho^*$. The proof of Lemma~\ref{lem:find-subcritical} is mainly based on the observation that if $\rho(\mathbf{w},\mathbf{v},\mathbf{p})<\rho^*$, then for a finer partition $(\mathbf{w}',\mathbf{v}')$, we can construct a matrix $\mathbf{p}\in[0,1)^{H\times M}$ with unit row sums based on $\mathbf{p}$ such that $\rho(\mathbf{w}',\mathbf{v}',\mathbf{p}')<\rho^*$.
\begin{lemma}\label{lem:find-subcritical}
Assume that the $f$-sequence $\{G^N\}_{N\in\N}$ is in a certain $(\mathbf{w},\mathbf{v},\mathbf{p})$-subcritical regime with $\rho(\mathbf{w},\mathbf{v},\mathbf{p})< \rho^*$. Then, Algorithm~\ref{algo: to find subcritical} will return a tuple $(\mathbf{w}',\mathbf{v}',\mathbf{p}')$ such that $\rho(\mathbf{w}',\mathbf{v}',\mathbf{p}')<\rho^*$. 
\end{lemma}
 
\begin{proof}
It is sufficient to show that for any $\gamma>0$, Algorithm~\ref{algo: to find subcritical} can find a tuple $(\mathbf{w}',\mathbf{v}',\mathbf{p})$ such that $\rho(\mathbf{w}',\mathbf{v}',\mathbf{p}')\le (1+\gamma)\rho(\mathbf{w},\mathbf{v},\mathbf{p})$. 
Suppose that the $f$-sequence $\{G^N\}$ is in the $(\mathbf{w},\mathbf{v},\mathbf{p})$-subcritical regime. The corresponding pair of partitions and the matrix are
$$(\mathbf{w},\mathbf{v})=(0=w_0<w_1<\cdots<w_H=1,0=v_0<v_1<\cdots<v_M=1),$$
and $$\mathbf{p}=(p_{h,m},h\in[H],m\in[M]),$$
respectively.
Recall the definition of $\lambda_h$ and $\mu^*_{h,m}$ in Algorithm~\ref{algo: to find subcritical}.
By the definition of $\rho(\mathbf{w},\mathbf{v},\mathbf{p})$, we have 
$$\sum_{h\in[H]}\frac{p_{h,m}\lambda_h}{(v_m-v_{m-1})\mu^*_{h,m}}\le \rho(\mathbf{w},\mathbf{v},\mathbf{p})<\rho^*,\quad \forall m\in[M].$$
Since the function $f$ is Lipschitz continuous, then for any $\varepsilon>0$, there exists an $\delta(\varepsilon)>0$ such that for any $\delta\leq \delta(\varepsilon)$ and any subintervals $E_1\subseteq[0,1)$, $E_2\subseteq[0,1)$ with $|E_1|\leq \delta$, $|E_2|\leq \delta$,
\begin{align}
    \max_{(x,y)\in E_1\times E_2}f(x,y)-\min_{(x,y)\in E_1\times E_2}f(x,y)\leq \varepsilon.
\end{align}
Since $M$ is finite, then for any fixed $\gamma>0$, there exists an $\varepsilon^*>0$ such that 
$$\sum_{h\in[H]}\frac{p_{h,m}\lambda_h}{(v_m-v_{m-1})(\mu^*_{h,m}-\varepsilon^*)}\leq (1+\gamma)\rho(\mathbf{w},\mathbf{v},\mathbf{p})< (1+\gamma)\rho^*,\quad \forall m\in[M].$$
Fix $n$ such that $1/2^n < \delta(\varepsilon^*)$, $1/2^n < \min_{h\in[H]} (w_h-w_{h-1})$ , and $1/2^n < \min_{m\in[M]} (v_m-v_{m-1})$. 
Consider the partition $0<1/2^n<2/2^n<\cdots<2^n/2^n=1$. Next, we construct a matrix $\mathbf{p}'=(p'_{\hat{h},\hat{m}},\hat{h}\in[2^n],\hat{m}\in[2^n])\in[0,1)^{2^n\times 2^n}$ with unit row sums satisfying \eqref{eq:P-subcritical-infinite}. Fix any $\hat{h}\in[2^n]$ and $\hat{m}\in[2^n]$. Based on the situation if $ [(\hat{h}-1)/2^n,\hat{h}/2^n)$ (resp. $ [(\hat{m}-1)/2^n,\hat{m}/2^n)$) contains any $ w_h$ (resp. $ v_m$),  we discuss the following four cases:

Case 1: $[(\hat{h}-1)/2^n,\hat{h}/2^n)$ does not contain any $w_h$, $h\in[H]$, and $[(\hat{m}-1)/2^n,\hat{m}/2^n)$ does not contain any $v_m$, $m\in[M]$. Since $1/2^n < \min_{h\in[H]} (w_h-w_{h-1})$ and $1/2^n<\min_{m\in[M]} (v_m-v_{m-1})$, then $[(\hat{h}-1)/2^n,\hat{h}/2^n)\times [(\hat{m}-1)/2^n,\hat{m}/2^n)$ is a set of $[w_{h-1},w_h)\times [v_{m-1},v_m)$ for certain $h\in[H]$ and $m\in[M]$. Let $$p'_{\hat{h},\hat{m}}=\frac{p_{h,m}}{2^n(v_m-v_{m-1})}.$$

Case 2: $[(\hat{h}-1)/2^n,\hat{h}/2^n)$ contains one of $w_h$, $h\in[H]$, and $[(\hat{m}-1)/2^n,\hat{m}/2^n)$ does not contain any $v_m$, $m\in[M]$. W.L.O.G., suppose that $[(\hat{h}-1)/2^n,\hat{h}/2^n)$ contains $w_h$ and $[(\hat{m}-1)/2^n,\hat{m}/2^n)$ is a subinterval of $[v_{m-1},v_m)$ for certain $h\in[H]$ and $m\in[M]$. Let 
$$p'_{\hat{h},\hat{m}}=\frac{\lambda_{\hat{h},1}}{\lambda_{\hat{h},1}+\lambda_{\hat{h},2}}\frac{p_{h,m}}{2^n(v_m-v_{m-1})}+\frac{\lambda_{\hat{h},2}}{\lambda_{\hat{h},1}+\lambda_{\hat{h},2}}\frac{p_{h+1,m}}{2^n(v_m-v_{m-1})},$$ where $\lambda_{\hat{h},1}=\int_{(\hat{h}-1)/2^n}^{w_h}\lambda(x)dx$ and $\lambda_{\hat{h},2}=\int^{\hat{h}/2^n}_{w_h}\lambda(x)dx$.

Case 3: $[(\hat{h}-1)/2^n,\hat{h}/2^n)$ does not contain any $w_h$, $h\in[H]$, and $[(\hat{m}-1)/2^n,\hat{m}/2^n)$ contains one of $v_m$, $m\in[M]$. W.L.O.G., suppose that $[(\hat{h}-1)/2^n,\hat{h}/2^n)$ is a subinterval of $[w_{h-1},w_h)$ and $[(\hat{m}-1)/2^n,\hat{m}/2^n)$ contains $v_m$ for certain $h\in[H]$ and $m\in[M]$. Let $$p'_{\hat{h},\hat{m}}=\frac{p_{h,m}(v_m-(\hat{m}-1)/2^n)}{v_m-v_{m-1}}+\frac{p_{h,m+1}(\hat{m}/2^n-v_m)}{v_{m+1}-v_m}.$$

Case 4: $[(\hat{h}-1)/2^n,\hat{h}/2^n)$ contains one of $w_h$, $h\in[H]$, and $[(\hat{m}-1)/2^n,\hat{m}/2^n)$ contains one of $v_m$, $m\in[M]$. W.L.O.G., suppose that $[(\hat{h}-1)/2^n,\hat{h}/2^n)$ contains $w_h$ and $[(\hat{m}-1)/2^n,\hat{m}/2^n)$ contains $v_m$ for certain $h\in[H]$ and $m\in[M]$. Let 
\begin{equation*}
    \begin{split}
    p'_{\hat{h},\hat{m}}=&\frac{\lambda_{\hat{h},1}}{\lambda_{\hat{h},1}+\lambda_{\hat{h},2}}\Big(\frac{p_{h,m}(v_m-(\hat{m}-1)/2^n)}{v_m-v_{m-1}}+\frac{p_{h,m+1}(\hat{m}/2^n-v_m)}{v_{m+1}-v_m}\Big)\\
    &+\frac{\lambda_{\hat{h},2}}{\lambda_{\hat{h},1}+\lambda_{\hat{h},2}}\Big(\frac{p_{h+1,m}(v_m-(\hat{m}-1)/2^n)}{v_m-v_{m-1}}+\frac{p_{h+1,m+1}(\hat{m}/2^n-v_m)}{v_{m+1}-v_m}\Big),
\end{split}
\end{equation*}
where $\lambda_{\hat{h},1}=\int_{(\hat{h}-1)/2^n}^{w_h}\lambda(x)dx$ and $\lambda_{\hat{h},2}=\int^{\hat{h}/2^n}_{w_h}\lambda(x)dx$.
Next, we need to show that for all $\hat{m}\in[2^n]$,
\begin{equation}\label{eq: lambda'-p'_hm-mu'}
    \sum_{\hat{h}\in[2^n]}\frac{\lambda_{\hat{h}}p'_{\hat{h},\hat{m}}}{(1/2^n)\mu'_{\hat{h},\hat{m}}}<(1+\gamma)\rho^*,
\end{equation}
where $\lambda_{\hat{h}}=\int_{(\hat{h}-1)/2^n}^{\hat{h}/2^n}\lambda(x)dx$ and $\mu'_{\hat{h},\hat{m}}=\min_{(x,y)\in[(\hat{h}-1)/2^n,\hat{h}/2^n)\times [(\hat{m}-1)/2^n,\hat{m}/2^n)}f(x,y)$.
For each $\hat{m}\in[2^n]$, we define the following fours sets:
$$C(\hat{m},1)\coloneqq\{\hat{h}\in[2^n]:(\hat{h},\hat{m})\text{ belongs to case 1}\},$$
$$C(\hat{m},2)\coloneqq\{\hat{h}\in[2^n]:(\hat{h},\hat{m})\text{ belongs to case 2}\},$$
$$C(\hat{m},3)\coloneqq\{\hat{h}\in[2^n]:(\hat{h},\hat{m})\text{ belongs to case 3}\},$$
$$C(\hat{m},4)\coloneqq\{\hat{h}\in[2^n]:(\hat{h},\hat{m})\text{ belongs to case 4}\},$$
so $[2^n]=\cup_{k\in[4]}C(\hat{m},k)$ for any fixed $\hat{m}\in[2^n]$. More precisely, if $\hat{m}\in[2^n]$ such that $[(\hat{m}-1)/2^n,\hat{m}/2^n)$ does not contain any $v_m$, $m\in[M]$, then $[2^n]=C(\hat{m},1) \cup C(\hat{m},2)$; otherwise, $[2^n]=C(\hat{m},3) \cup C(\hat{m},4)$. 
Also, for each $\hat{m}\in[2^n]$ and $h\in[H]$, we define
$$C(\hat{m},h,1)\coloneqq\{\hat{h}\in (C(\hat{m},1)\cup C(\hat{m},3)): [(\hat{h}-1)/2^n,\hat{h}/2^n)\subseteq [w_{h-1},w_h)\},$$
and $$C(\hat{m},h,2)\coloneqq\{\hat{h}\in (C(\hat{m},2)\cup C(\hat{m},4)): [(\hat{h}-1)/2^n,\hat{h}/2^n)\text{ contains }w_h\}.$$
Consider any $\hat{m}\in[2^n]$ such that $[(\hat{m}-1)/2^n,\hat{m}/2^n)$ does not contain any $v_m$ and suppose that $[(\hat{m}-1)/2^n,\hat{m}/2^n)\subseteq[v_{m-1},v_m)$ for certain $m\in[M]$. We have 
\allowdisplaybreaks
\begin{align*}
    &\sum_{\hat{h}\in[2^n]}\frac{\lambda_{\hat{h}}p'_{\hat{h},\hat{m}}}{(1/2^n)\mu'_{\hat{h},\hat{m}}}\\
        =&\sum_{\hat{h}\in C(\hat{m},1)}\frac{\lambda_{\hat{h}} p'_{\hat{h},\hat{m}}}{(1/2^n)\mu'_{\hat{h},\hat{m}}}+\sum_{\hat{h}\in C(\hat{m},2)}\frac{\lambda_{\hat{h}} p'_{\hat{h},\hat{m}}}{(1/2^n)\mu'_{\hat{h},\hat{m}}}\\
        =& \sum_{h\in[H]}\sum_{\hat{h}\in C(\hat{m},h,1)}\frac{\lambda_{\hat{h}} p'_{\hat{h},\hat{m}}}{(1/2^n)\mu'_{\hat{h},\hat{m}}}+\sum_{h\in[H]}\sum_{\hat{h}\in C(\hat{m},h,2)}\frac{\lambda_{\hat{h}} p'_{\hat{h},\hat{m}}}{(1/2^n)\mu'_{\hat{h},\hat{m}}}\\
        \overset{(a)}{\leq}& \sum_{h\in[H]}\sum_{\hat{h}\in C(\hat{m},h,1)}\frac{\lambda_{\hat{h}} p'_{\hat{h},\hat{m}}}{(1/2^n)\mu^*_{h,m}}+\sum_{h\in[H]}\sum_{\hat{h}\in C(\hat{m},h,2)}\frac{\lambda_{\hat{h}} p'_{\hat{h},\hat{m}}}{(1/2^n)(f(w_h,\hat{m}/2^n)-\varepsilon^*)}\\
        =&\sum_{h\in[H]}\sum_{\hat{h}\in C(\hat{m},h,1)}\frac{\lambda_{\hat{h}}}{\mu^*_{h,m}}\frac{p_{h,m}}{(v_m-v_{m-1})}\\
        &+\sum_{h\in[H]}\sum_{\hat{h}\in C(\hat{m},h,2)}\frac{\lambda_{\hat{h}}}{f(w_h,\hat{m}/2^n)-\varepsilon^*}\Big(\frac{\lambda_{\hat{h},1}}{\lambda_{\hat{h},1}+\lambda_{\hat{h},2}}\frac{p_{h,m}}{(v_m-v_{m-1})}+\frac{\lambda_{\hat{h},2}}{\lambda_{\hat{h},1}+\lambda_{\hat{h},2}}\frac{p_{h+1,m}}{(v_m-v_{m-1})}\Big)\\
        \leq &\sum_{h\in[H]}\Big(\sum_{\hat{h}\in C(\hat{m},h,1)}\frac{\lambda_{\hat{h}}}{(\mu^*_{h,m}-\varepsilon^*)}\frac{p_{h,m}}{(v_m-v_{m-1})}+\sum_{\hat{h}\in C(\hat{m},h-1,2)}\frac{\lambda_{\hat{h},2}}{(\mu^*_{h,m}-\varepsilon^*)}\frac{p_{h,m}}{(v_m-v_{m-1})}\\
        &\quad \quad +\sum_{\hat{h}\in C(\hat{m},h,2)}\frac{\lambda_{\hat{h},1}}{(\mu^*_{h,m}-\varepsilon^*)}\frac{p_{h,m}}{(v_m-v_{m-1})}\Big)\\
        \overset{(b)}{=} & \sum_{h\in[H]}\frac{p_{h,m}\lambda_h}{(v_m-v_{m-1})(\mu^*_{h,m}-\varepsilon^*)}< (1+\gamma)\rho^*,
\end{align*}
where $(a)$ is from the definition of $\mu'_{\hat{h},\hat{m}}$ and $\mu^*_{h,m}$ and the property of the Lipschitz continuity of function $f(\cdot,\cdot)$, and $(b)$ is due to the fact that $ \sum_{\hat{h}\in C(\hat{m},h,1)}\lambda_{\hat{h}}+\sum_{\hat{h}\in C(\hat{m},h-1,2)}\lambda_{\hat{h},2}+\sum_{\hat{h}\in C(\hat{m},h,2)}\lambda_{\hat{h},1}=\lambda_h$.
The discussion about the case that $\hat{m}\in[2^n]$ contains certain $v_m$ is similar as above. Hence, \eqref{eq: lambda'-p'_hm-mu'} holds. Since $\gamma>0$ is arbitrary, the desired result holds.

\end{proof}

\begin{proof}[Proof of Proposition~\ref{prop:algo-find-subcritical}]
We need to show two things: (i) if for all $n\in\N$, $\Bar{\rho}^n\leq \rho^*$, then Algorithm~\ref{algo: to find subcritical} can always return a tuple $(\mathbf{w},\mathbf{v},\mathbf{p})$ such that $\rho(\mathbf{w},\mathbf{v},\mathbf{p})<1$. (ii) if for some $n\in\N$, $\bar{\rho}^n\ge\rho^*$, there is no partition $(\mathbf{w},\mathbf{v})$ such that $\rho(\mathbf{w},\mathbf{v})<\rho^*$. 

We prove (i) first. For any fixed $n\in\N$, $\Bar{\rho}^n$ is the optimal value of the following LP:
    \begin{align}\label{ref:lp-hat-mu}
        \min_{\bar{\mathbf{p}}\in[0,1)^{2^n\times 2^n}}&\quad \bar{\rho}\nonumber\\
        s.t.\quad \sum_{m\in[2^n]}\bar{p}_{h,m}&=1,\quad \forall h\in[2^n],\\
        \sum_{h\in[2^n]}\frac{\bar{p}_{h,m}\lambda_h}{\hat{\mu}_{h,m}/2^n}&\leq \bar{\rho},\quad \forall m\in[2^n],\nonumber
    \end{align}
    where $\lambda_h$ and $\hat{\mu}_{h,m}$ are defined in Algorithm~\ref{algo: to find subcritical}. 
    Similarly, $\rho^n$ is the optimal value of the following LP:
    \begin{align}\label{ref:lp-mu*}
        \min_{\mathbf{p}\in[0,1)^{2^n\times 2^n}}&\quad \rho\nonumber\\
        s.t.\quad \sum_{m\in[2^n]}p_{h,m}&=1,\quad \forall h\in[2^n],\\
        \sum_{h\in[2^n]}\frac{p_{h,m}\lambda_h}{\mu^*_{h,m}/2^n}&\leq \rho,\quad \forall m\in[2^n],\nonumber
    \end{align}
    where $\mu^*_{h,m}$ is defined in Algorithm~\ref{algo: to find subcritical}. 
    
    Since the function $f$ is Lipschitz continuous, then for any $\varepsilon>0$, there exists an $\delta(\varepsilon)>0$ such that for any $\delta\leq \delta(\varepsilon)$, the following holds: for any subintervals $E_1\subseteq[0,1)$, $E_2\subseteq[0,1)$ with $|E_1|\leq \delta$, $|E_2|\leq \delta$,
\begin{align}
    \max_{(x,y)\in E_1\times E_2}f(x,y)-\min_{(x,y)\in E_1\times E_2}f(x,y)\leq \varepsilon.
\end{align}
Fix any $\varepsilon>0$, which will be determined later. Consider any $n$ with $2^n\leq \delta(\varepsilon)$. Then for any $(h,m)\in[2^n]\times [2^n]$, we have $ \hat{\mu}_{h,m}-\mu^*_{h,m}\leq \varepsilon$.
Consider an optimal solution $\bar{\mathbf{p}}^*=(\bar{p}^*_{h,m},h\in[2^n],m\in[2^n])$ to \eqref{ref:lp-hat-mu} such that for all $m\in[2^n]$, 
\begin{equation}\label{eq:hat-mu-optimal}
    \sum_{h\in[2^n]}\frac{\bar{p}^*_{h,m}\lambda_h}{\hat{\mu}_{h,m}/2^n}\leq \rho^*.
\end{equation}
Next, we construct a feasible solution $\mathbf{p}=(p_{h,m},h\in[2^n],m\in[2^n])$ to \eqref{ref:lp-mu*}. The main issue here is dealing with terms where $\mu^*_{h,m}$ is small, especially, $\mu^*_{h,m}=0$. Consider any fixed $x>0$ which will be determined later. For each $h\in[2^n]$, let $M^h_1=\inf\{m\in[2^n]:\hat{\mu}_{h,m}\geq \hat{\mu}_{h,m'},\forall m'\neq m\}$, $M^h_2=\{m\in[2^n]\setminus M^h_1: \mu^*_{h,m}\geq x\}$ and $M^h_3=\{m\in [2^n]: \mu^*_{h,m}<x\}$. Let $p^h_x=\sum_{m\in M^h_3}\bar{p}^*_{h,m}$. We construct $\mathbf{p}$ as follows: for $m\in M^h_1$, $p_{h,m}=\Bar{p}^*_{h,m}+p^h_x$; for $m\in M^h_2$, $p_{h,m}=\bar{p}^*_{h,m}$; for $m\in M^h_3$, $p_{h,m}=0$. By \eqref{eq:hat-mu-optimal}, it is easy to get that $\sum_{h\in[2^n]}p^h_x\leq \rho^* x/\lambda$. Now, we verify the constructed $\mathbf{p}$ satisfying \eqref{ref:lp-mu*}. First, for each $h\in[2^n]$, it is easy to verify $\sum_{m\in[2^n]}p_{h,m}=\sum_{m\in[2^n]}\Bar{p}^*_{h,m}=1$. 
For each $m\in[2^n]$,
\begin{align}\label{eq:verify-pkm}
    \sum_{h\in[H]}\frac{p_{h,m}\lambda_h}{\mu^*_{h,m}/2^n}&=\sum_{h\in[2^n]: m\in M^h_1}\frac{p_{h,m}\lambda_h}{\mu^*_{h,m}/2^n}+\sum_{h\in[2^n]: m\in M^h_2}\frac{p_{h,m}\lambda_h}{\mu^*_{h,m}/2^n}\nonumber\\
    &\overset{(a)}{\leq} \sum_{h\in[2^n]: m\in M^h_1}\frac{\Bar{p}^*_{h,m}\lambda_h}{\max(\hat{\mu}_{h,m}-\varepsilon,x)/2^n}+\sum_{h\in[2^n]: m\in M^h_2}\frac{\bar{p}^*_{h,m}\lambda_h}{\max(\hat{\mu}_{h,m}-\varepsilon,x)/2^n}+\sum_{h\in[2^n]: m\in M^h_1}\frac{p^h_x\lambda_h}{(\mu^0-\varepsilon)/2^n}\nonumber\\
    &\overset{(b)}{\leq } (1+\frac{\varepsilon}{x^2})\Big(\sum_{h\in[2^n]: m\in M^h_1}\frac{\Bar{p}^*_{h,m}\lambda_h}{(\hat{\mu}_{h,m}-\varepsilon)/2^n}+\sum_{h\in[2^n]: m\in M^h_2}\frac{\bar{p}^*_{h,m}\lambda_h}{(\hat{\mu}_{h,m}-\varepsilon)/2^n}\Big)+\frac{\rho^*x}{\mu^0-\varepsilon}\nonumber\\
    &\overset{(c)}{\leq } \rho^*+\frac{\varepsilon\rho^*}{x^2}+\frac{\rho^*x}{\mu^0-\varepsilon}.
\end{align}
The first two terms of (a) are by the fact that $\hat{\mu}_{h,m}-\mu^*_{h,m}\leq \varepsilon$ for all $(h,m)\in[2^n]\times [2^n]$ and $\mu^*_{h,m}\geq x$ for all $m\in M^h_1\cup M^h_2$ and $h\in[2^n]$. The last term of (a) comes from $\max_{y\in[0,1)}f(x,y)\geq \mu^0$ for all $x\in[0,1)$. (b) is due to the fact that the function $1/y$ is Lipschitz continuous on $[x,0)$ with the coefficient $1/x^2$. (c) is by \eqref{eq:hat-mu-optimal}. We can choose $(\varepsilon,x)$ such that $ \frac{\varepsilon\rho^*}{x^2}+\frac{\rho^*x}{\mu^0-\varepsilon}<1-\rho^*$, which implies \eqref{eq:verify-pkm}$<1$. Such $(\varepsilon,x)$ exists since $\varepsilon>0$ can be an arbitrary positive number. Hence, $\rho^n$ must be strictly less than 1, which implies that Algorithm~\ref{algo: to find subcritical} must return the tuple $(\mathbf{w},\mathbf{v},\mathbf{p})$ at most $n$ steps.

Next, we prove (ii). By definition, we have $\Bar{\rho}^n\leq \rho^n$ for all $n\in \N$. By Lemma~\ref{lem:monotonicity-rho}, we have that for each $n\in \N$,
\begin{equation}\label{eq:rho-bar-rho}
    \Bar{\rho}^1\leq \Bar{\rho}^2\leq \cdots \leq \Bar{\rho}^n\leq  \rho^n\leq \cdots\leq  \rho^2\leq  \rho^1.
\end{equation}
By Lemma~\ref{lem:find-subcritical}, if there exists a tuple $(\mathbf{w},\mathbf{v},\mathbf{p})$ with $\rho(\mathbf{w},\mathbf{v},\mathbf{p})<\rho^*$, then there exists $n'\in\N$ with $\rho^{n'}<\rho^*$. By \eqref{eq:rho-bar-rho}, for all $n\in\N$, $\bar{\rho}^n\leq \rho^{n'}<\rho^*$, which is contradicted by the fact that some $n\in\N$, $\bar{\rho}^n\ge\rho^*$. 
\end{proof}

\section{\textsc{MinDrift} policy~\cite{AS05}}\label{sec:MinDrift}
Recall the weighted queue length vector $\mathbf{Q}^N$ defined by~\eqref{defn:weighted-queue-length}.
In~\cite{AS05}, the author proposed a policy called \textsc{MinDrift} to asymptotically minimize the cost function $C(\mathbf{Q}^N)$. Here, we give a special case below when $C(\mathbf{Q}^N)\coloneqq\frac{1}{2}\sum_{j\in\mathcal{V}}(Q^N_j)^2$.

\begin{defn}[\textsc{MinDrift}]\label{def:MinDrift}
Define the selected set $\sel_i{(t)}$ of servers for a type $i$ task at time $t$ as $$\sel_i{(t)}\coloneqq\argmin_{j\in \mathcal{V}:\mu_{i,j}>0} \frac{Q_j(t)}{u_{i,j}}.$$
When a type-$i$ task comes at time $t$, the corresponding dispatcher will select a server from $\sel_i(t)$ uniformly at random and assign the task to the selected one immediately. 
\end{defn}
Note that $Q_j(t)/\mu_{i,j}$ approximates the increment of the cost for an arriving task, especially when $Q_j(t)$ is large. 
Thus, sending the type $i$ task to $\sel_i(t)$ asymptotically minimize the drift. 

\section{More numerical results}\label{app:more-numerical}
\textbf{Scenario 1}: At $t=0$, all servers have 1 task in the queue.
\begin{figure}[!htb]
     \centering
     \begin{subfigure}[b]{0.3\textwidth}
         \centering
         \includegraphics[width=\textwidth]{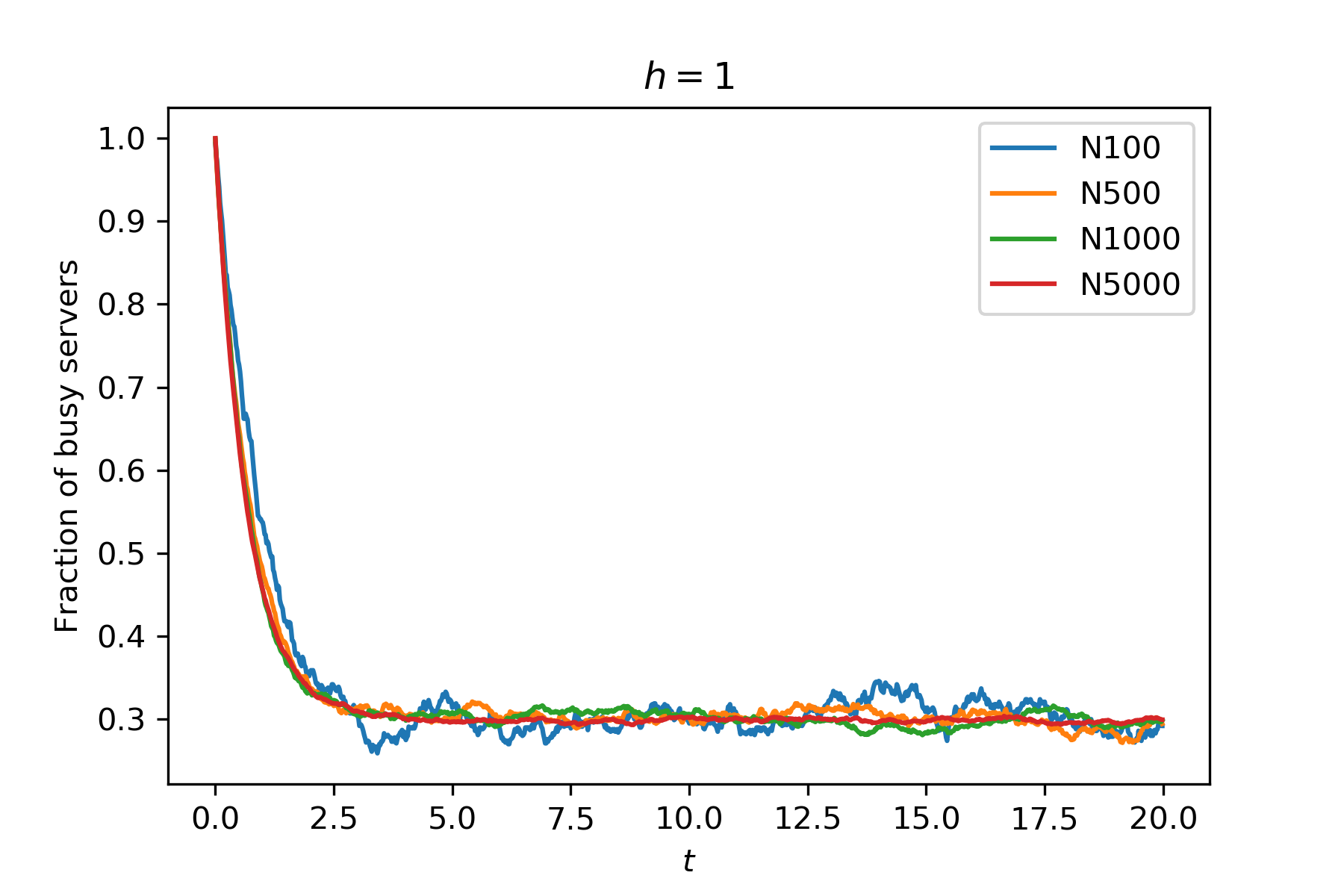}
     \end{subfigure}
     \hfill
     \begin{subfigure}[b]{0.3\textwidth}
         \centering
         \includegraphics[width=\textwidth]{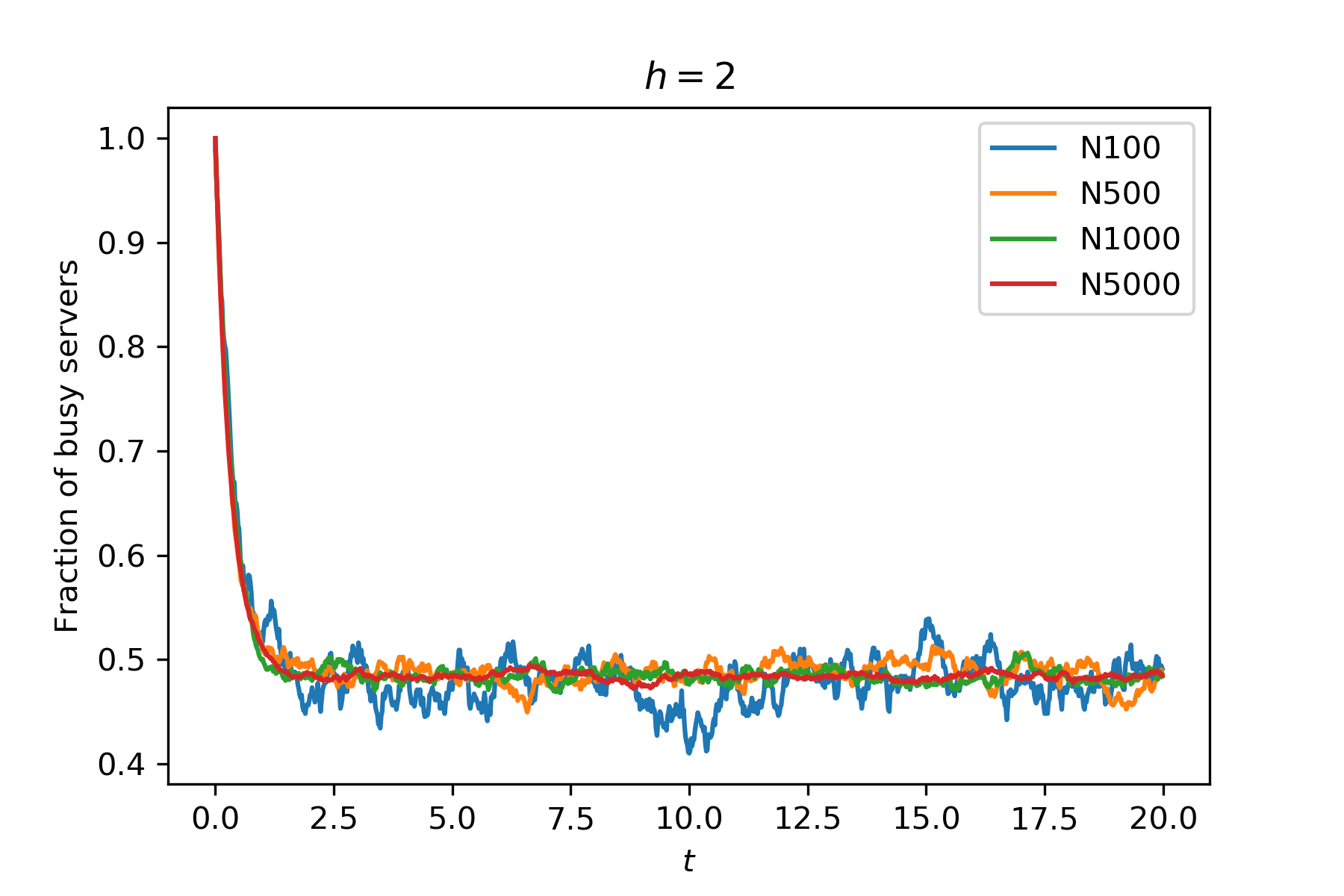}
     \end{subfigure}
     \hfill
     \begin{subfigure}[b]{0.3\textwidth}
         \centering
         \includegraphics[width=\textwidth]{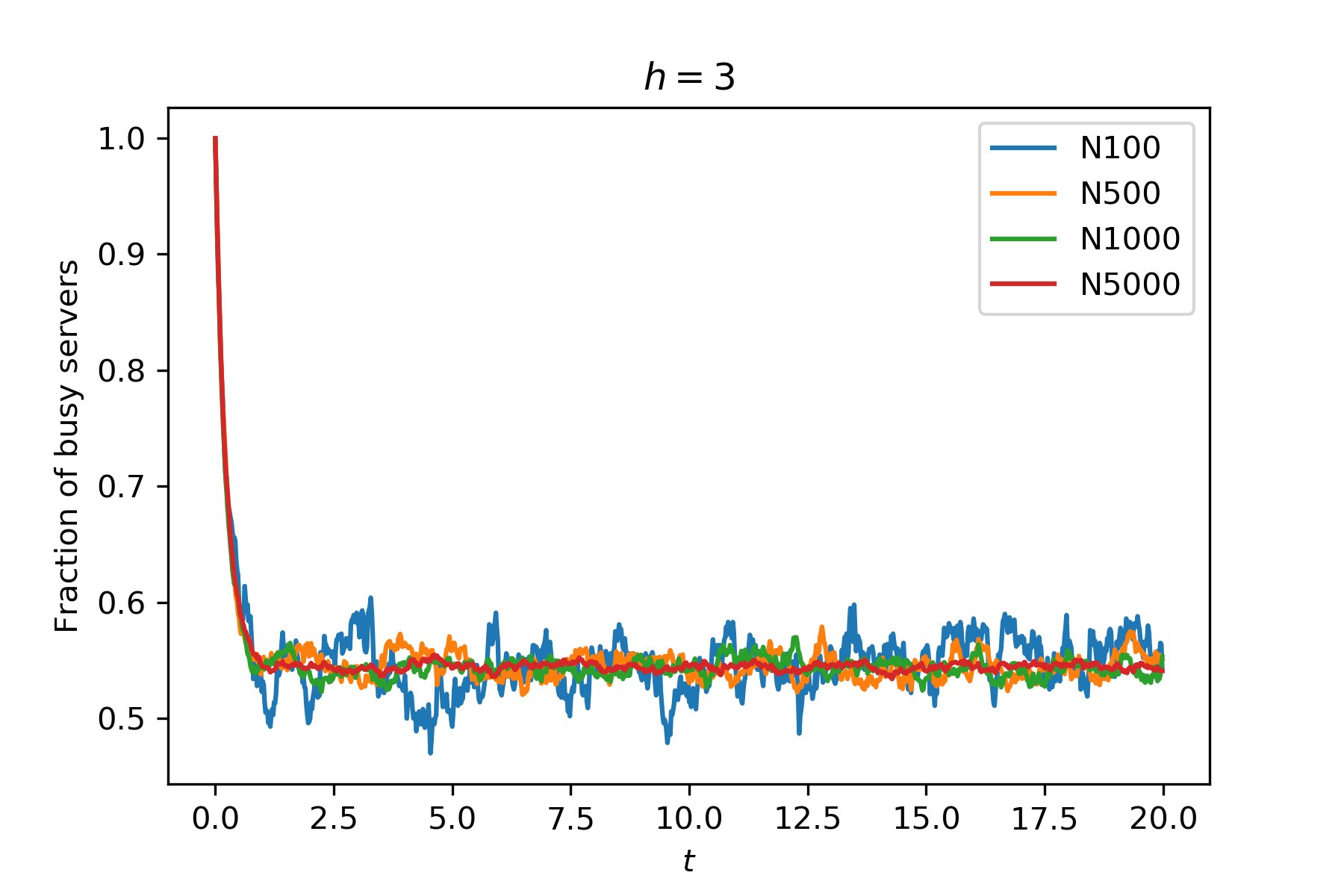}
     \end{subfigure}
     \hfill
     \begin{subfigure}[b]{0.3\textwidth}
         \centering
         \includegraphics[width=\textwidth]{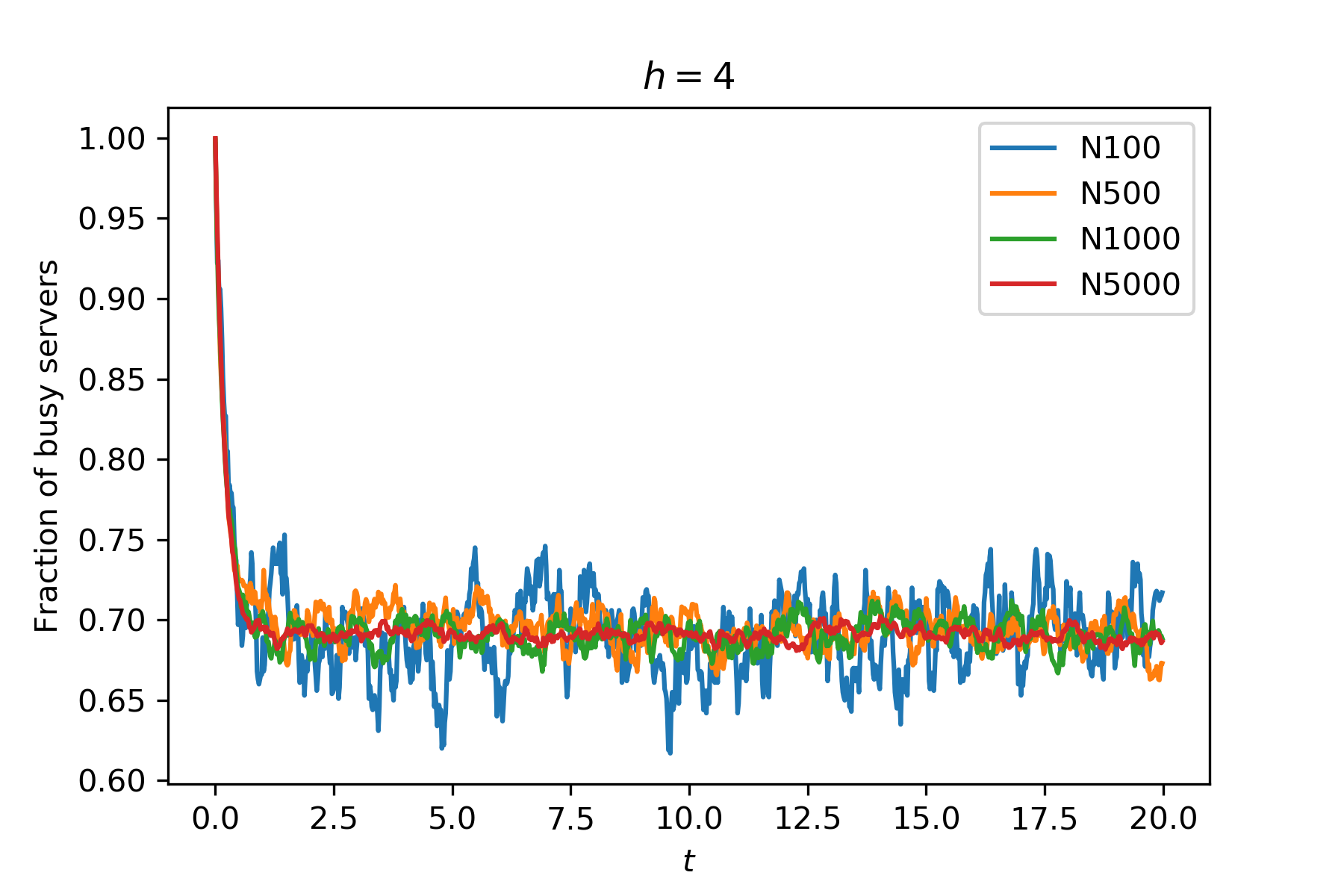}
     \end{subfigure}
     \begin{subfigure}[b]{0.3\textwidth}
         \centering
         \includegraphics[width=\textwidth]{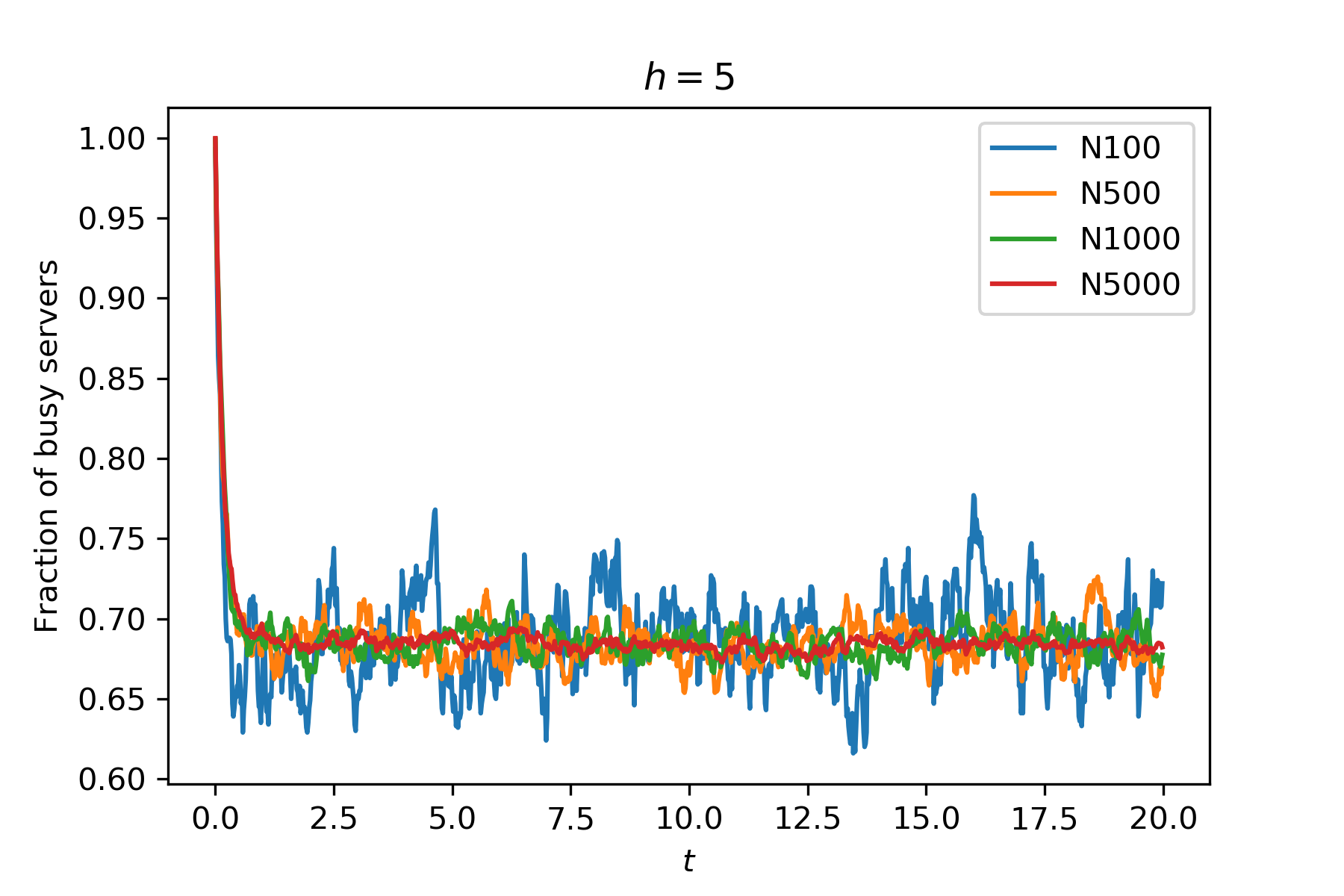}
     \end{subfigure}
     \caption{Scenario 1}
    \label{fig:sample path-busy}
\end{figure}


\textbf{Scenario 2}: At $t=0$, half of the servers is idle and the rest has 1 in the queue.

\begin{figure}[!htb]
     \centering
     \begin{subfigure}[b]{0.3\textwidth}
         \centering
         \includegraphics[width=\textwidth]{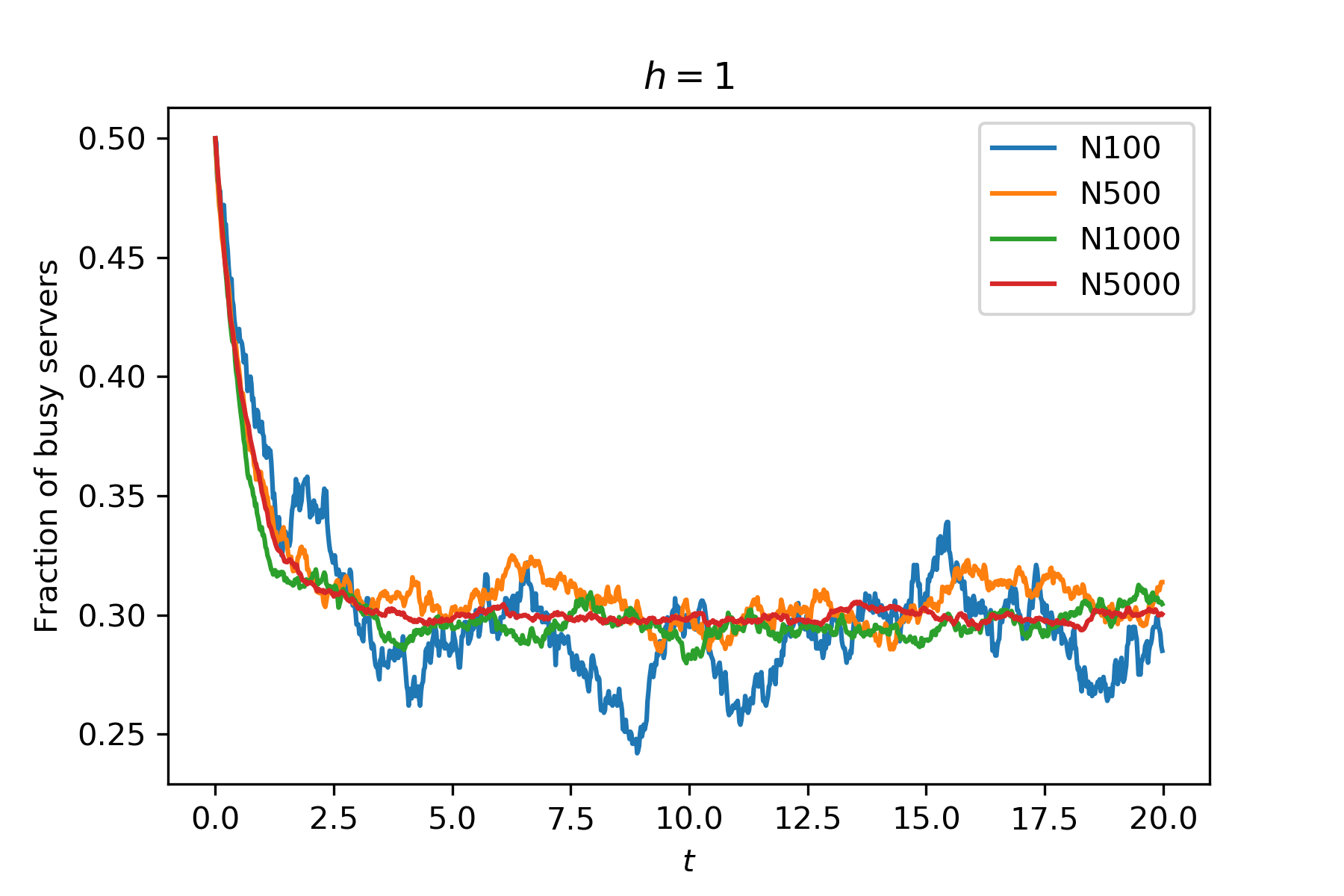}
     \end{subfigure}
     \hfill
     \begin{subfigure}[b]{0.3\textwidth}
         \centering
         \includegraphics[width=\textwidth]{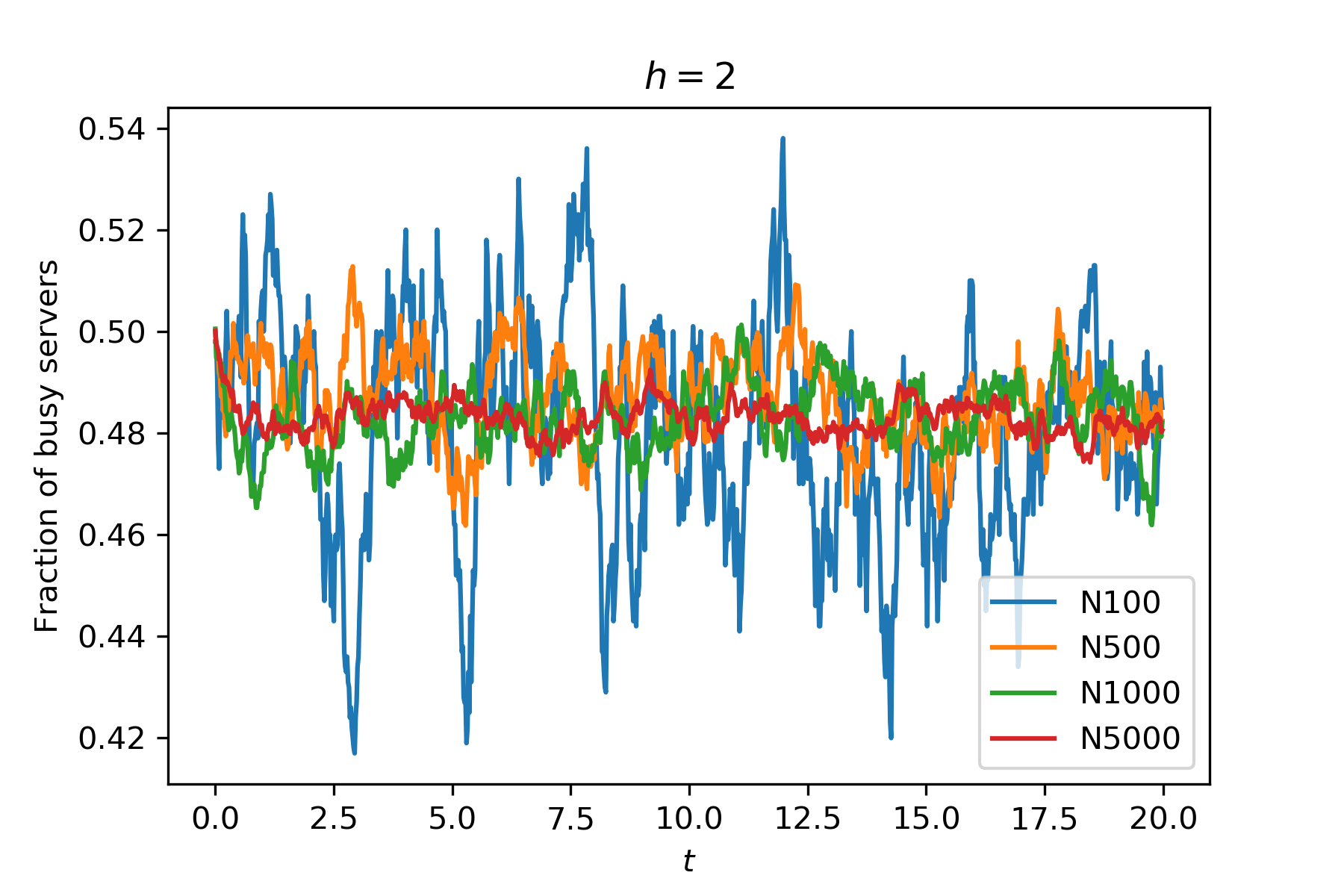}
     \end{subfigure}
     \hfill
     \begin{subfigure}[b]{0.3\textwidth}
         \centering
         \includegraphics[width=\textwidth]{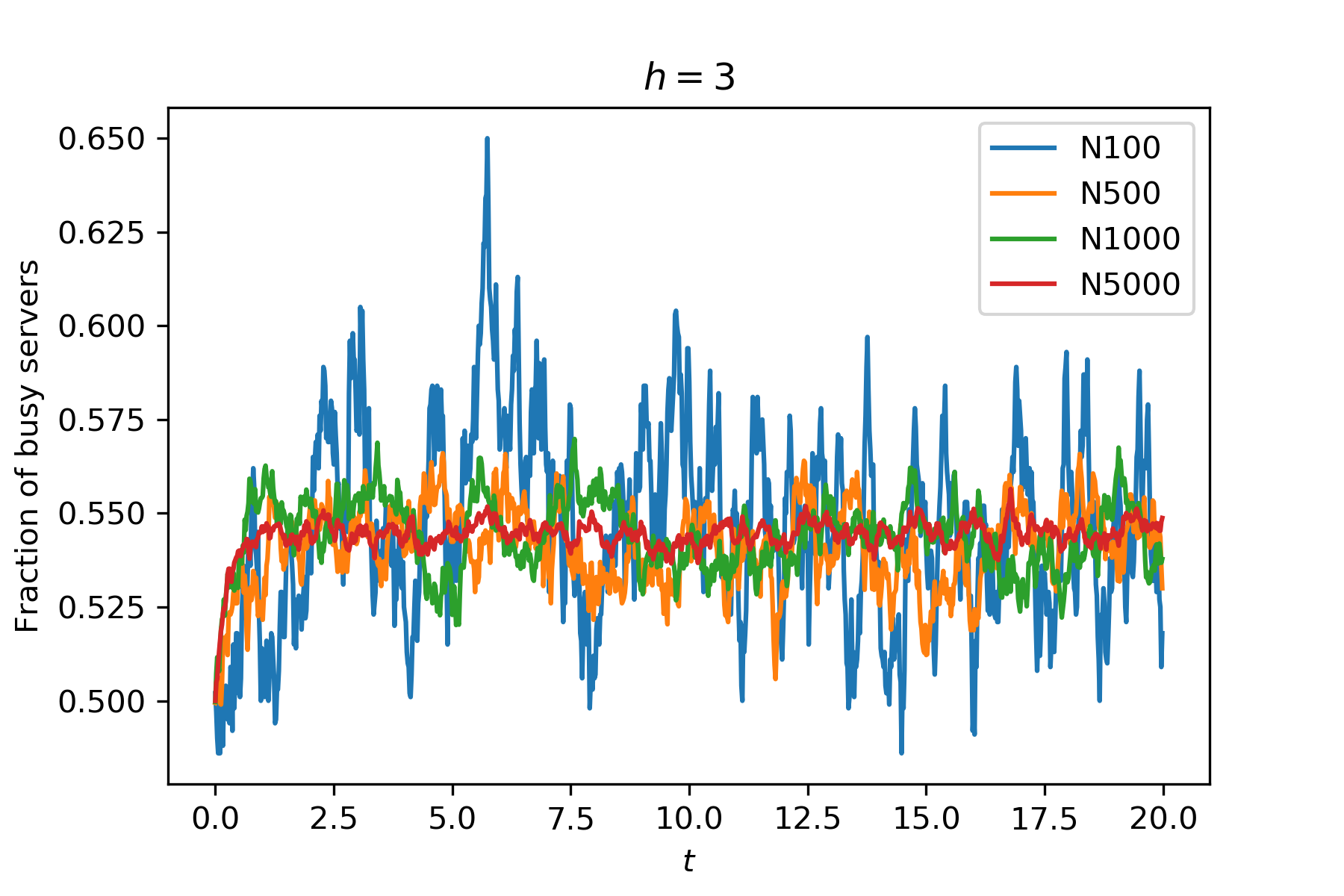}
     \end{subfigure}
     \hfill
     \begin{subfigure}[b]{0.3\textwidth}
         \centering
         \includegraphics[width=\textwidth]{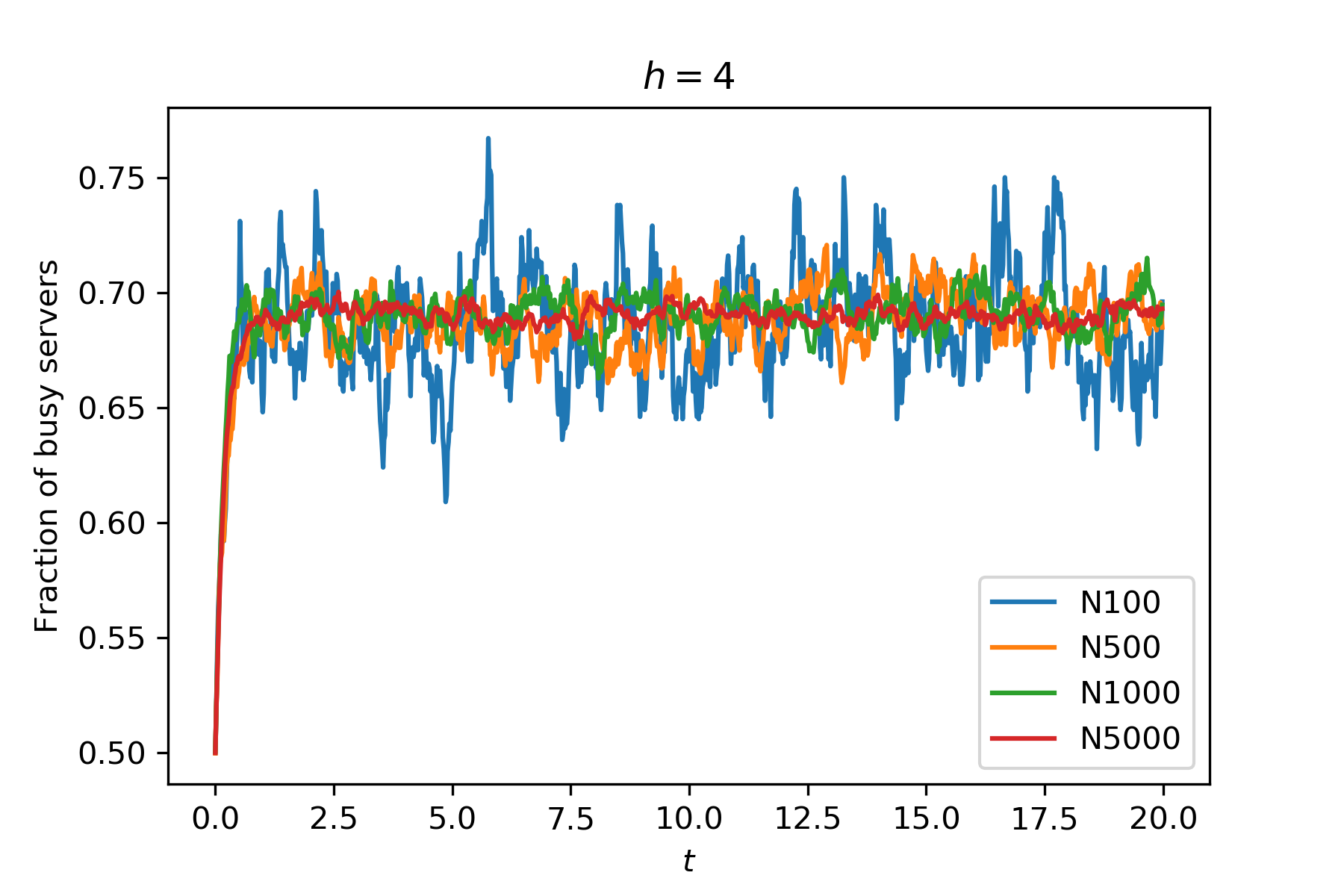}
     \end{subfigure}
     \begin{subfigure}[b]{0.3\textwidth}
         \centering
         \includegraphics[width=\textwidth]{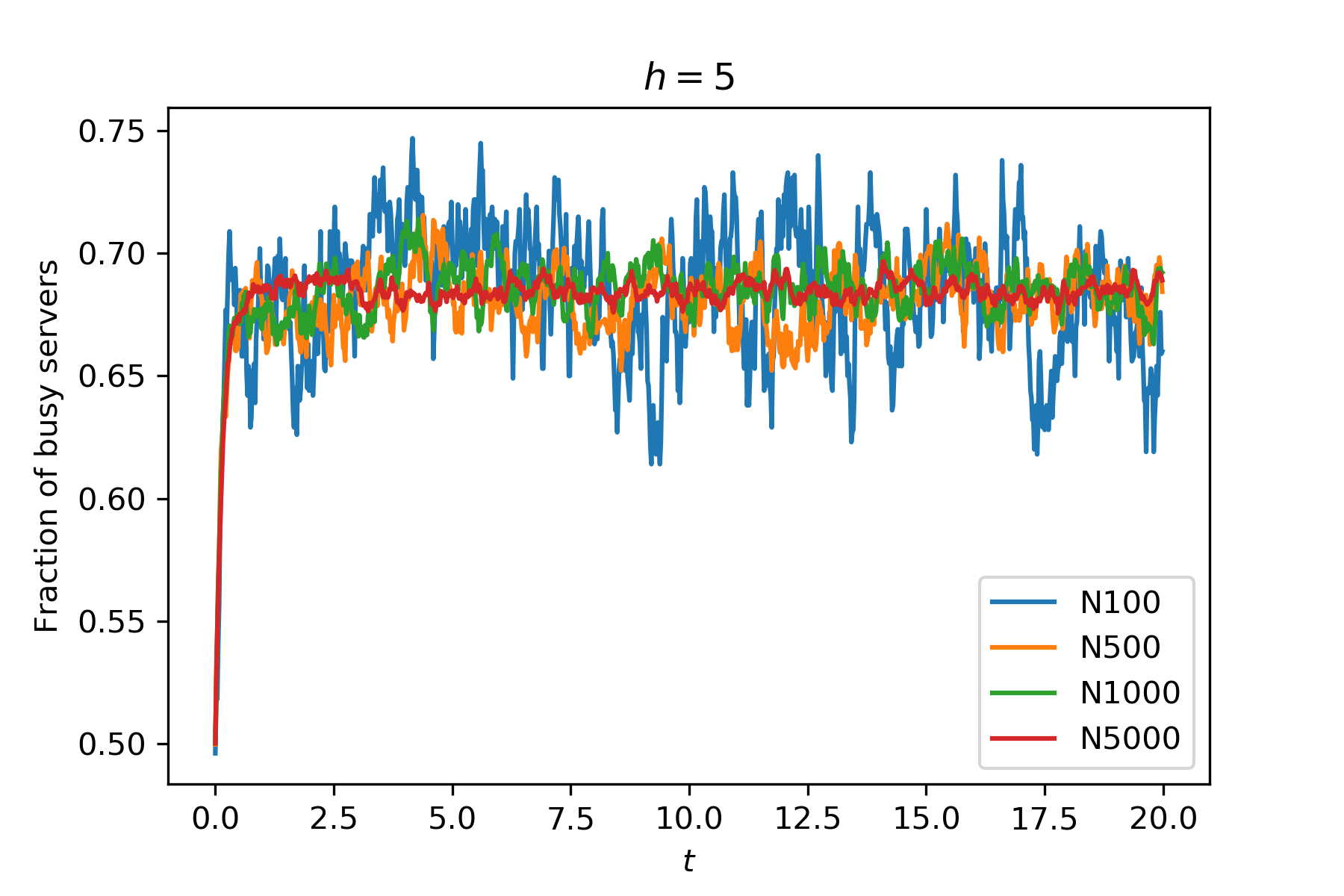}
     \end{subfigure}
     \caption{Scenario 2}
    \label{fig:sample path-half}
\end{figure}


\bibliographystyle{unsrtnat} 


\bibliography{references,references-debankur,reference-manual}

\begin{thebibliography}{40}
\providecommand{\natexlab}[1]{#1}
\providecommand{\url}[1]{\texttt{#1}}
\expandafter\ifx\csname urlstyle\endcsname\relax
  \providecommand{\doi}[1]{doi: #1}\else
  \providecommand{\doi}{doi: \begingroup \urlstyle{rm}\Url}\fi

\bibitem[Mitzenmacher(1996)]{Mitzenmacher96}
Michael Mitzenmacher.
\newblock \emph{{The power of two choices in randomized load balancing}}.
\newblock PhD thesis, University of California, Berkeley, 1996.

\bibitem[Vvedenskaya et~al.(1996)Vvedenskaya, Dobrushin, and
  Karpelevich]{VDK96}
Nikita~Dmitrievna Vvedenskaya, Roland~L'vovich Dobrushin, and
  Fridrikh~Izrailevich Karpelevich.
\newblock {Queueing system with selection of the shortest of two queues: An
  asymptotic approach}.
\newblock \emph{Problemy Peredachi Informatsii}, 32\penalty0 (1):\penalty0
  20--34, 1996.

\bibitem[Bramson(2011)]{Bramson11}
Maury Bramson.
\newblock {Stability of join the shortest queue networks}.
\newblock \emph{Ann. Appl. Probab.}, 21\penalty0 (4):\penalty0 1568--1625,
  2011.
\newblock \doi{10.1214/10-AAP726}.
\newblock URL \url{https://doi.org/10.1214/10-AAP726}.

\bibitem[Lu et~al.(2011)Lu, Xie, Kliot, Geller, Larus, and Greenberg]{LXKGLG11}
Yi~Lu, Qiaomin Xie, Gabriel Kliot, Alan Geller, James~R. Larus, and Albert
  Greenberg.
\newblock {Join-idle-queue: a novel load balancing algorithm for dynamically
  scalable web services}.
\newblock \emph{Perform. Eval.}, 68\penalty0 (11):\penalty0 1056--1071, 2011.
\newblock ISSN 01665316.
\newblock \doi{10.1016/j.peva.2011.07.015}.

\bibitem[Eschenfeldt and Gamarnik(2018)]{EG18}
Patrick Eschenfeldt and David Gamarnik.
\newblock {Join the Shortest Queue with Many Servers . The Heavy-Traffic
  Asymptotics}.
\newblock \emph{Math. Oper. Res.}, \penalty0 (October), 2018.

\bibitem[Gupta and Walton(2019)]{GW19}
Varun Gupta and Neil Walton.
\newblock {Load balancing in the nondegenerate slowdown regime}.
\newblock \emph{Oper. Res.}, 67\penalty0 (1):\penalty0 281--294, 2019.
\newblock \doi{10.1287/opre.2018.1768}.

\bibitem[Weng et~al.(2020)Weng, Zhou, and Srikant]{WZS20}
Wentao Weng, Xingyu Zhou, and R.~Srikant.
\newblock {Optimal load balancing with locality constraints}.
\newblock \emph{Proc. ACM Meas. Anal. Comput. Syst.}, 4\penalty0 (3):\penalty0
  1--37, 2020.
\newblock \doi{10.1145/3428330}.
\newblock URL \url{https://doi.org/10.1145/3428330}.

\bibitem[Rutten and Mukherjee(2022)]{RM22}
Daan Rutten and Debankur Mukherjee.
\newblock {Load balancing under strict compatibility constraints}.
\newblock \emph{Math. Oper. Res.}, 2022.
\newblock \doi{10.1287/moor.2022.1258}.

\bibitem[van~der Boor et~al.(2022)van~der Boor, Borst, van Leeuwaarden, and
  Mukherjee]{BBLM21}
M.~van~der Boor, S.C. Borst, J.S.H. van Leeuwaarden, and D.~Mukherjee.
\newblock {Scalable load balancing in networked systems: A survey of recent
  advances}.
\newblock \emph{SIAM Rev.}, 64\penalty0 (3):\penalty0 554--622, 2022.
\newblock \doi{10.1137/20M1323746}.
\newblock URL \url{https://doi.org/10.1137/20M1323746}.

\bibitem[Lov{\'{a}}sz(2012)]{ll12}
László Lov{\'{a}}sz.
\newblock \emph{{Large Networks and Graph Limits}}.
\newblock Colloquium Publications, 2012.

\bibitem[Stolyar(2005)]{AS05}
Alexander~L. Stolyar.
\newblock {Optimal Routing in Output-Queued Flexible Server Systems}.
\newblock \emph{Probab. Eng. Inf. Sci.}, 19:\penalty0 141--189, 2005.

\bibitem[Mandelbaum and Stolyar(2004)]{MS04}
Avi Mandelbaum and Alexander~L. Stolyar.
\newblock {Scheduling flexible servers with convex delay costs: Heavy-traffic
  optimality of the generalized c{$\mu$}-rule}.
\newblock \emph{Oper. Res.}, 52\penalty0 (6):\penalty0 836--855, 2004.

\bibitem[Zhao et~al.(2021)Zhao, Banerjee, and Mukherjee]{zhao21}
Zhisheng Zhao, Sayan Banerjee, and Debankur Mukherjee.
\newblock {Many-server asymptotics for Join-the-Shortest Queue in the
  Super-Halfin-Whitt Scaling Window}.
\newblock 5 2021.
\newblock URL \url{http://arxiv.org/abs/2106.00121}.

\bibitem[Banerjee and Mukherjee(2019)]{BM19a}
Sayan Banerjee and Debankur Mukherjee.
\newblock {Join-the-shortest queue diffusion limit in Halfin-Whitt regime: Tail
  asymptotics and scaling of extrema}.
\newblock \emph{Ann. Appl. Probab.}, 29\penalty0 (2):\penalty0 1262--1309,
  2019.
\newblock \doi{10.1214/18-AAP1436}.
\newblock URL \url{http://arxiv.org/abs/1803.03306}.

\bibitem[Banerjee and Mukherjee(2020)]{BM19b}
Sayan Banerjee and Debankur Mukherjee.
\newblock {Join-the-shortest queue diffusion limit in Halfin-Whitt regime:
  Sensitivity on the heavy traffic parameter}.
\newblock \emph{Ann. Appl. Probab.}, 30\penalty0 (1):\penalty0 80--144, 2020.
\newblock \doi{10.1214/19-AAP1496}.
\newblock URL \url{http://dx.doi.org/10.1214/19-AAP1496}.

\bibitem[Braverman(2020)]{Braverman18}
Anton Braverman.
\newblock {Steady-state analysis of the join-the-shortest-queue model in the
  Halfin-Whitt regime}.
\newblock \emph{Math. Oper. Res.}, 45\penalty0 (3):\penalty0 1069--1103, 2020.
\newblock \doi{10.1287/moor.2019.1023}.
\newblock URL \url{https://arxiv.org/pdf/1801.05121.pdf}.

\bibitem[Hurtado-Lange and Maguluri()]{HM20}
Daniela Hurtado-Lange and Siva~Theja Maguluri.
\newblock {Load balancing system under Join the Shortest Queue:
  Many-server-heavy-traffic asymptotics}.
\newblock \emph{QUESTA}.
\newblock URL \url{http://arxiv.org/abs/2004.04826}.

\bibitem[Hurtado-Lange and Maguluri(2022)]{HLM22}
Daniela Hurtado-Lange and Siva~Theja Maguluri.
\newblock {A load balancing system in the many-server heavy-traffic
  asymptotics}.
\newblock \emph{Queueing Syst.}, 101\penalty0 (3-4):\penalty0 353--391, 2022.
\newblock ISSN 15729443.
\newblock \doi{10.1007/s11134-022-09847-7}.
\newblock URL \url{https://doi.org/10.1007/s11134-022-09847-7}.

\bibitem[Liu and Ying(2019)]{LY19}
Xin Liu and Lei Ying.
\newblock {A simple steady-state analysis of load balancing algorithms in the
  sub-Halfin-Whitt regime}.
\newblock \emph{ACM SIGMETRICS Perform. Eval. Rev.}, 46\penalty0 (2):\penalty0
  15–17, 2019.
\newblock ISSN 0163-5999.
\newblock \doi{10.1145/3305218.3305225}.
\newblock URL \url{https://doi.org/10.1145/3305218.3305225}.

\bibitem[Liu and Ying(2022)]{LY21}
Xin Liu and Lei Ying.
\newblock {Universal scaling of distributed queues under load balancing in the
  super-Halfin-Whitt regime}.
\newblock \emph{IEEE/ACM Trans. Netw.}, 30\penalty0 (1):\penalty0 190--201,
  2022.
\newblock \doi{10.1109/TNET.2021.3105480}.

\bibitem[Braverman(2022)]{Braverman22}
Anton Braverman.
\newblock {Join the shortest queue in the Halfin-Whitt regime: Rates of
  convergence to the diffusion limit}.
\newblock 2022.

\bibitem[Rutten and Mukherjee(2023)]{RM23}
Daan Rutten and Debankur Mukherjee.
\newblock {Mean-field Analysis for Load Balancing on Spatial Graphs}.
\newblock 1 2023.
\newblock URL \url{http://arxiv.org/abs/2301.03493}.

\bibitem[Stolyar(2015{\natexlab{a}})]{AS15}
Alexander~L. Stolyar.
\newblock {Pull-based load distribution in large-scale heterogeneous service
  systems}.
\newblock \emph{Queueing Syst.}, 80\penalty0 (4):\penalty0 341--361,
  2015{\natexlab{a}}.
\newblock ISSN 15729443.
\newblock \doi{10.1007/s11134-015-9448-8}.

\bibitem[Mukherjee et~al.(2016)Mukherjee, Borst, Van~Leeuwaarden, and
  Whiting]{MBLW16-1}
Debankur Mukherjee, Sem~C. Borst, Johan S.~H. Van~Leeuwaarden, and Philip~A.
  Whiting.
\newblock {Universality of load balancing schemes on the diffusion scale}.
\newblock \emph{J. Appl. Probab.}, 53\penalty0 (4), 2016.
\newblock \doi{10.1017/jpr.2016.68}.
\newblock URL \url{https://doi.org/10.1017/jpr.2016.68}.

\bibitem[Stolyar(2017)]{Stolyar17}
Alexander~L. Stolyar.
\newblock {Pull-based load distribution among heterogeneous parallel servers:
  the case of multiple routers}.
\newblock \emph{Queueing Syst.}, 85\penalty0 (1):\penalty0 31--65, 2017.
\newblock ISSN 1572-9443.
\newblock \doi{10.1007/s11134-016-9508-8}.
\newblock URL \url{http://dx.doi.org/10.1007/s11134-016-9508-8}.

\bibitem[Bhambay and Mukhopadhyay(2022)]{BM22}
Sanidhay Bhambay and Arpan Mukhopadhyay.
\newblock {Asymptotic optimality of speed-aware JSQ for heterogeneous service
  systems}.
\newblock \emph{Perform. Eval.}, 157-158:\penalty0 102320, 2022.
\newblock ISSN 0166-5316.
\newblock \doi{https://doi.org/10.1016/j.peva.2022.102320}.
\newblock URL
  \url{https://www.sciencedirect.com/science/article/pii/S0166531622000281}.

\bibitem[Hurtado-Lange and Maguluri(2021)]{HLM21}
Daniela Hurtado-Lange and Siva~Theja Maguluri.
\newblock {Throughput and delay optimality of power-of-d choices in
  inhomogeneous load balancing systems}.
\newblock \emph{Oper. Res. Lett.}, 49\penalty0 (4):\penalty0 616--622, 2021.
\newblock ISSN 0167-6377.
\newblock \doi{https://doi.org/10.1016/j.orl.2021.06.010}.
\newblock URL
  \url{https://www.sciencedirect.com/science/article/pii/S0167637721000924}.

\bibitem[Mukhopadhyay and Mazumdar(2016)]{MM16}
A~Mukhopadhyay and R~R Mazumdar.
\newblock {Analysis of randomized Join-the-Shortest-Queue (JSQ) schemes in
  large heterogeneous processor-sharing systems}.
\newblock \emph{IEEE Trans. Control Netw. Syst.}, 3\penalty0 (2):\penalty0
  116--126, 2016.
\newblock ISSN 2325-5870.
\newblock \doi{10.1109/TCNS.2015.2428331}.

\bibitem[Zhao et~al.(2022)Zhao, Mukherjee, and Wu]{ZhaoMW22}
Zhisheng Zhao, Debankur Mukherjee, and Ruoyu Wu.
\newblock {Exploiting Data Locality to Improve Performance of Heterogeneous
  Server Clusters}.
\newblock 2022.
\newblock URL \url{http://arxiv.org/abs/2211.16416}.

\bibitem[Allmeier and Gast(2022)]{AG22}
Sebastian Allmeier and Nicolas Gast.
\newblock {Mean field and refined mean field approximations for heterogeneous
  systems: It works!}
\newblock \emph{Proc. ACM Meas. Anal. Comput. Syst.}, 6\penalty0 (1), 2 2022.
\newblock \doi{10.1145/3508033}.
\newblock URL \url{https://doi.org/10.1145/3508033}.

\bibitem[Gardner and Righter(2020)]{GR20}
Kristen Gardner and Rhonda Righter.
\newblock {Product forms for FCFS queueing models with arbitrary server-job
  compatibilities: an overview}.
\newblock \emph{Queueing Systems}, 96\penalty0 (1):\penalty0 3--51, 2020.
\newblock ISSN 1572-9443.
\newblock \doi{10.1007/s11134-020-09668-6}.
\newblock URL \url{https://doi.org/10.1007/s11134-020-09668-6}.

\bibitem[Harrison(1998)]{Harrison98}
B~Y J~Michael Harrison.
\newblock {Heavy Traffic Analysis of a System with Parallel Servers :
  Asymptotic Optimality of Discrete-Review Policies}.
\newblock \emph{Ann. Appl. Probab.}, 8\penalty0 (3):\penalty0 822--848, 1998.

\bibitem[Harrison and L{\'{o}}pez(1999)]{HL99}
J.~Michael Harrison and Marcel~J. L{\'{o}}pez.
\newblock {Heavy traffic resource pooling in parallel-server systems}.
\newblock \emph{Queueing Syst.}, 33\penalty0 (4):\penalty0 339--368, 1999.
\newblock ISSN 15729443.
\newblock \doi{10.1023/A:1019188531950}.

\bibitem[Foss and Chernova(1998)]{FC98}
Serguei~G. Foss and Natalia~I. Chernova.
\newblock {On the stability of a partially accessible multi‐station queue
  with state‐dependent routing}.
\newblock \emph{Queueing Syst.}, 29\penalty0 (1):\penalty0 55--73, 1998.
\newblock ISSN 02570130.
\newblock \doi{10.1023/A:1019175812444}.
\newblock URL \url{http://link.springer.com/10.1023/A:1019175812444}.

\bibitem[Dai and Tezcan(2011)]{DT11}
J.~G. Dai and Tolga Tezcan.
\newblock {State space collapse in many-server diffusion limits of parallel
  server systems}.
\newblock \emph{Math. Oper. Res.}, 36\penalty0 (2):\penalty0 271--320, 2011.
\newblock ISSN 0364765X.
\newblock \doi{10.1287/moor.1110.0494}.

\bibitem[Mukherjee et~al.(2018)Mukherjee, Borst, van Leeuwaarden, and
  Whiting]{MBLW16-3}
Debankur Mukherjee, Sem~C Borst, Johan S~H van Leeuwaarden, and Philip~A
  Whiting.
\newblock {Universality of power-of-d load balancing in many-server systems}.
\newblock \emph{Stoch. Syst.}, 8\penalty0 (4):\penalty0 265--292, 2018.
\newblock \doi{10.1287/stsy.2018.0016}.
\newblock URL \url{https://doi.org/10.1287/stsy.2018.0016}.

\bibitem[Stolyar(2015{\natexlab{b}})]{Stolyar15}
Alexander~L. Stolyar.
\newblock {Pull-based load distribution in large-scale heterogeneous service
  systems}.
\newblock \emph{Queueing Syst.}, 80\penalty0 (4):\penalty0 341--361,
  2015{\natexlab{b}}.
\newblock ISSN 0257-0130.
\newblock \doi{10.1007/s11134-015-9448-8}.
\newblock URL \url{http://link.springer.com/10.1007/s11134-015-9448-8}.

\bibitem[Tsitsiklis and Xu(2013)]{TX13}
John~N. Tsitsiklis and Kuang Xu.
\newblock {Queueing system topologies with limited flexibility}.
\newblock In \emph{Proc. SIGMETRICS '13}, pages 167--178, 2013.
\newblock \doi{10.1145/2465529.2465757}.

\bibitem[Tsitsiklis and Xu(2017)]{TX17}
John~N. Tsitsiklis and Kuang Xu.
\newblock {Flexible queueing architectures}.
\newblock \emph{Oper. Res.}, 65\penalty0 (5):\penalty0 1398--1413, 2017.
\newblock \doi{10.1287/opre.2017.1620}.

\bibitem[Ethier and Kurtz(2009)]{EK09}
Stewart~N. Ethier and Thomas~G. Kurtz.
\newblock \emph{{Markov Processes: Characterization and Convergence}}.
\newblock 2009.
\newblock ISBN 047176986X.
\newblock \doi{10.1007/978-3-0348-8645-1{\_}6}.

\end{thebibliography}

\end{document}